\documentclass[reqno,10pt]{amsart}
\numberwithin{equation}{section}

\usepackage{pdfsync}
\usepackage{ifthen}
\usepackage{amsmath}
\usepackage{amsthm}
\usepackage{graphicx,eucal}
\usepackage{graphpap,latexsym,epsf,bm}
\usepackage{color}
\usepackage{hyperref}
\usepackage{amssymb,mathrsfs,enumerate}

\newtheorem{theorem}{Theorem}[section]
\newtheorem{corollary}[theorem]{Corollary}
\newtheorem{lemma}[theorem]{Lemma}
\newtheorem{proposition}[theorem]{Proposition}
\newtheorem{definition}[theorem]{Definition}%
\newtheorem{remark}[theorem]{Remark}%
\newtheorem{example}[theorem]{Example}

\newcommand{\N}{\mathbb{N}} 
\newcommand{\R}{\mathbb{R}} 

\newcommand{\calT}{\mathcal{T}}
\newcommand{\calO}{\mathcal{O}}
\newcommand{\teta}{\vartheta}
\newcommand{\dd}{\mathrm{d}}
\newcommand{\rmd}{\mathrm{d}}
\newcommand{\eps}{\varepsilon}

\newcommand{\weaksto}{{\rightharpoonup^*}\,}
\newcommand{\weakto}{\rightharpoonup}
\newcommand{\Restr}[1]{\lower2pt\hbox{$|_{#1}$}}
\newcommand{\argmin}{\mathop{\mathrm{Argmin}}}

\newcommand{\EE}{\mathcal{E}}

\newcommand{\ene}[2]{\mathcal{E}_{#1}(#2)}
\newcommand{\power}[2]{\mathcal P_{#1}(#2)}

\newcommand{\pVarname}[1]{\text{\slshape Var}_{#1}}
\newcommand{\Var}[4]{\mathop{\mathrm{Var}}\nolimits_{
    {#1} }(#2;[#3,#4])}
\newcommand{\pVar}[4]{\mathop{\text{\slshape Var}}\nolimits_{#1}(#2;[#3,#4])}
\newcommand{\JVar}[4]{{\mathrm{Jmp}}_{#1}(#2;[#3,#4])}

\newcommand{\BV}{\ensuremath{\mathrm{BV}}}
\newcommand{\AC}{\mathrm{AC}}

\newcommand{\vvm}[3]{{\mathfrak{f}}_{#1}({#2};{#3})}

\newcommand{\vvmV}[3]{
  \fre_{#1}(#2)
  \ifthenelse{\equal{#3}{}} {} {\|#3\|} }
\newcommand{\vvmname}{\mathfrak{f}}
\newcommand{\vvmnametil}{\mathfrak f}

\newcommand{\vvmfulleps}[5]{\mathfrak{F}_\eps({#2},{#3};{#1},{#4})}

\newcommand{\vvmfull}[5]{{\mathfrak{F}}({#2},{#3};{#1},{#4})}
\newcommand{\vvmfullV}[5]{\mathfrak{G}({#2},{#3};{#1},{#4})}
\newcommand{\vvmfullepsname}{\mathfrak{F}_\eps}
\newcommand{\vvmfullname}{\mathfrak{F}}
\newcommand{\vvmfullVname}{\mathfrak{G}}

\newcommand{\FNormname}[1]{\mathscr{P}}

\newcommand{\Leb}[1]{\mathscr{L}^{#1}}
\newcommand{\topref}[2]{\stackrel{\eqref{#1}}#2}

\newcommand{\RIS}{\ensuremath{( \V,\mathcal{E}, \Psiz, \Psiv)}}
\newcommand{\ris}{RIS}

\newcommand{\bptname}{\mathfrak{p}}
\newcommand{\bpt}[2]{\mathfrak p(#1,#2)}

\newcommand{\Cost}[4]{\ifthenelse{
    \equal{#1}\Phi}
  {\Phi(#4-#3)}
  {\Delta_{#1_{#2}}(#3,#4)}
}

\newcommand{\foraa}{\text{for a.a.\,}}
\newcommand{\aein}{\text{a.e.\ in }}
\newcommand{\forae}{\text{for a.a.\,}}
\newcommand{\down}{\downarrow}
\newcommand{\up}{\uparrow}
\newcommand{\Banach}{B}
\newcommand{\Vanach}{V}
\newcommand{\V}{\Vanach}
\newcommand{\B}{\Banach}
\newcommand{\rmC}{\mathrm{C}}

\newcommand{\la}{\langle}
\newcommand{\ra}{\rangle}

\newcommand{\taue}{{\tau,\eps}}

\newcommand{\Utaue}[1]{\mathrm{U}_{\tau,\eps}^{#1}}
\newcommand{\Vtaue}[1]{\mathrm{V}_{\kern-1pt\tau,\eps}^{#1}}
\newcommand{\Xitaue}[1]{\Xi_{\kern-1pt\tau,\eps}^{#1}}
\newcommand{\UU}{\mathrm{U}}
\newcommand{\piecewiseConstant}[2]{\overline{\mathrm{#1}}_{\kern-1pt#2}}
\newcommand{\pwC}{\piecewiseConstant}
\newcommand{\underpiecewiseConstant}[2]{\underline{\mathrm{#1}}_{\kern-1pt#2}}

\newcommand{\piecewiseLinear}[2]{\mathrm{#1}_{\kern-1pt#2}}
\newcommand{\pwL}{\piecewiseLinear}
\newcommand{\pwM}[2]{\widetilde{\mathrm{#1}}_{\kern-1pt#2}}
 \def\trait #1 #2 #3 {\vrule width #1pt height #2pt depth #3pt}

  \newcommand{\parat}{\mathsf{t}}
\newcommand{\parau}{\mathsf{u}}
\newcommand{\sfr}{\mathsf{r}}
\newcommand{\sfM}{\mathsf{M}}
\newcommand{\sfs}{\mathsf{s}}
\newcommand{\sfu}{\mathsf{u}}
\newcommand{\sfv}{\mathsf{v}}
\newcommand{\sft}{\mathsf{t}}
\newcommand{\sfa}{\mathsf{a}}
\newcommand{\sfe}{\mathsf{e}}
\newcommand{\sfb}{\mathsf{b}}
\newcommand{\sfS}{\mathsf{S}}
\newcommand{\frf}{\mathfrak f}
\newcommand{\fre}{\mathfrak e}

\newcommand{\frF}{\mathfrak F}

\newcommand{\frG}{\mathfrak G}

\newcommand{\frk}{\mathfrak k}

\newcommand{\Diss}[2]{
  \ifthenelse      {   \equal{#1}{V}  }         {\Psiv}
  {
    \ifthenelse      {   \equal{#1}{B}  }         {\Psiz}
    {
      \ifthenelse      {   \equal{#1}{\Bo}  }     {\Psiz}
      {\Psi_{\kern-1pt#1}}
      }
    }
  ^{#2}
}
\newcommand{\Bnorm}[1]{\Psiz{(#1)}}
\newcommand{\Dnorm}[1]{\Psiz_{\kern-2pt \scriptscriptstyle\land}(#1)}
\newcommand{\Dnormname}{\Psiz_{\kern-2pt \scriptscriptstyle\land}}

\newcommand{\Vnorm}[1]{\|{#1}\|}

\newcommand{\pairing}[4]{ \sideset{_{#1 }}{_{ #2}}  {\mathop{\langle #3 , #4  \rangle}}}

\newcommand{\cE}{\mathcal{E}}

\newcommand{\cg}[1]{\mathcal{G}(#1)}
\newcommand{\domainenergy}{D}
\newcommand{\frsub}{\partial}

\newcommand{\adm}[5]{\mathscr{A}(#1,#2;[#3,#4]\times{#5})}

\def\nchi{{\raise.3ex\hbox{$\chi$}}}
\newcommand{\Bo}{{}}

\newcommand{\rmJ}{\mathrm J}
\newcommand{\rmV}{\mathrm V}
\newcommand{\scalarV}[3]{\rmV_{\!#2}
  \ifthenelse{\equal{#3}{}}{}{(#3)}}
\newcommand{\scalarmu}[2]{\mu}
\newcommand{\scalarmuco}[2]{\mu_{\rm d}}
\newcommand{\scalarmuj}[2]{\mu_{\rm J}}
\newcommand{\scalarmul}[2]{\mu_{\mathscr L}}
\newcommand{\scalarmuc}[2]{\mu_{\rm C}}
\newcommand{\scalardens}[2]{#1[#2']}
\newcommand{\Psiv}{\Phi}
\newcommand{\Psiz}{\Psi}
\newcommand{\forevery}{\text{for every }}

\newcommand{\nn}{{\mbox{\boldmath$n$}}}
\newcommand{\scrD}{\mathscr D}

\definecolor{ddmagenta}{rgb}{0.7,0,0.9}
\definecolor{ddcyan}{rgb}{0,0.2,1.0}
\newenvironment{alc}{\color{red}}{\color{black}}
\newcommand{\bealc}{\begin{alc}}
\newcommand{\ealc}{\end{alc}}

\newenvironment{ric}{\color{ddmagenta}}{\color{black}}
\newcommand{\beric}{\begin{ric}}
\newcommand{\eric}{\end{ric}}

\newenvironment{coq}{\color{ddcyan}}{\color{black}}
\newcommand{\becoq}{\begin{coq}}
\newcommand{\ecoq}{\end{coq}}

\begin{document}

\title[$\BV$ solutions to rate-independent systems]
{Balanced Viscosity ($\BV$) solutions to\\ infinite-dimensional
rate-independent systems}

\date{September 21, 2013}

\author{Alexander Mielke}
\address{Weierstra\ss-Institut,
  Mohrenstra\ss{}e 39, 10117 D--Berlin and Institut f\"ur
  Mathematik, Humboldt-Universit\"at zu
  Berlin, Rudower Chaussee 25, D--12489 Berlin (Adlershof), Germany.}
\email{mielke\,@\,wias-berlin.de}

\author{Riccarda Rossi}
\address{DICATAM -- Sezione di Matematica, Universit\`a di
  Brescia, via Valotti 9, I--25133 Brescia, Italy.}
\email{riccarda.rossi\,@\,ing.unibs.it}

\author{Giuseppe Savar\'e}
\address{Dipartimento di Matematica ``F.\ Casorati'', Universit\`a di
  Pavia.  Via Ferrata, 1 -- 27100 Pavia, Italy.}  
\email{giuseppe.savare\,@\,unipv.it} 

\thanks{A.M. has been partially
  supported by the  ERC grant no.267802 AnaMultiScale. 
  R.R. and G.S. have been partially supported by a MIUR-PRIN'10-11 grant
  for the project ``Calculus of Variations''.}

\begin{abstract}
  Balanced Viscosity solutions to rate-independent systems arise as
  limits of regularized rate-independent flows by adding a superlinear
  vanishing-viscosity dissipation.

  We address the main issue of proving the existence of such limits
  for infinite-dimensional systems and of characterizing them by a
  couple of variational properties that combine a local stability
  condition and a balanced energy-dissipation identity.

  A careful description of the jump behavior of the solutions, of
  their differentiability properties, and of their equivalent
  representation by time rescaling is also presented.

  Our techniques rely on a suitable chain-rule inequality for
  functions of bounded variation in Banach spaces, on refined lower
  semicontinuity-compactness arguments, and on new \BV-estimates that
  are of independent interest.

\end{abstract}
\maketitle

\tableofcontents

\section{Introduction}
\label{s:Intro}

This paper concerns the asymptotic behavior of the solutions
$u_\eps:[0,T]\to V$, $\eps\down0$, of singularly perturbed doubly
nonlinear evolution equations of the type 
\begin{equation}
\label{viscous-dne}
\partial \Diss{\eps}{} (\dot{u}_\eps(t))+\frsub \ene t{u_\eps(t)} \ni
0 \quad \text{in $\V^*$,} \quad t \in (0,T). 
\end{equation}
Here $(\V,\Vnorm{\cdot})$ is a Banach space satisfying the
Radon-Nikod\'ym property (e.g.\ a reflexive space, see
\cite{Diestel-Uhl77}), $\partial\cE$ is the Fr\'echet subdifferential
of a time-dependent energy functional $\cE: [0,T]\times \V \to
(-\infty,+\infty],$ and $ \Diss{\eps}{} : \V \to [0,+\infty)$ is a
family of convex and superlinear dissipation potentials; the main
coercivity and structural assumptions on $\Psiz_\eps,\cE$ will be
discussed in Section \ref{ss:2.1}.

The main feature we want to address here is the degeneration of
the superlinear character of $\Psi_\eps$ as $\eps\down0$,
approximating a degree-$1$ positively homogeneous convex potential
$\Psiz:V\to [0,+\infty),$
\begin{equation}
  \label{eq:76}
  \Psiz(\lambda v)=\lambda \Psiz(v)\quad
  \forevery v\in V,\ \lambda\ge0;\qquad
  \Psiz(v)>0\quad\text{if }v\neq 0.
\end{equation}
An important example motivating our investigation is the
vanishing quadratic approximation
\begin{gather}
  \label{def:psi-eps-intro}
  \Diss{\eps}{}(v)= \Diss{\Bo}{}(v) +\frac \eps2
  \|v\|^2
  ,\quad\text{associated with the viscous potential}
  \quad
  \Phi(v):=\frac 12\|v\|^2.
\end{gather}

\subsubsection*{\bfseries The superlinear case}
Equations of the type \eqref{viscous-dne} arise in several
contexts, ranging from thermomechanics to the modeling of rate-independent
evolution. In the realm of these applications, \eqref{viscous-dne}
may be interpreted as  \emph{generalized balance relation},
balancing viscous and potential forces. 

The analysis of \eqref{viscous-dne}
when the energy $\cE$ has the typical form
\begin{displaymath}
\ene tu = \cE(u) - \langle \ell(t),u\rangle,\quad \text{with $\ell
  :[0,T] \to \V^*$ smooth and $\cE: \V \to (-\infty,+\infty]$ \emph{convex}}
\end{displaymath}
goes back to the seminal papers \cite{ColliVisintin90,Colli92}.
Therein, the existence of absolutely continuous solutions to the Cauchy problem for
\eqref{viscous-dne} was proved by means of maximal monotone operator techniques.
Existence and approximation results for a broad class of
\emph{nonconvex}
energies, also featuring a singular dependence on time,
have been recently obtained in \cite{MRS-dne},
relying on various contributions from 
the theory of curves of Maximal Slope
\cite{DeGiorgi-Marino-Tosques80,Marino-Saccon-Tosques89}
and from the variational approach to gradient flows
\cite{DeGiorgi93,Rossi-Savare06,AGS08,RMS08}.

\subsubsection*{\bfseries Positively $1$-homogeneous dissipations:
  energetic solutions.}
Since $ \Diss{}{} $  is positively homogeneous of degree $1$,
when $\eps=0$ the formal limit of
\eqref{viscous-dne}
\begin{equation}
\label{eq:74}
\partial \Diss{}{} (\dot{u}(t))+\frsub \ene t{u(t)} \ni 0 \quad \text{in $\V^*$,} \quad t \in (0,T),
\end{equation}
describes a rate-independent evolution. In this case,
even for convex energies $\cE_t(\cdot)$
one cannot expect the existence of absolutely continuous solutions to
\eqref{viscous-dne}: in general,
they may  be only $\BV$ with respect to 
time and in fact have jumps, so that
even the precise meaning of the differential inclusion
\eqref{eq:74} is a delicate question.

This has called for \emph{weak-variational} characterization of the
solutions of \eqref{eq:74}, leading to the concept of \emph{energetic
  solution} to the rate-independent system $(V,\cE,\Psiz)$: it dates
back to \cite{Mie-Theil99} and was further developed in
\cite{Mielke-Theil04,DaFrTo05QCGN}, see also
\cite{Miesurvey,Miel08?DEMF} and the references therein.

In this setting, $u:[0,T]\to V$ is an energetic solution to equation
\eqref{eq:74}, if it complies with the \emph{global stability}
\eqref{eq:64bis-intro} and with the \emph{energy balance}
\eqref{eq:83-intro} conditions
 \begin{equation}
    \label{eq:64bis-intro}
    \tag{$\mathrm{S}$}
\forall\, v \in V\, : \quad \ene t{u(t)}\leq \ene tv + \Bnorm{v{-}u(t)} \quad \text{ for all } t\in [0,T],
  \end{equation}
  \begin{equation}
    \label{eq:83-intro}
    \Var{\Diss {\B}{}} u{0}{t}+\ene{t}{u(t)}=\ene{0}{u(0)}+
    \int_{0}^{t} \partial_t\ene s{u(s)}\,\mathrm{d}s \quad \text{ for all } t\in [0,T],
    \tag{$\mathrm{E}$}
  \end{equation}
  where $\Var\Psiz u ab$ is the total variation induced by
  $\Bnorm{\cdot}$
on the interval $[a,b] \subset
[0,T]$, viz.
\begin{equation}
  \label{eq:78}
  \Var{\Psi}uab:=
  \sup\Big\{\sum_{m=1}^M \Psi\big(u(t_m)-u(t_{m-1})\big):a=t_0< t_1<\cdots<t_{M-1}<t_M=b\Big\}.
\end{equation}
  The energetic formulation
  \eqref{eq:64bis-intro}--\eqref{eq:83-intro}
  has several strong points: it is derivative-free,  it
  bypasses all the technical differentiability
  issues on $V$ (related to the validity of the Radon-Nikod\'ym property), on
  $u$ (related to its behavior on the Cantor-Jump set), and
  on the energy $\cE$ (related to its Frech\'et subdifferential). Furthermore, 
  it provides nice existence-stability results under simple
coercivity and time-regularity assumptions.

Nonetheless, in the case of \emph{nonconvex} energies it is now well
known \cite{MRS09,MRS10,Miel08?DEMF,Mielke-Zelik,Rossi-Savare13} that
the global stability condition \eqref{eq:64bis-intro} involves a
variational characterization of the jump behavior of the system, that
is affected by the whole energetic landscape of $\cE$.

\subsubsection*{\bfseries Positively $1$-homogeneous dissipations:
     the  vanishing-viscosity  approach}
The by now well-established \emph{vanishing-viscosity} approach aims
to find good local conditions describing rate-independent evolution
(and in particular the behavior of the solutions at jumps).  It also
leads to a clarification of the  connections with the
metric-variational theory of gradient flows.

  While referring to \cite{MRS10} for a more detailed survey, here we
  recall
  the works where the  vanishing-viscosity 
  analysis is carried out
  via the \emph{reparameterization technique}
  introduced in  \cite{ef-mie06}. 
  They range from
  applicative contexts in material modeling
  (such as crack propagation \cite{KnMiZa07?ILMC,KnMiZa08?CPPM}, Cam-Clay and non-associative plasticity \cite{DalMaso-DeSimone-Solombrino2011,DalMaso-DeSimone-Solombrino2012,BabFraMor12}, and
  damage \cite{KnRoZa2011}), to
 the analysis of
 parabolic PDEs with rate-independent dissipation terms
 \cite{Mielke-Zelik}.

  Abstract
 rate-independent systems in a finite-dimensional setting
 have been studied by \cite{MRS09,MRS10}.
 In particular,
 the vanishing-viscosity limit of  gradient systems of the
 type \eqref{viscous-dne} has been studied in
 \cite{MRS10} when $V$ is a \emph{finite}-dimensional space
 and
 $\cE\in \mathrm{C}^1([0,T]\times
 V)$.
 Here we aim to generalize
the results from  \cite{MRS10}
 to the present
 \emph{nonsmooth, infinite-dimensional} setting.

A simple prototype of the situation we have in mind
(see also Section \ref{s:examples}) is
\begin{equation}
\V=L^2(\Omega),\quad
\Psi_\eps(v)=\int_\Omega |v| + \frac{\eps}{2} |v|^2 \,\dd x, \quad
\ene tu= \int_\Omega \left(\frac12 |\nabla u|^2 + W(u) -\ell(t) u
\right)\,\dd x\label{eq:83} 
\end{equation}
where $\Omega$ is a bounded open subset of $\R^d$,
$\ell\in \mathrm C^1([0,T];L^2(\Omega))$ and $W\in \mathrm C^1(\R)$
is, e.g., a double-well type nonlinearity. The abstract subdifferential inclusion 
\eqref{viscous-dne} leads to the nonlinear parabolic equation
\begin{equation}
\label{e:mie-zel-ex}
\eps \, {\partial_t} {u} + \mathrm{Sign}\big(
{\partial_t}{u}\big)-\Delta u +W'(u)= \ell \quad  \text{in }  \Omega \times (0,T),
\end{equation}
for which the vanishing-viscosity limit $\eps \down 0$ was in fact
analyzed in \cite{Mielke-Zelik},  based on the
\emph{reparameterization} technique and on the concept of
\emph{parameterized solution}, from \cite{ef-mie06}.

In this work we will propose a direct characterization of the limit
evolution, in the same spirit of conditions
\eqref{eq:64bis-intro}--\eqref{eq:83-intro}, and we will show how it
is related to a parameterized formulation.  A particular emphasis will
be on the crucial property encoded in the \emph{balanced
  energy--dissipation identities}, both in the original and in the
rescaled time variables.  The notion of \emph{Balanced Viscosity
  $(\BV)$ solution} to a rate-independent system tries to capture this
essential feature.

\subsubsection*{\bfseries Balanced Viscosity $(\BV)$ solutions}
Let us briefly describe what we mean by a
\emph{balanced viscosity $(\BV)$ solution to the rate-independent
system}   (\ris) $\RIS$, where now also the viscosity correction induced by
$\Phi$ characterizes the evolution.
To simplify the exposition in this introduction,    we suppose that
$\Psi$ is $V$-coercive, i.e.~$\Psi(v)\ge c\|v\|$ for all
$v\in V$ and
for a constant $c>0$.

A crucial role is played by the dual convex set
\begin{equation}
K^*:=\{\xi\in V^*: \langle
\xi,v\rangle\le \Psiz(v)\ \forevery v\in V\}\label{eq:79}
\end{equation}
whose support function is $\Psiz$.
Following
\cite{MRS10}, we say that a curve $u\in \BV([0,T];V)$ is a $\BV$ solution to
the \ris\
$\RIS$, if it
fulfills
the following \emph{local stability condition}
 \begin{equation}
    \label{eq:65bis-intro}
    \tag{$\mathrm{S}_\mathrm{loc}$}
     K^* + \frsub \ene t{u(t)} \ni 0 
  \quad \text{for all}\, t \in [0,T] \setminus \mathrm{J}_u,
  \end{equation}
where $\mathrm J_u$ is the jump set of $u$,
and the \emph{Energy-Dissipation Balance}
  \begin{equation}
    \label{eq:84-intro}
    \pVar{\vvmnametil}u{0}{t}+\ene{t}{u(t)}=\ene{0}{u(0)}+
    \int_{0}^{t} \partial_t\ene s{u(s)}\,\mathrm{d}s \quad \text{ for all } t\in [0,T].
    \tag{E$_{\vvmnametil}$}
  \end{equation}
Like \eqref{eq:83-intro}, \eqref{eq:84-intro} as well balances
at every evolution time $t \in [0,T]$
 the energy dissipated by the system  and the current energy, with the initial energy and the work of the external
 forces. However, in \eqref{eq:84-intro}
 dissipation is measured by  the total variation functional $\pVarname{\vvmnametil}$. While referring to
 the forthcoming
  Definition
  \ref{def:Finsler_jump} for a precise formula, we may mention here that
  the main difference of $\pVarname{\vvmnametil}$
  with respect to $\mathrm{Var}_\Psiz$ concerns the
  contribution of the jumps. In fact, in the definition of
  $\pVarname{\vvmnametil}$ the cost $\Psiz(u(t_+)-u(t_-))$ of the transition
  from the left limit $u(t_-)$ to the right limit $u(t_+)$ at a time
  $t\in \mathrm J_u$ is replaced by
  the \emph{Finsler dissipation cost}
  \begin{equation}
  \label{finsler-distance-intro}
  \begin{aligned}
   \Cost{\vvmnametil}t{u_0}{u_1}:=\inf\Big\{\int_{0}^{1}  &
   \vvm{t}{\teta}{\dot{\teta}}
      \,\mathrm{d}r:
      \vartheta\in \AC([0,1];V),\ \vartheta(0)=u(t_-),\
      \vartheta(1)=u(t_+)\Big\},
      \end{aligned}
\end{equation}
where
\begin{equation}
 \label{bipotential-functional}
  \vvm t{\vartheta}{\dot\vartheta} =
  \Diss B{} (\dot\vartheta) +   \fre_t(\vartheta)\Vnorm{\dot\vartheta},\quad
  \fre_t(\vartheta):=\inf\Big\{\|\xi-z\|_*:\xi\in -\frsub \ene
  t\vartheta,\ z\in  K^*\Big\}. 
\end{equation}
Formula \eqref{bipotential-functional} clearly shows that
the Finsler dissipation cost \eqref{finsler-distance-intro} (and thus the
 total variation $\pVarname{\vvmnametil}$) encompasses both \emph{rate-independent} effects through
 $\Bnorm{\cdot}$, and \emph{viscous} effects
through
 $\Vnorm{\cdot}$. The latter are active whenever
 $\fre_t(\vartheta)>0$, precisely when the local stability
 condition
  \eqref{eq:65bis-intro} is violated, since
  $  K^* +\frsub \ene tu \ni 0$ if and only if
  $\fre_t(u)=0$.
  Ultimately, by virtue of  \eqref{eq:84-intro},
 \emph{viscous} dissipation enters in the description of
  the energetic behavior of the system at jumps.

  The link between the particular structure of
  \eqref{bipotential-functional} and
  the vanishing-viscosity approximation \eqref{viscous-dne}
  can be better understood by recalling the strucure of the energy-dissipation
  balance satisfied by the solutions to the viscous evolution:
  \begin{equation}
    \label{eq:80}
    \ene t{u_\eps(t)}+\int_0^t \Big(\Psi_\eps(\dot
    u_\eps)+\Psi_\eps^*(\xi_\eps)\Big)\,\dd r=
    \ene 0{u_\eps(0)}+\int_0^t \partial_t \ene r{u_\eps(r)},\quad
    \xi_\eps(r)\in -\partial\ene r{u_\eps(r)}.
  \end{equation}
  It turns out that $\frf_t$ admits the variational representation
  \begin{equation}
    \label{eq:82}
    \frf_t(\vartheta,\dot\vartheta)=
    \inf\Big\{\Psi_\eps(\dot
    \vartheta)+\Psi_\eps^*(\xi):
    \xi\in -\partial \ene t{\vartheta},\ \eps>0\Big\}.
  \end{equation}
 This feature is in some sense reflected
 by
 the so-called
 \emph{optimal jump transitions} connecting
 $u(t_-)$ and $u(t_+)$:
 they are curves $\teta \in \AC([0,1];V)$ which
 attain the infimum in  formula  \eqref{finsler-distance-intro}
 and keep track of the asymptotic profile of the converging
 solutions $u_\eps$ around a jump point.
 By means of a careful rescaling technique,
 we will show that
 optimal  transitions
 fulfill
 the doubly nonlinear equation
 \begin{equation}
   \label{ojt-intro}
   \partial\Diss{\Bo}{}(\dot{\teta}(r))+\partial\Diss{\V}{}(\eps(r) \dot{\teta}(r))+ \frsub \ene t {\teta(r)} \ni 0 \quad
   \foraa\, r \in (0,1)
 \end{equation}
 for some map $r \mapsto \eps(r) \in [0,+\infty)$.

\subsubsection*{\bfseries Lack of differentiability and  non-coercive
rate-independent dissipations}
Up to now,
for the
sake of simplicity, we have overlooked one crucial issue in the analysis
of the rate-independent equation \eqref{eq:74}, namely
the lack of differentiability of the limiting solution $u$
when $\Psiz$ is not coercive with respect to the norm $\| \cdot \|$ on $V$ (as in the example \eqref{eq:83}).
Even the introduction of a weaker norm cannot  avoid
this technical issue, since in many interesting examples
norms of $L^1$-type do not comply with the Radon-Nikod\'ym property.

This fact leads to significant technical difficulties,
in that $\Psiz$-absolutely continuous curves need not be
pointwise differentiable with respect to time. Hence, for example
formulae
\eqref{finsler-distance-intro}--\eqref{bipotential-functional}
need to be carefully modified by introducing the convenient notion
of the metric $\Psiz$-derivative,
and differential inclusions like \eqref{ojt-intro}
have to be suitably interpreted.

On the other hand, we will show that under slightly stronger
assumptions
on the energy functional $\cE$,
limiting solutions still belong to $\BV([0,T];V)$ even in the case of
a degenerate rate-independent dissipation $\Psiz$.
For this class of $V$-parameterizable solutions
we can recover a more precise differential
characterization, and several  expressions take a simpler form.
\subsubsection*{\bfseries Main results and plan of the paper}
In this paper we provide existence and approximation results for 
Balanced Viscosity solutions
to the \ris\ $\RIS$
under quite general conditions on the dissipation potentials
$\Diss{\Bo}{},\,\Diss{\V}{}$ and on the energy functional $\cE$,
enlisted in
Section
\ref{ss:2.1}.
Let us mention in advance that, our
standing assumptions  on
$\cE$
guarantee  the lower semicontinuity, coercivity, uniform
subdifferentiability
of the functional $u \mapsto \ene tu$, and (sufficient) smoothness of the time-dependent
function $t \mapsto \ene tu$. In  \S \ref{s:appA} we provide some preliminary
 results on absolutely continuous
and BV curves, while the main
existence and structural properties of viscous gradient systems are recalled 
in \S\,\ref{ss:G1}--\S\,\ref{ss:G2}.

In Section \ref{ss:2.-bv} we present our main results concerning
Balanced Viscosity solutions.  The Finsler cost
\eqref{finsler-distance-intro} and its related total variation are
discussed in \S\,\ref{ss:3.1}.  In \textbf{Theorem \ref{th:1}} we
state the relative compactness of viscous solutions $(u_\eps)_\eps$ to
\eqref{viscous-dne} with respect to pointwise convergence, and we show
that any limit point as $\eps\down0$ is a $\BV$ solution.
A similar result (\textbf{Theorem \ref{th:2}}) addresses the passage
to the limit in the time-incremental minimization scheme
\cite{DeGiorgi93} for the viscous problem: given a time step $\tau>0$,
the uniform partition $t_n:=n\tau$, $n=0,\cdots, N_\tau$, of the time
interval $[0,T]$ so that $\tau(N_\tau-1)<T\le \tau N_\tau$, and an
initial datum $\Utaue0$, the scheme produces discrete sequences
$(\mathrm U^n_{\taue})$, $n\in \N,$ by solving the minimization
problem
 \begin{equation}
  \label{eq:58}\tag{IP$_{\eps,\tau}$}
  \mathrm{U}^n_\taue\in \argmin_{\mathrm{U} \in \V}  
 \left\{\tau\Diss{\eps}{}\Big(\frac{\mathrm{U}-\mathrm{U}^{n-1}_\taue}{\tau}\Big)+ 
    \ene{t_n}{\mathrm{U}}\right\}\quad \text{for }n=1,\cdots, N_\tau.
\end{equation}
As $\tau,\, \eps \down 0$ with $\tau/\eps \down0$ we will prove that
the piecewise affine interpolants (see \eqref{e:pwl}) $(\pwL
{U}{\taue})_\taue$ of the discrete values $\Utaue n$ converge (up to
subsequences) to a $\BV$ solution of the \ris\ $\RIS$.  Under slightly
stronger assumptions on the energy functional $\cE$, \textbf{Theorems
  \ref{th:3-discrete} and Corollary
  \ref{cor:3-limit}} 
show that the limits obtained by this variational scheme belong to
$\BV([0,T];V)$ and are $\V$-parameterizable, a distinguished class of
solutions studied in \S\,\ref{ss:Vparam}.  Other important properties
of \BV\ solutions are discussed in \S\,\ref{ss:BV} and
\ref{ss:opt-trans}: the latter is focused in particular on the notion
of \textbf{optimal jump transition}, a useful tool to describe the
asymptotic profile of the solution $u_\eps$ around a jump limit point.

We discuss parameterized solutions in Section \ref{s:5}:
\textbf{Theorem \ref{thm-van-param}} provides the main existence and
convergence result, the tight connections with \BV\ solutions are
clarified in \textbf{Theorem \ref{prop:bv}}, and the case of
$V$-parameterized solutions is investigated in Section \ref{ss:3.2}.

Section \ref{s:examples} is devoted to a series of examples, where we
discuss the validity of the abstract conditions on the energy
enucleated in \S\,\ref{ss:2.1}, and in particular of the chain-rule
inequality. Furthermore, Example \ref{ex:3.a} shows that there exist
\BV\ solutions which are not $\V$-parameterizable.  Most of the proofs
and of the technical tools are collected in the last three sections.
\textbf{Section \ref{s:chain}} is devoted to the main theme of the
chain-rule inequalities in the parameterized (\S\,\ref{ss:8-chain})
and BV setting (\S\,\ref{ss:chain-bv}).

\textbf{Section \ref{s:last}} contains the main stability,
compactness, and lower semicontinuity results that lie at the core of
our proofs.  In \S\,\ref{ss:last1} and \S\,\ref{ss:clsc} we alternate
the parameterized and the non-parameterized point of view to describe
the limit of various integral functionals.  The crucial lower
semicontinuity result in the \BV\ setting is \textbf{Proposition
  \ref{cor:2}}, where  we adapt ideas introduced in \cite{MRS12}. 
The proofs of the main Theorems are eventually collected in \S
\ref{ss:8-vanvisc}.  The crucial \BV\ estimate for the discrete
Minimizing Movements leading to \textbf{$V$-parameterizable solutions}
are collected in \S\,\ref{ss:5.2}.

\section{Notation, assumptions and preliminary results}
\label{s:2}

\subsection{The energy-dissipation framework}
\label{ss:2.1}

Throughout the present paper we will suppose that
\begin{equation} \label{e:top}
\begin{gathered}
  \text{$(\Vanach,\Vnorm\cdot)$ is
   a separable Banach
   space satisfying the Radon-Nikod\'ym property.}
 \end{gathered}
\end{equation}
This means that absolutely continuous curves with values in $V$ are
$\Leb 1$-a.e.~differentiable, see Section \ref{s:appA}.  This
condition is certainly satisfied if $V$ is reflexive or if it is the
dual of a separable Banach space, see \cite{Diestel-Uhl77}.
With 
$\Vnorm{\cdot}_*$ we will denote the dual norm in $\V^*$, while
$\pairing{}{}{\cdot}{\cdot}$ stands for the duality pairing between
$\V^*$ and $\V$.

\subsubsection*{\bfseries Rate-independent and viscous dissipation}
On $V$ are defined two
\begin{equation}
  \text{continuous convex dissipation potentials
    $\Psiz,\Psiv:V\to [0,+\infty),$
    strictly positive in $V\setminus\{0\}$.}\label{e:2.1}  \tag{D$.0$}
\end{equation}
The \emph{``rate-independent''} potential $\Psiz$ is positively
$1$-homogeneous (a ``gauge'' functional, \cite{Rockafellar70})
\begin{equation}
  \tag{D$.1$}
  \label{eq:33}
  \Psiz(\lambda v)=\lambda \Psiz(v)\quad
  \text{for all }\,\lambda\ge0 \text{ and }
  \ v\in V.
\end{equation}
Notice that if $\Psiz(-v)=\Psiz(v)$ for every $v\in \V$, then $\Psiz$
is a norm in $\V$; we will say that $\Psi$ is coercive if $\Psi(v)\ge
c\|v\|$ for every $v\in V$ and some constant $c>0$.  However, in
general we will not assume any coercivity on $\Psi$, so that the
sublevel sets $\{v\in \V:\Psiz(v)\le r\}$ are not bounded.

Coercivity will be recovered by the addition of a \emph{``viscous''}
dissipation potential $\Psiv$ of the form
\begin{equation}
  \label{def-psiV}\tag{D${.2}$}
  \begin{gathered}
  \Psiv(v)=F(\|v\|)\quad
   \text{for }F\in \mathrm C^1([0,+\infty)) \text{ convex, with }
   \\
   F(r) >0   \text{ for } r>0, \
   F(0)=F'(0)=0,\ \lim_{r\up+\infty}F'(r)=+\infty.
  \end{gathered}
\end{equation}
We then consider a vanishing-viscosity family
$\Psi_\eps:V\to[0,+\infty)$, $\eps>0,$ of dissipation potentials
approximating $\Psiz$:
\begin{equation}
  \label{eq:2}
  \Psi_\eps(v):=\Psiz(v)+\eps^{-1}\Psiv(\eps v) =: 
  \eps^{-1}\Psi_1(\eps v),\quad
  \Psi_0(v):=\Psiz(v)=\lim_{\eps\down0}\Psi_\eps(v)=\inf_{\eps>0}\Psi_\eps(v).
\end{equation}
Observe that the whole theory is restricted to the case $\Psi_\eps(v)
<+ \infty$. Indeed, allowing for $\Psi_\eps(v)=+\infty$ as in
unidirectional processes such as damage,  hardening, or fracture
(cf., e.g., \cite{DaFrTo05QCGN,MieRou06RIDP,MaiMie08?GERI,%
  KnRoZa2011,BabFraMor12}) would give rise to additional
complications, which we prefer not to address in this paper. Still, a
typical situation that is relevant in elastoplasticity is given by the
choices $V= L^p (\Omega;\R^m)$ for $p \in (1,\infty)$, $\| v \| =
(\int_\Omega |v(x)|^p \, \dd x )^{1/p}$, $\Psi(v) = \int_\Omega
\sigma_Y |v(x)|\, \dd x $, and $F(r) = \nu r^p$.  In particular,
$\Psi_\eps$ has the simple form $\Psi_\eps (v) =\int_\Omega \sigma_Y
|v(x)| +\eps^{p-1}\nu |v(x)|^p \, \dd x $.

\subsubsection*{\bfseries Subdifferential of the rate-independent dissipation and
the dual convex stability set}
$\Diss{\Bo}{}$ is the support function of the $w^*$-closed
and bounded convex subset of $V^*$
\begin{equation}
  \label{eq:18}
  K^*
  :=\Big\{\xi\in \Vanach^*:
  \pairing{}{}{\xi}{w} \le \Diss{}{}(w)\text{ for every }w\in \Vanach
  \Big\}\subset \V^*,\quad
  \Psiz(v)=\sup_{\xi\in K^*}\pairing{}{}\xi v,
\end{equation}
which will play a prominent role in the following.
$K^*$ is related to $\Psiz$ by two
different important relations:
first of all, it is the proper domain of the conjugate function of $\Psi^*$:
\begin{equation}
  \label{eq:59}
  \Psi^*(\xi):=\sup_{v\in V} \left( \langle \xi,v\rangle -\Psi(v)
  \right)=\mathrm I_{K^*}(\xi)= 
  \begin{cases}
    0&\text{if }\xi\in K^*,\\
    +\infty&\text{otherwise.}
  \end{cases}
\end{equation}
Second, $K^*$ can be characterized in terms of the subdifferential
$\partial\Psi:V\rightrightarrows V^*$ of $\Psi$,
defined as
\begin{equation}
  \label{eq:26}
  \xi\in \partial\Psi (v)\quad\Leftrightarrow\quad
  \langle \xi,w-v\rangle \le \Psi(w)-\Psi(v)
  \quad\forall w\in V,
\end{equation}
so that
\begin{equation}
  \label{eq:27}
  K^*=\partial\Psi(0);\qquad
  \xi\in \partial\Psi (v)\quad\Leftrightarrow \quad
  \xi\in K^* \text{ and } \langle \xi,v\rangle=\Psi(v).
\end{equation}

\subsubsection*{\bfseries The energy functional and its subdifferential}
We shall consider a time-dependent
\begin{equation}
  \tag{${\mathrm{E}.0}$} \label{Ezero}
  \text{lower semicontinuous
    energy functional $\cE: [0,T] \times D
 \to \R$, $D\subset V$.}
\end{equation}
To simplify some formulae, we will set $\cE_t(u)=+\infty$ if
$u\not\in D$ and we will assume the following properties:
\begin{description}
   \itemsep0.4em
\item[Coercivity]
  the map
  \begin{equation}
  \label{eq:17}
  \tag{E.1}
  \begin{gathered}
    u\mapsto \cg u:=\Psiz(u)+\sup_{t\in [0,T]} \cE_t(u)\quad\text{has
      compact sublevels in $\V$,}\\
    \text{i.e.~for every $E>0$ the set $D_E:=\{u\in D:\cg u\le E\}$ is compact.}
  \end{gathered}
  \end{equation}
\item[Power-control] for all $u \in \domainenergy$ the function $
  t\mapsto \ene tu$ is differentiable on $ [0,T]$ with derivative
  $\power tu:= \partial_t\ene tu$ satisfying for a constant $C_P\ge0$
 \begin{equation}
    \label{hyp:en3}
  \tag{${\mathrm{E}.2}$} 
  |\power tu|\leq C_P \big(\Psiz(u)+\ene tu\big),\quad
   \limsup_{w\to u,w\in D_E} \power t w\le \power t u
  \end{equation}
  for every $(t,u)\in (0,T)\times D$, $E>0$.
\item[$\Psiz$-uniform subdifferentiability] for every $E>0$ there
  exists an upper semicontinuous map $ \omega^E:[0,T]\times D_E\times
  D_E\to \R,$ with $\omega_\cdot^E(u,u)\equiv0$ for every $u\in D_E$,
  such that
  \begin{equation}
    \label{hyp:en-subdif}
  \tag{${\mathrm{E}.3}$}
  \begin{gathered}
  \EE_t(v)\ge \EE_t(u)+
  \la \xi,v-u\ra-
  \omega^E_t(u,v)\Dnorm{v-u} \qquad
  \forall\, t\in [0,T],\ u,v\in D_E,\quad
  \xi\in \partial\EE_t(u),
  \end{gathered}
  \end{equation}
  where
  \begin{equation}
  \label{eq:39}
  \Dnorm w:=\min\Big(\Psiz(w),\Psiz(-w)\Big).
  \end{equation}
\end{description}
Recall that the Fr\'echet subdifferential of $\cE_t$ is the possibly
multivalued map $\partial \EE_t :V \rightrightarrows V^*$ defined at
$u\in D$ by
\begin{equation}
  \label{eq:RIF2:16}
  \xi\in \partial\EE_t(u)\quad\Longleftrightarrow\quad
  \xi\in \Vanach^*,\quad
  \EE_t(v)-\EE_t(u)-\la \xi,v-u\ra\ge o(\Vnorm{v-u}) \text{ as $v\to
    u$ in $\Vanach$,} 
\end{equation}
Thus \eqref{hyp:en-subdif} prescribes a uniform and specific form for
the remainder infinitesimal term on the right-hand side of
\eqref{eq:RIF2:16}. For later use, we observe that \eqref{hyp:en3}
and the Gronwall Lemma yield
\begin{equation}
\label{gronwall-dixit}
0\le \Psiz(u)+\ene s u\le \cg u \leq \exp(C_P T)
\big(\Psiz(u)+\ene tu\big) \qquad \text{for all } s,t\in [0,T], \ u\in D.
\end{equation}
Since $\cE$ is lower semicontinuous, \eqref{gronwall-dixit} joint with
\eqref{eq:17} yields that the maps
\begin{equation}
  \label{eq:60}
  u\mapsto \Psiz(u)+ \cE_t(u)\quad\text{have
    compact sublevels in $\V$}\quad
  \forevery t\in [0,T].
\end{equation}

\begin{remark}
  \upshape
  Most of the results of the present paper could be extended to
  the cases when $\Psiz$ depends on the
  state of the system (as in \cite{MRS-dne}),
  or it is replaced by a distance on $D$
  (as in \cite{RMS08,MRS09}) and when the viscous
  correction $\Phi$ is a general convex superlinear functional
  (as in \cite{MRS10}).
  We have chosen the current simpler structure
  to focus on the main features and techniques of the vanishing-viscosity method
  in the infinite-dimensional setting.
\end{remark}

\subsection{Absolutely continuous and $\BV$ functions}
\label{s:appA}

As in Section \ref{ss:2.1} let $\Psi:V\to[0,\infty)$ be a gauge
function with $\Psi(v)>0$ if $v\neq0$ and let $Z$ a subset of $V$.
The function
\begin{equation}
  \label{eq:21}
  Z\ni u,v\mapsto \Cost\Psi{}uv:=\Psi(v-u)\quad\text{is an asymmetric continuous
    distance on }Z.
\end{equation}
We say that a curve $u:[0,T]\to Z$ is $\Psi$-absolutely continuous if
there exists  a nonnegative function $m\in L^1(0,T)$ such that
\begin{equation}
  \label{eq:13a}
  \Cost\Psi{}{u(t_0)}{u(t_1)}\le \int_{t_0}^{t_1} m(s)\,\dd s\quad
  \forevery 0\le t_0<t_1\le T.
\end{equation}
We denote by $\AC([0,T];Z,\Psi)$ the set of all $\Psi$-absolutely
continuous curves with values in $Z$.  There is a minimal function $m$
such that \eqref{eq:13a} holds \cite{AGS08,RMS08}, and with a slight abuse
of notation we denote it by $\scalardens \Psi u$, since it admits the
expression
\begin{equation}
  \label{eq:13}
  \scalardens\Psi u(t)=\lim_{h\to 0}\Psi\Big(\frac{u(t+h)-u(t)}h\Big)
  \quad\text{for $\Leb 1$-a.a.\ }t\in (0,T),
\end{equation}
so that $\scalardens\Psi u(t)=\Psi(\dot u(t))$ whenever $u$ is
differentiable at $t$.  Since $V$ has the Radon-Nikod\'ym property,
this happens at $\Leb 1$-a.a.~$t\in (0,T)$ ($\Leb 1$ denoting the
Lebesgue measure on $(0,T)$), when $\Psi$ is coercive: if this is the
case and $Z=V$, we will simply write $u\in \AC(0,T;V)$.

$ \Var{\Psi}uab$ is the
\emph{pointwise} total variation induced by $\Psi$ 
on the interval $[a,b] \subset
[0,T]$, viz.
\begin{equation}
  \label{ptwise-tvar} \!\!\!\!\!\!
  \Var{\Psi}uab:=
  \sup\Big\{\sum_{m=1}^M \Psi\big(u(t_m)-u(t_{m-1})\big):a=t_0< t_1<\cdots<t_{M-1}<t_M=b\Big\}.
\end{equation}
If $Z\subset V$, $\BV([0,T];Z,\Psi)$ will denote the set of all curves
$u:[0,T]\to Z$
with finite $\Psi$-total variation in $[0,T]$.
When $\Psi:=\|\cdot\|$ we will simply write $\BV([0,T];V)$
and we will omit the index $\Psiz$ in the symbol of the total variation.
Notice that $\BV([0,T];V)\subset \BV([0,T];V,\Psi)$ for every choice
of $\Psi$, whereas the opposite inclusion  only holds  when
$\Psi$ is coercive on $V$.

To every $u\in \BV([0,T];Z,\Psi)$
we can associate the nondecreasing scalar function
$\scalarV{}{}{} :\R\to [0,\infty)$
\begin{equation}
  \label{eq:14}
  \scalarV {}{}t:=
  \begin{cases}
    0&\text{if }t\le 0,\\
    \Var{\Psi}u0t&\text{if $t\in (0,T)$,}\\
    \Var{\Psi}u0T&\text{if $t\ge T$}
  \end{cases}
  \quad
  \text{with distributional derivative}\quad
  \scalarmu{}u =\frac{\dd}{\dd t}\scalarV {}{}{}.
\end{equation}
The finite Borel measure $\scalarmu{}u$ is supported in $[0,T]$ and it
can be decomposed into the sum
$\scalarmu{}u=\scalarmuco{}u+\scalarmuj {}u$ of a diffuse part
$\scalarmuco{}u$ (such that
$\scalarmuco{}u(\{t\})=0$ for every $t\in \R$),
and a jump part $\scalarmuj {}u$
concentrated in a countable set $\mathrm J_u\subset [0,T]$.

When $Z$ is compact (or when $\Psiz$ is coercive),
for every $\delta>0$ there exists a constant
$M_\delta>0$ such that (recall \eqref{eq:39} for the definition of $\Dnormname$)
\begin{equation}
  \label{eq:23}
  \|u-v\|\le \delta+ M_\delta\,
  \Dnorm{v-u}
  \quad
  \forevery u,v\in Z.
\end{equation}
By introducing the continuous and concave modulus of continuity
\begin{equation}
  \label{eq:121}
  \Omega_Z:[0,+\infty)\to [0,+\infty),\quad
  \Omega_Z(r):=\inf_{\delta>0}\delta+M_\delta\, r\quad
  \text{so that }
  \lim_{r\down0}\Omega_Z(r)=0,
\end{equation}
\eqref{eq:23} rewrites as 
\begin{equation}
  \label{eq:120}
  \|u-v\|\le \Omega_Z(\Dnorm{u-v}) \quad\forevery u,v\in Z.
\end{equation}
If \eqref{eq:23} holds, it is easy to show that
 a function $u\in \BV([0,T];Z,\Psi)$
is continuous in $[0,T]\setminus \rmJ_u$ and
its left and right
limits exist at every  $t\in [0,T]:$
\begin{equation}
    \label{eq:41}
    u(t_-):=\lim_{s\uparrow t}u(s),\ \
    u(t_+):=\lim_{s\downarrow t}u(s) \ \text{ with the convention }
    u(0_-):=u(0),\ u(T_+):=u(T),
\end{equation}
so that ${\mathrm{J}}_u$ admits the representation 
\begin{equation}
    \label{eq:88}
    {\mathrm{J}}_u:=\big\{t\in [0,T]:u(t_-)\neq u(t)\text{ or
    }u(t)\neq u(t_+)\big\}
\end{equation}
and
\begin{equation}
  \label{eq:19}
  \scalarmuj {}u(\{t\})=\Psi(u(t)-u(t_-))+\Psi(u(t_+)-u(t))\quad
  \forevery t\in \rmJ_u.
\end{equation}
Furthermore, 
$\scalarmuco {}u$ admits the  Lebesgue decomposition
$\scalarmuco{}u=\scalarmul{}u+\scalarmuc{}u$ with
$\scalarmul{}u\ll\Leb 1$
and $\scalarmuc{}u\perp\Leb 1$.
The density of $\scalarmul{}u$ with respect to $\Leb 1$
is provided by the same formula \eqref{eq:13} and
one has 
\begin{equation}
  \label{eq:24}
  \text{$u\in \AC([0,T];Z,\Psi)$ if and only if $\scalarmuj {}u=\scalarmuc{}u\equiv0$, with}
  \quad \Var\Psi uab=\int_a^b \scalardens \Psi u(t)\,\dd t.
\end{equation}
In this case, when $Z$ is compact or $\Psiz$ coercive, $u$ is a continuous curve.
In general we have
\begin{equation}
\label{repre-rn}
\begin{aligned}
  \Var{\Psi{}}uab&= \scalarmuco{}u([a,b])+\JVar{\Psi{}}uab,
\end{aligned}
\end{equation}
where the jump contribution $\JVar{\Psi{}}uab$ can be described
by
\begin{equation}
  \label{eq:37}
 \begin{aligned}  \JVar{\Psi}uab&:=
   \Cost{\Psi}{}{u(a)}{u(a_+)}+\Cost{\Psi}{}{u(b_-)}{u(b)}
\\ &\qquad +
    \sum_{t\in{\mathrm{J}}_u\cap
      (a,b)} \big(\Cost{\Psi}{}{u(t_-)}{u(t)} +
                  \Cost{\Psi}{}{u(t)}{u(t_+)}\big), \\
                  &\phantom{:}=
                  \Cost{\Psi}{}{u(a)}{u(a_+)}+\Cost{\Psi}{}{u(b_-)}{u(b)}+
                  \scalarmuj{}u((a,b)).
\end{aligned}
\end{equation}
\begin{remark}[Scalar vs.\  vector measures]
  \label{re:BVV}
  \upshape
  If $u\in \BV(0,T;V)$ all the previous definitions have an important
  vector counterpart in terms of the vector measure
  $u'_{\mathscr D}$ associated with  the distributional
  derivative of $u$:
$u_{\scrD}'$ is a Radon vector measure on $(0,T)$ with values in $\V$,
with finite total variation $\|u_\scrD'\|$.
The measure
$u'_\scrD$ can
be decomposed into the sum of the three mutually singular measures
\begin{equation}
  \label{eq:87}
  u_\scrD'=u'_{\mathscr{L}}+ u'_\mathrm{C} +u'_{\mathrm{J}},\quad
  \quad
  u'_\mathrm{d}:=u'_{\mathscr{L}}+ u'_\mathrm{C}, 
\end{equation}
where  $u'_{\mathscr{L}}$ is its absolutely continuous part with
respect to 
 $\Leb1$. 
 $u'_{\mathrm{J}}$
is a discrete measure concentrated on ${\mathrm{J}}_u$,
and $u'_\mathrm{C}$ is the so-called Cantor part,
still satisfying  $u'_\mathrm{C}(\{t\})=0$ for every $t\in [0,T]$.
Therefore $u'_{\mathrm{d}}=u'_{\mathscr{L}}+ u'_\mathrm{C} $   is the
diffuse part of the measure, which does not charge ${\mathrm{J}}_u$.

Since $V$ has the Radon-Nikod\'ym property,
$u$ is differentiable $\Leb 1$-a.e.~in $(0,T)$ (we denote by $\dot
u$ its derivative), and we can
express $u_{\rm d}'$ in terms of its density
$\nn$ with respect to its total variation
$\|u_{\rm d}'\|$ as
\begin{equation}
  u_{\rm d}'=\nn \|u_{\rm d}'\|\quad\text{where }
  \|\nn\|=1\ \text{$\|u_{\rm d}'\|$-a.e.},
  \quad
  u_{\mathscr L}'=\dot u\,\Leb 1,\quad
 \nn=\frac{\dot u}{\|\dot u\|}
    \quad\text{$\|u_{\mathscr L}'\|$-a.e.}.
\label{eq:11}
\end{equation}
  The relation to the previously introduced measures $\scalarmuco {}u$,  $\scalarmuc {}u$, and
  $\scalarmul {}u$ is 
  \begin{equation}
    \label{eq:15}
    \scalarmuco {}u=\Psiz(\nn) \|u_{\rm d}'\|,\quad
    \scalarmuc {}u=\Psiz(\nn) \|u_{\rm C}'\|,\quad
    \scalarmul {}u=\Psiz(\nn)\|u_{\mathscr L}'\|=
    \Psiz(\dot u)\Leb 1.
  \end{equation}
\end{remark}


\subsection{Two useful properties from the theory of gradient systems}
\label{ss:G1}
The assumptions on the dissipation potentials  $\Psiz$ and $\Psiv$  and on the energy $\cE$
stated in the previous section yield two important consequences, stated in Theorem \ref{thm-from-mrs12}
below, that
play a crucial role in the variational approach to gradient systems and
rate-independent evolutions.

Before stating them, let us recall that 
for every map
$\Lambda:\V \to (-\infty,+\infty]$ bounded from below by a
continuous and
affine function, $\Lambda^*:\V^*\to(-\infty,+\infty]$ will denote the
conjugate
\begin{equation}
  \label{eq:5}
  \Lambda^*(\xi):=\sup_{v\in \V}\pairing{}{}\xi v-\Lambda(v).
\end{equation}
For the functional $\Phi$ in \eqref{def-psiV} we have
\begin{equation}
  \label{eq:42}
  \Phi^*(\xi)=F^*\big(\|\xi\|_*\big),\quad
  \text{where}\quad
  F^*(s)=\sup_{r\ge0} rs-F(r),
\end{equation}
so that, by the inf-convolution duality formula (see
e.g. \cite[Thm.\,1, p.\,178]{ioffe-tihomirov}) and the monotonicity of
$F^*$  we find 
\begin{equation}
  \label{eq:3}
  \Diss\eps{*}(\xi)=\frac 1\eps\min_{z\in K^*}\Psiv^*(\xi-z)=
  \frac 1\eps\min_{z\in K^*}F^*\big (\|\xi-z\|_*\big)=
  \frac 1\eps F^*\Big(\min_{z\in K^*}\|\xi-z\|_*\Big).
\end{equation}
\begin{theorem}[{\cite[Prop.\,2.4]{MRS-dne}}]
\label{thm-from-mrs12}
  Under the assumptions of Section \ref{ss:2.1} the following
  properties hold.
  \begin{description}
  \item[\emph{Chain rule}] For every $u\in \AC([0,T];\V)$ and $\xi\in
    L^1(0,T;\V^*)$ with
    \begin{equation}
      \label{conditions-for-chain-rule}
      \begin{gathered}
        \sup_{t \in [0,T]} \left|\ene{t}{u(t)} \right|<+\infty, \ \
        \xi(t)\in-\frsub \ene t{u(t)} \text{ for a.a.\   $\, t \in (0,T)$, and}\\
        \int_0^T\Psi_\eps(\dot u(t))\, \dd t<+\infty,\quad
        \int_0^T\Psi_\eps^*(\xi(t))\, \dd t<+\infty,
      \end{gathered}
    \end{equation}
    the map $t \mapsto \ene t{u(t)}$ is absolutely continuous and
    \begin{equation}
      \label{eq:45tris}
      \begin{gathered}
        \frac {\mathrm{d}}{\mathrm{d}t}\ene t{u(t)} =-\la
        \xi(t),u'(t)\ra + \power t{u(t)} \quad\text{for
          a.a.}\, t \in (0,T)\,.
      \end{gathered}
    \end{equation}
  \item[\emph{Strong-Weak closedness of the graph of $(\cE,\frsub \cE)$}]
    For all sequences $(t_n)\subset [0,T]$,
    $(u_n) \subset \V$ and $(\xi_n)
    \subset \V^* $ we have the following condition:
    \begin{equation}
      \label{eq:45}
      \begin{aligned}
        &\text{if $t_n\to t$ in $[0,T]$, \ $u_n\to u $ in $\V,$ } \
        \text{$\xi_n\weakto \xi$ in 
          $\V^*$,} \ \xi_n \in \frsub \ene
        {t_n}{u_n}, 
        \\
        &\text{and if } \ene {t_n}{u_n} \to \EE \text{ in }\R,\qquad
        \text{then} \quad \xi \in \frsub \ene tu \ \text{ and } \
         \EE=\ene tu.
      \end{aligned}
    \end{equation}
\end{description}
Furthermore, \eqref{eq:45} implies that $\frsub \cE_t(u)$ is a
weakly$^*$-closed, convex subset (possibly empty) of $\V^*$. 
\end{theorem}

\subsection{Variational gradient systems}
\label{ss:G2}

We recall an application of the general
existence and approximation  result
of \cite{MRS-dne} for the Cauchy problem
associated with \eqref{viscous-dne}.
\begin{theorem}[\cite{MRS-dne}]
\label{th:0}
Let us assume that
\eqref{e:2.1}--\eqref{def-psiV} and
\eqref{Ezero}--\eqref{hyp:en-subdif}
hold.
%
%
%
Then, for every $u_{0,\eps} \in \domainenergy$  there exists  a curve $ u_\eps
\in \AC ([0,T];\V)$ solving \eqref{viscous-dne} and fulfilling the
Cauchy condition $u(0)=u_{0,\eps}.$
More precisely, there exists a function $\xi_\eps
\in L^1 (0,T;\V^*)$ fulfilling
\begin{equation}
  \label{xi-selection} \xi_\eps(t) \in -\frsub \ene t{u_\eps(t)}, \quad \xi_\eps(t)
\in  \partial\Psi_\eps(\dot u_\eps(t))
\quad \foraa\, t \in (0,T),
\end{equation}
and the \emph{energy identity} for all $0 \leq s \leq t \leq T$
  \begin{equation}
    \label{eq:52bis}
    \int_{s}^{t}
    \Big(\Psi_\eps(\dot{u_\eps}(r)){+}\Psi_\eps^*(\xi_\eps(r)) \Big)\,\dd r + \ene {t}{u_\eps(t)}=
    \ene
    {s}{u_\eps(s)} + \int_s^t \power {r}{u_\eps(r)} \, \dd r.
  \end{equation}
\end{theorem}

\subsubsection*{\bfseries Minimizing Movement solutions}
Theorem \ref{th:0} was proved in \cite[Thm.\,4.4]{MRS-dne} by passing
to the limit in the time-discretization scheme \eqref{eq:58}, see the
last paragraph of the introduction.  Here we quote the main
convergence result:

\begin{theorem}[Minimizing Movement solutions to \eqref{viscous-dne}]
  \label{thm:ref0}
  Under our standard assumptions \eqref{e:2.1}--\eqref{def-psiV} and
  \eqref{Ezero}--\eqref{hyp:en-subdif}, Problem \eqref{eq:58} has at
  least a solution $(\mathrm{U}^n_\taue)_{n=0}^{N_\tau}$.
  For every $\eps>0$
  there exist a sequence $\tau_k
  \downarrow 0$ as $k \to \infty$ and a limit solution
  $u_\eps \in \AC ([0,T];\V)$ to
  \eqref{xi-selection} and \eqref{eq:52bis}
  such that the
  piecewise affine interpolants $\pwL {U}{\tau}$
  satisfy
  \begin{equation}
    \label{conver-dne-approx} \pwL {U}{\tau_k,\eps} \to u_\eps \quad
    \text{in $\V$, uniformly in $[0,T]$.}
  \end{equation}
\end{theorem}
Since solutions obtained as such limits
have special properties, we will
call them \emph{Minimizing Movement solutions}
according to \cite{DeGiorgi93} (see also
\cite{AGS08}).
%
%
%
\section{Balanced Viscosity ($\BV$) solutions}
\label{ss:2.-bv}
Throughout this section we will keep to the notation and assumptions
of Section \ref{ss:2.1}, in particular we will suppose that $\Diss
B{}, \Diss\V{}$ fulfill \eqref{e:2.1}--\eqref{def-psiV} and that $\cE$
complies with \eqref{Ezero}--\eqref{hyp:en-subdif}.

After a discussion of the main concepts of \emph{contact potential} and
\emph{Finsler dissipation cost} in  \S \ref{ss:3.1},
we will introduce the notion of Balanced Viscosity ($\BV$) solutions
in \S \ref{ss:BV} and we will present the main results related to this crucial
concept.
The distinguished subclass of $V$-parameterizable solutions
will be considered in the last part \S\,\ref{ss:Vparam}.
%
%
 \subsection{Finsler dissipation functionals}
 \label{ss:3.1}

As in \cite{MRS10}, the
\emph{vanishing-viscosity contact potential}
$ \bptname : \V \times \V^* \to [0,+\infty)$
induced by the dissipation potentials $\Psiz_\eps$ is
\begin{equation}
\label{def-bipot-sez2}
\begin{aligned}
 \bpt{v}{\xi} & :=\inf_{\eps>0} \Big(\Diss{\eps}{}(v) +
 \Diss{\eps}{*}(\xi)\Big),\qquad v\in V,\ \xi\in V^*.
 \end{aligned}
\end{equation}
The representation formula \eqref{eq:3} for $\Diss{\eps}{*}$ and the fact that
\begin{displaymath}
  \inf_{\eps>0} \eps^{-1}\Big(F(\eps r)+F^*(s)\Big)=rs\quad\forevery r,s\ge0,
\end{displaymath}
yield the useful splitting of $\bptname$:
\begin{equation}
  \label{eq:6}
  \bpt v\xi=\Diss{}{}(v)+
  \|v\|\,\min_{z\in K^*}\|\xi-z\|_*.
\end{equation}

\begin{remark}[More general viscous dissipations and contact
  potentials]
  \label{rem:bip}
  \upshape The particular form \eqref{def-psiV} of $\Phi$ allows for
  the simple representation \eqref{eq:6} of $\mathfrak p$, which is
  useful to understand the role played by the two different
  viscosities.  The general case concerning arbitrary convex
  superlinear functions $\Phi$ has been analyzed in \cite{MRS10} and
  almost all the crucial properties can also be adapted to the present
  infinite-dimensional setting. Here we just mention that every
  contact potential is convex and degree-$1$ homogeneous with respect
  to its first variable and it fulfills the Fenchel inequality
  \begin{equation}
    \label{basic-bpt-prop} \bpt{v}{\xi} \geq \pairing{}{}{\xi}{v}, \quad
    \text{and}\quad
    \begin{cases}
      \bpt{v}{\xi} \geq \Diss{\Bo}{}(v) &\text{for all } (v,\xi)\in
      \V \times \V^*,\\
      \bpt{v}{\xi} =\Diss{\Bo}{}(v)&\text{if and only if } \xi\in K^*.
    \end{cases}
  \end{equation}
\end{remark}

\noindent
Next, we associate with $\bptname$ and with the Fr\'echet
subdifferential $\frsub\cE$ the time-dependent family of Finsler
dissipation functionals
\begin{equation}
\label{def-vvm}
\begin{aligned}
  \vvmname:  [0,T]  \times \domainenergy   \times  \V \to
  [0,+\infty], \quad
  \vvm tuv :=
  \inf\Big\{ \bpt{v}{\xi}: \xi \in -\frsub \ene tu \Big\},
  \end{aligned}
\end{equation}
where we adopt the standard convention $\inf\emptyset=+\infty$.
Notice that when $\frsub\ene tu\neq\emptyset$ the $\inf$ in formula
\eqref{def-vvm} is attained; moreover, the functional $v \mapsto \vvm
tuv$ is lower semicontinuous, convex, and positively $1$-homogeneous.

In accord with \eqref{eq:6} it will also be useful to split $\vvm tuv$
into the sum of the dissipation $\Psiz(v)$ (independent of $u$) and of
the correction term induced by the viscous norm $\|\cdot\|$ and
$\partial\cE$, viz.
\begin{equation}
  \label{eq:7}
  \vvm tuv=\Psiz(v)+\vvmV tuv,\quad
  \vvmV tu{}:=
  \inf\Big\{\|\xi-z\|_*: \xi \in -\frsub \ene tu,\ z\in K^*\Big\}.
\end{equation}
By \eqref{eq:45}, for every $E>0$ the function $\fre:[0,T]\times D\to
[0,\infty]$ satisfies
the crucial properties
\begin{equation}
  \label{eq:62}
  \fre\text{ is l.s.c.\ in }[0,T]\times D_E  \quad \text{ and } \quad
  \fre_t(u)=0 \ \Longleftrightarrow \ K^* + \partial\ene tu\ni 0,
\end{equation}
where $D_E$ denotes the $E$-sublevel of the energy, cf.\ \eqref{eq:17}. 

If $\Psiz$ were coercive on $V$, then the Finsler cost associated to
$\vvmname_t$ could be simply defined as
\begin{equation}
  \label{eq:8}
  \Cost{\vvmname}t{u_0}{u_1}:=\inf\Big\{
  \int_{r_0}^{r_1}  \vvm{t}{\vartheta(r)}{\dot \vartheta(r)}\,\dd r:
  \vartheta\in \AC([r_0,r_1];\V),\ \vartheta(r_i)=u_i,\ i=0,1\Big\},
\end{equation}
and it would be possible to show that
the infimum in \eqref{eq:8} is attained whenever the cost is finite.
Notice that, since $\vvm tu\cdot$ is positively $1$-homogeneous,
the choice of the interval $[r_0,r_1]$ in \eqref{eq:8} is irrelevant
and one can also assume that
the competing curves $\vartheta$ belong to
$\mathrm{Lip}([r_0,r_1];\V)$.

On the other hand, since $\Psiz$ is not coercive in general, the definition
\eqref{eq:8} has to be conveniently adapted to cover the case of
curves $\vartheta$ that may lack differentiability at every time.
The next definition focuses on this aspect (see \S \ref{s:appA}
for
BV and AC curves with respect to $\Psiz$).
\begin{definition}[Admissible curves]
  \label{def:admissible}
  A curve $\vartheta:[r_0,r_1]\to \V$ is called \emph{admissible}
  if it belongs to $\AC([r_0,r_1];D_E,\Psiz)$ for some $E>0$, and if 
its restriction to the (relatively) open set
  \begin{equation}
    \label{eq:10}
    G_t=G_t[\vartheta]:=\big\{r\in [r_0,r_1]:
    \vvmV t{\vartheta(r)}{}
    >0\big\},
  \end{equation}
  belongs to $\AC_{\rm loc}(G_t[\vartheta];\V)$.
  We call $\calT_t(u_0,u_1)$ the class
  of all \emph{admissible transition curves} $\vartheta:[0,1]\to\V$
  such that $\vartheta(i)=u_i$,
  $i=0,1,$ and we set
  \begin{equation}
    \label{eq:28}
    \vvmname_t[\vartheta;\vartheta'](r):=
    \begin{cases}
      \vvm t{\vartheta(r)}{\dot\vartheta(r)}=
      \Psiz(\dot\vartheta(r))+\vvmV
      t{\vartheta(r)}{\dot\vartheta(r)}&\text{if }r\in G_t[\vartheta],\\
      \scalardens\Psiz \vartheta(r)&\text{if }r\in [0,1]\setminus
      G_t[\vartheta].
    \end{cases}
  \end{equation}
\end{definition}
\begin{remark}
  \upshape
  Let us add a few comments on the previous definition.
  First of all, as we discussed in Section \ref{s:appA},
  we notice that the continuity of $\vartheta$ follows from the
  compactness of $D_E$ in $\V$
  and the fact that $\Psiz$ is continuous and nondegenerate, so that
  $\Psiz(v)=0\ \Rightarrow\ v=0$.

  Once $\vartheta$ is continuous, the l.s.c.~property of $\fre$ stated in \eqref{eq:62}
    implies that the set
  $G_t[\vartheta]$
  defined
   in \eqref{eq:10} is open.
  Since $V$ has the Radon-Nikod\'ym property,
  $\vartheta$ is differentiable $\Leb 1$-a.e.~in $G_t[\vartheta]$.
  It is immediate to see that for every
  admissibile curve $\vartheta$
  \begin{equation}
    \label{eq:29}
    \int_{0}^{1} \vvmname_t[\vartheta;\vartheta'](r)\,\dd r=
    \int_{0}^{1} \scalardens\Psi\vartheta(r)\,\dd r+
    \int_{G_t[\vartheta]}  \vvmV{t}{\vartheta(r)}{\dot \vartheta(r)}\,\dd r.
  \end{equation}
\end{remark}

We are now in the position to extend the definition  \eqref{eq:8} of  $\Delta_\frf$. 

\begin{definition}[Finsler dissipation cost]
  \label{def:Finsler_diss}
  Let $t \in [0,T]$ be fixed and let us consider $u_0,u_1\in D$.  The
  (possibly asymmetric) Finsler cost induced by $\vvmname$ at the time
  $t$ is given by
  \begin{align}
    \label{eq:69}
    \Cost{\vvmname}t{u_0}{u_1}:=&\inf_{ \vartheta\in \calT_{t}(u_0,u_1)}
    \int_{0}^{1}
    \vvmname_t[\vartheta,\vartheta'](r)\,\dd r
    \\=&\label{eq:53}
    \inf_{ \vartheta\in \calT_{t}(u_0,u_1)}
    \int_{0}^{1}
    \scalardens\Psi\vartheta(r)\,\dd r+
      \int_{G_t[\vartheta]}  \vvmV{t}{\vartheta(r)}{\dot \vartheta(r)}\,\dd r,
     \end{align}
  with the usual convention
  of setting $\Cost{\vvmname}t{u_0}{u_1}
  =+\infty$ if
  $\calT_t(u_0,u_1)$ is empty.
\end{definition}
\noindent
Let us notice that in general
$\Cost{\vvmname}t{\cdot}{\cdot}$ is not symmetric, unless
 $\Psiz$ is symmetric, and
 that
\begin{equation}
  \label{eq:85}
  \Cost{\vvmname}t{u_0}{u_1}\ge \Cost{\Psiz}{}{u_0}{u_1}\quad
  \forevery u_0,u_1\in D,\ t\in [0,T].
\end{equation}
This follows from the fact that 
in \eqref{eq:53} we have
\begin{displaymath}
  \int_{0}^{1} \scalardens\Psi\vartheta(r)\,\dd r=
  \Var{\Diss {}{}}{\vartheta}{0}{1} \ge \Psiz(u_1-u_0)=\Cost{\Psiz}{}{u_0}{u_1}.
\end{displaymath}
In the next important result we collect a few crucial properties of
the Finsler dissipation cost, namely the existence of optimal
transition paths and the lower semicontinuity properties needed in
what follows.  Theorem \ref{thm:fcost} will be proved in
Section \ref{ss:clsc}. 
 \begin{theorem}
   \label{thm:fcost}
   Let   \eqref{e:2.1}--\eqref{def-psiV}
    and
    \eqref{Ezero}--\eqref{hyp:en-subdif} hold. 
   Let $t\in [0,T]$, $E>0$ and $u_-,u_+\in D_E$.
   \begin{enumerate}[\rm (F1)]
   \item If
     $\Delta_{\frf_t}(u_-,u_+)<\infty$ there exists a transition path
     $\vartheta\in \calT_t(u_-,u_+)$ attaining  the infimum in
     \eqref{eq:53}.
     Moreover
     \begin{equation}
       \label{eq:36}
       \Cost\vvmname t{u_-}{u_+}\ge \Big|\ene t{u_-}-\ene t{u_+}\Big|.
     \end{equation}
     \item
       If $u_{0,n},u_{1,n}\in D_E$, $n\in \N$, then
       \begin{equation}
         \label{eq:65}
         \lim_{n\to\infty}u_{0,n}=u_-,\quad
         \lim_{n\to\infty}u_{1,n}=u_+\quad
         \Longrightarrow\quad
         \liminf_{n\to\infty}\Cost\frf t{u_{0,n}}{u_{1,n}}\ge
         \Cost\frf t{u_-}{u_+}.
       \end{equation}
     \item
       If $u_n\in \AC([\alpha_n,\beta_n];V)$,
       $\tilde u_n: [\alpha_n,\beta_n]\to D_E$ measurable,
       $\xi_n\in L^1(\alpha_n,\beta_n;V^*)$,
       $\eps_n>0$, $n\in \N$, are sequences
       satisfying
       \begin{equation}
         \label{eq:92first}
         \lim_{n\to\infty}
         \sup_{r\in [\alpha_n,\beta_n]} \| \tilde u_n(r)-u_n(r)\|=0,\quad
         \xi_n(r)\in -\partial\ene r{\tilde u_n(r)}\quad\text{for
           a.a.~}r\in (\alpha_n,\beta_n),
       \end{equation}
      \begin{equation}
         \label{eq:63}
         \lim_{n\to\infty}u_n(\alpha_n)=u_-,\
         \lim_{n\to\infty}u_n(\beta_n)=u_+,\quad
         \lim_{n\to\infty}\alpha_n=\lim_{n\to\infty}\beta_n=t,
       \end{equation}
       and
       %
       \begin{equation}
         \label{eq:64}
         \lim_{n\to\infty}\eps_n=0,\quad
         \Delta:=\lim_{n\to\infty}\int_{\alpha_n}^{\beta_n}\Big(\Psiz_{\eps_n}(\dot
         u_n)+
         \Psiz_{\eps_n}^*(\xi_n)\Big)\,\dd r <\infty,
       \end{equation}
       then
       there exist an increasing  subsequence $(n_k)_k\subset \N$,
       increasing and
       surjective time  rescalings 
        $\sft_{n_k}\in \AC([0,1];[\alpha_{n_k},\beta_{n_k}])$,
       and an admissible transition
       $\vartheta\in \calT_t(u_-,u_+)$ such
       that
       \begin{gather}
         \label{eq:109}
         \lim_{k\to\infty}u_{\eps_{n_k}}\circ\sft_{n_k}= \vartheta\quad
         \text{strongly in $V$, uniformly on }[0,1],
         \quad
         \int_0^1 \frf_t[\vartheta,\vartheta'](r)\,\dd r\le \Delta.
       \end{gather}
       In particular, whenever \eqref{eq:92first} and \eqref{eq:63} hold,
       along any  sequence $\eps_n\down0$ we have 
       \begin{equation}
         \label{eq:64bis}
         \liminf_{n\to\infty}\int_{\alpha_n}^{\beta_n}\Big(\Psiz_{\eps_n}(\dot
         u_n)+
         \Psiz_{\eps_n}^*(\xi_n)\Big)\,\dd r\ge \Cost \frf t{u_-}{u_+}.
       \end{equation}
   \end{enumerate}
\end{theorem}

Solutions to \eqref{viscous-dne}, with $\tilde u_n=u_n$, provide a
particularly important example of sequences in assertion (F3) of Thm.\
\ref{thm:fcost}.  Notice that by \eqref{eq:36} the Finsler cost
controls the amount of energy dissipation between two arbitrary points
at a fixed time $t$.  On the other hand, \eqref{eq:64bis} shows that
$\Delta_{\vvmname}$ captures the concentration of the asymptotic
energy dissipation of a family of solutions to the viscous gradient
flow \eqref{xi-selection}.

We now use the Finsler cost $\Delta_\frf$ to characterize the minimal
dissipated energy along \emph{any} curve $u \in
\BV_{\Psiz}([0,T];\V)$, by means of a suitable notion of total
variation, which involves $\Delta_\frf$ to measure the contributions
due to the jumps of $u$ (recall \eqref{repre-rn} and \eqref{eq:37}).
 \begin{definition}[Jump and total variation induced by $\vvmname$]
  \label{def:Finsler_jump}
  Let $E>0$ and
  $u\in \BV([0,T];D_E,\Psiz)$ be a given curve with jump set ${\mathrm{J}}_u$.
  For every subinterval $[a,b]\subset [0,T]$
the  \emph{jump variation} of $u$ induced by $\vvmname$ on
$[a,b]$ is
\begin{equation}
  \label{eq:37bis}
  \begin{aligned}
    \JVar{\vvmname}uab &:=
    \Cost{\vvmname}a{u(a)}{u(a_+)}+\Cost{\vvmname}b{u(b_-)}{u(b)}
    \\
    & \quad + \sum_{t\in {\mathrm{J}}_u\cap
      (a,b)} \big(\Cost{\vvmname}t{u(t_-)}{u(t)}+
         \Cost{\vvmname}t{u(t)}{u(t_+)}\big).
  \end{aligned}
\end{equation}
The $\frf$-total variation induced of $u$ on $[a,b]$ for $a<b$ is
\begin{align}
  \label{eq:81}
    \pVar{\vvmnametil}uab &:=
    \Var{\Psiz}  uab - \JVar{\Psiz}uab + \JVar{\vvmname}uab
    \\ & \:=
    \scalarmuco{} u(a,b)+\JVar{\vvmname}uab.
    \label{eq:81bis}
 \end{align}
\end{definition}

\begin{remark}
 \label{rmk:not-standard}
 \upshape As already pointed out in \cite[Rmk.\,3.5]{MRS10},
 $\pVarname{\vvmnametil}$ is not a \emph{standard} total variation
 functional: for instance, it is not induced by any distance on $\V$,
 and it is not lower semicontinuous with respect to pointwise
 convergence in $\V$, unless a further local stability constraint is
 imposed.

 Nevertheless, $\pVarname\frf$ enjoys the nice additivity property
 \begin{equation}
   \label{eq:68}
   \pVar\frf u ab+\pVar\frf ubc=\pVar\frf uac\quad\text{whenever}\quad
   0\le a<b<c\le T.
 \end{equation}
\end{remark}

\subsection{Balanced Viscosity (BV) solutions}
\label{ss:BV}
Based on Definition \ref{def:Finsler_jump}, we can now specify the
concept of \emph{Balanced Viscosity $(\BV)$ solution} to the
rate-independent system generated by $\RIS$: the global stability
condition in the definition of \emph{energetic solutions} is replaced
by the \emph{local} stability condition \eqref{eq:65bis}, and the
energy balance features the total variation functional
$\pVarname{\vvmnametil}$.  As usual, we will always assume that $\Diss
B{}, \Diss\V{}$ fulfill \eqref{e:2.1}--\eqref{def-psiV} and that $\cE$
complies with \eqref{Ezero}--\eqref{hyp:en-subdif}.

\begin{definition}[$\BV$ solutions]
  \label{def:BV-solution}
  A curve $u\in \BV([0,T];D,\Psiz)$ is a \emph{$\BV$ solution of the
    rate-independent system $\RIS$} if the \emph{local stability}
  \eqref{eq:65bis} and the \eqref{eq:84}-\emph{energy balance} hold:
  \begin{equation}
    \label{eq:65bis}
    \tag{$\mathrm{S}_\mathrm{loc}$}
     K^* + \frsub \ene t{u(t)}\ni 0 
    \quad \text{for all}\quad t \in [0,T] \setminus \mathrm{J}_u,
  \end{equation}
  \begin{equation}
    \label{eq:84}
    \pVar{\vvmnametil}u{0}{t}+\ene{t}{u(t)}=\ene{0}{u(0)}+
    \int_{0}^{t} \power s{u(s)}\,\mathrm{d}s \quad \text{ for all } t\in (0,T].
    \tag{E$_{\vvmnametil}$}
  \end{equation}
\end{definition}
\noindent
Every BV solution $u$ to the \ris\ $\RIS$ satisfies the
energy balance in each subinterval
\begin{equation}
    \label{eq:72}
    \pVar{\frf}u{s}{t}+\ene{t}{u(t)} = \ene{s}{u(s)}+
    \int_{s}^{t} \power r{u(r)}\,\mathrm{d}r\quad
    \forevery 0\le s<t\le T,
\end{equation}
thanks to \eqref{eq:84} and the additivity \eqref{eq:68} of the total
variation functional $\pVarname{\frf}$.

Before studying other properties and characterizations of
balanced viscosity solutions, let us first present 
our main existence and convergence results.

\subsubsection*{\bfseries Main existence and convergence results}
Our first result states the convergence in the vanishing-viscosity
limit $\eps \down 0$ of solutions to \eqref{viscous-dne} to a $\BV$
solution of the rate-independent system $\RIS$.  As a byproduct, we
can prove in this way the existence of $\BV$ solutions.  Let us
emphasize that Definition \ref{def:BV-solution} of $\BV$ solutions is
only inspired by the vanishing-viscosity approach but otherwise
completely independent of it.  We postpone the proofs to Section
\ref{ss:8-vanvisc}.

The reader should be aware that, here and in what follows, we will
call a sequence $(\eps_k)_k $ converging to $0$ simply a
\emph{vanishing sequence}.
 
\begin{theorem}[Existence of BV solutions and convergence of
    viscous approximations]
    \label{th:1}
    If \eqref{e:2.1}--\eqref{def-psiV}
    and
    \eqref{Ezero}--\eqref{hyp:en-subdif} hold,
    then for every $u_0\in D$ there exists a
    \BV\ solution
    $u$ of the \ris\  $\RIS$.

    Moreover
    for every family $(u_\eps,\xi_\eps)_\eps \subset \AC
    ([0,T];\V) \times L^1 (0,T;\V^*)$  of solutions of
    the doubly nonlinear equation  \eqref{xi-selection} with
\begin{equation}
\label{conve-initi-data} u_\eps (0) \to u_0 \quad \text{in $\V$
 and}\quad \ene 0{u_\eps (0)}\to \ene 0{u_0} \quad \text{as
$\eps \down 0$}
\end{equation}
and
for every vanishing sequence $(\eps_k)_k $ there exist $E>0$, a
further (not relabeled) subsequence, and a limit function $u\in \BV
([0,T];D_E,\Psiz)$ such that as $k \to \infty$
\begin{align}
\label{e:conv1}
&u_{\eps_k}(t) \to u(t) \quad \text{in $\V$ }\text{for all
$t \in [0,T],$}
\\
\label{e:conv2}
&\lim_{k \to \infty} \ene t{u_{\eps_k}(t)}=\ene
t{u(t)} \quad \text{for all $t \in [0,T],$}
\\
 \label{e:conv3}
&    \pVar{\vvmnametil}{u}{s}{t}  = \lim_{k \to
  \infty}\pVar{\vvmnametil}{u_{\eps_k}}{s}{t}
    = \lim_{k \to \infty}
\int_{s}^{t}
    \Big(\Psi_{\eps_k}(\dot{u}_{\eps_k}(r)){+}\Psi_{\eps_k}^*(-\xi_{\eps_k}(r)) \Big)\,\dd r
\end{align}
for all $0\le s<t\le T$.
Any pointwise limit function $u$ obtained in this way
is a $\BV$ solution to the \ris\  $\RIS$.
\end{theorem}
Let us emphasize that, in view of the above result, \emph{every} limit
point $u$ of solutions $(u_\eps)_\eps$ of \eqref{xi-selection} such
that \eqref{e:conv1}--\eqref{e:conv3} hold is a $\BV$ solution.

The next theorem concerns the convergence of the discrete solutions of
the \emph{viscous} time-incremental problem \eqref{eq:58}, as
\emph{both} the viscosity parameter $\eps$ \emph{and} the time-step
$\tau$ tend to zero. Similar results for the finite-dimensional case
were obtained in \cite[Thm.\,4.10]{MRS10}.

\begin{theorem}[Discrete-viscous approximations converge to BV solutions]
  \label{th:2}
  Assume that \eqref{e:2.1}--\eqref{def-psiV} and
  \eqref{Ezero}--\eqref{hyp:en-subdif} hold.  Let $u_0 \in
  \domainenergy$ be fixed, and let $(\pwL{U}{\tau,\eps})_{\tau,\eps}$
  be a family of piecewise  affine interpolants of discrete solutions
  $(\mathrm{U}^n_\taue)_{n,\taue}$  to \eqref{eq:58}, with
\begin{equation}
  \label{eq:114}
  \Utaue 0 \to u_0 \quad \text{in $\V$
    and}\quad \ene 0{\Utaue 0}\to \ene 0{u_0} \quad \text{as
    $\tau,\eps \down 0$}.
\end{equation}
Then
for all sequences $(\tau_k,\eps_k)_{k\in \N}$ satisfying
\begin{equation}
    \label{eq:103-k}
    \lim_{k\to\infty} \eps_k=
    \lim_{k\to\infty} \frac{\tau_k}{\eps_k}=0,
\end{equation}
there exists $E>0$, a  (not relabeled) subsequence  and a curve $u\in
\BV([0,T];D_E,\Psiz)$ such that 
\begin{align}
\label{e:conv1-discr} &
 \pwC {U}{\tau_k,\eps_k}(t)\to u(t) \quad
 \text{in $\V$ for all
$t \in [0,T],$}
\\
\label{e:conv2-discr} & \ene{t}{\pwC {U}{\tau_k,\eps_k}(t)} \to \ene
t{u(t)}\quad \text{for all $t \in [0,T],$}
\end{align}
as $k \to \infty$, and the limit $u$ is a $\BV$ solution to the \ris\  $\RIS$.
\end{theorem}

We now aim to shed more light onto the definition and the properties
of $\BV$ solutions: first of all, we derive a characterization of
$\BV$ solutions in terms of a one-sided version of the energy identity
\eqref{eq:84}, based on the \emph{chain-rule} inequality stated in
Theorem \ref{prop:bv-chainrule}.  A second characterization is given
through a ``metric'' subdifferential inclusion and a set of jump
conditions.

\subsubsection*{\bfseries Chain-rule inequalities and characterizations of $\BV$ solutions}
The next result is the infinite-dimensional analogue of
\cite[Prop.\,4]{MRS09} and is especially adapted to rate-independent
systems. In particular, the fact that $\pVarname{\vvmnametil}$ is not
a true total variation functional is here compensated by assuming that
$u$ fulfills the local stability condition \eqref{eq:65bis}.

\begin{theorem}[A chain-rule inequality for BV curves]
  \label{prop:bv-chainrule}
  If $u \in \BV ([0,T];D_E,\Psiz)$, $E>0$, satisfies the local
  stability condition \eqref{eq:65bis} and $\pVar \frf u0T<\infty$,
  then the map $t\mapsto e(t):=\ene t{u(t)}$ belongs to $\BV ([0,T])$
  and satisfies the following chain-rule inequality:
\begin{equation}
\label{ch-rule-ineq}
\Big| e(t_1)  - e(t_0) 
-\int_{t_0}^{t_1} \power t{u(t)} \dd t \Big|\leq
\pVar{\vvmnametil}u{t_0}{t_1} \quad \text{for all } 0 \leq t_0 \leq
t_1 \leq T.
\end{equation}
If moreover $u\in \BV([0,T];V)$ and $\xi:[0,T]\to K^*$ is a Borel map
such that $\xi(t)\in -\partial \ene t{u(t)}$ for every $t\in
[0,T]\setminus \mathrm J_u$ then the diffuse part $e_{\rm d}'$ of the
distributional derivative $e_\scrD'$ of $e$ can be represented as
(recall \eqref{eq:11})
\begin{equation}
  \label{eq:86}
  e_{\rm d}'=-\langle \xi,\nn\rangle \|u_\rmd'\|+
  \power\cdot{u}\Leb 1=
  -\langle \xi,\nn\rangle \|u_{\rm C}'\|+
  \Big({-}\langle \xi,\dot u\rangle+
  \power\cdot{u}\Big) \Leb 1,
\end{equation}
where $\nn$ is as in \eqref{eq:11}, and $u_\rmd'$, $u_{\rm C}'$ are from \eqref{eq:87}. 
\end{theorem}
\noindent Indeed, \eqref{ch-rule-ineq} is the counterpart to
 the
\emph{parameterized} chain-rule inequality which shall be stated in Theorem
\ref{th:3.8} ahead.
Both Theorems will be proved in Section \ref{s:chain}.

As a direct consequence of Theorem \ref{prop:bv-chainrule} we
have a \emph{characterization} of $\BV$ solutions in terms of a
single, global in time, energy-dissipation inequality.

\begin{corollary}[A global energy-dissipation inequality
  characterizing BV solutions]
\label{prop:BV-charact}
A curve\\
 $u \in \BV ([0,T];D_E,\Psiz)$  for some $E>0$ is a $\BV$ solution to the
\ris\  $\RIS$ 
if and only if it satisfies the local
stability \eqref{eq:65bis} and the one-sided global in time version of
\eqref{eq:84}, viz.
\begin{equation}
    \label{eq:84-oneside}
    \pVar{\vvmnametil}u{0}{T}+\ene{T}{u(T)} \leq \ene{0}{u(0)}+
    \int_{0}^{T} \power s{u(s)}\,\mathrm{d}s\,.
    \tag{E$_{\vvmnametil,\mathrm{ineq}}$}
  \end{equation}
\end{corollary}
\begin{proof}
  In order to deduce the energy balance \eqref{eq:84} from \eqref{eq:84-oneside}, 
  we define $a(t):= \ene t{u(t)} -\int_0^t\power s{u(s)}\,\mathrm{d}s $ and
  $v(t):= \pVar{\vvmnametil}u{0}{t}$  such that
  \eqref{eq:84-oneside} takes the form $a(T)+v(T)\leq a(0)+v(0)$,
  because $v(0)=0$.  The additivity  \eqref{eq:68} gives
  $\pVar{\vvmnametil}u{s}{t}= v(t) -v(s)$, so that 
   the chain-rule estimate \eqref{ch-rule-ineq} rephrases as
  $|a(t) -a(s)| \leq v(t) - v(s)$ for all $0 \leq s \leq t \leq T$. 
  This implies the monotonicity $a(t)+v(t)\geq
  a(s)+v(s)$, and  we conclude $a(t)+v(t)=a(0)+v(0)$ for all $t$,
  which is \eqref{eq:84}. 
\end{proof}

The importance of using the viscous total variation induced by
$\vvmname$ (instead of the simpler one associated with $\Psiz$) is
clarified by the next result, characterizing the jump conditions.

\begin{theorem}[Local stability, $(\Psiz)$-energy dissipation and jump conditions]
  \label{thm:inquality+jump}
  A curve\\
  $u\in \BV([0,T];D_E,\Psiz)$ is a $\BV$ solution of the \ris\ $\RIS$
  if and only if it satisfies the local stability condition
  \eqref{eq:65bis}, the $(\Psiz)$-energy dissipation inequality
  \begin{equation}
    \label{eq:21bis}
    \Var{\Psiz}u{s}{t}+\ene{t}{u(t)} \leq \ene{s}{u(s)}+
    \int_{s}^{t} \power r{u(r)}\,\mathrm{d}r\quad
    \forevery 0\le s<t\le T,
    \tag{E$_{\Psiz,\mathrm{ineq}}$}
  \end{equation}
  and the following jump conditions at each point $t\in{\mathrm{J}}_u$ of the jump set
  \eqref{eq:88}
  \begin{equation}
    \label{eq:67}
      \begin{aligned}
    \ene{t}{u(t)}-\ene t{u(t_-)}&=-\Cost{\vvmnametil}t{u(t_-)}{u(t)},\\
    \ene{t}{u(t_+)}-\ene t{u(t)}&=-\Cost{\vvmnametil}t{u(t)}{u(t_+)},\\
    \ene{t}{u(t_+)}-\ene
    t{u(t_-)}&=-\Cost{\vvmnametil}t{u(t_-)}{u(t_+)}=
    -\Big(\Cost{\vvmnametil}t{u(t_-)}{u(t)}+
    \Cost{\vvmnametil}t{u(t)}{u(t_+)}\Big).
  \end{aligned}
    \tag{J$_{\BV}$}
  \end{equation}
\end{theorem}
\begin{proof}
  If $u$ is a BV solution to $\RIS$, then \eqref{eq:21bis} is a
  trivial consequence of the energy balance \eqref{eq:72} since
  $\pVar\frf ust\ge \Var\Psiz ust$ for every interval $[s,t]$.  The
  jump conditions \eqref{eq:67} follow by writing \eqref{eq:72} in the
  intervals $[t,t+\eta]$ or $[t-\eta,t]$ for small $\eta>0$ and then
  passing to the limit as $\eta\down0$.

  In order to prove the converse implication,
  let suppose that $\mathrm J_u=(t_n)_n\subset (0,T)$ and
  let us call $0=\sft_0<\sft_1<\cdots<\sft_{N}<\sft_{N+1}=T$
  an ordered subdivision of $[0,T]$ such that $\{\sft_1,\sft_2,\cdots,\sft_{N}\}$
  is a permutation of $\{t_1,t_2,\cdots,t_{N}\}\subset \mathrm J_u$.

  Writing \eqref{eq:21bis} in each interval $[\sft_i+\eta,\sft_{i+1}-\eta]$
  for sufficiently small $\eta>0$ and taking the limit as $\eta\down0$, also 
   recalling $\Var\Psiz uab\ge \scalarmuco{} u(a,b)$  (cf.\ \eqref{repre-rn}), 
   we get
  \begin{equation}
    \label{eq:71}
    \scalarmuco {}u({\sft_i},{\sft_{i+1}})\le
    \ene {\sft_i}{u(\sft_{i,+})}-
    \ene {\sft_{i+1}}{u(\sft_{i+1,-})}+
    \int_{\sft_i}^{\sft_{i+1}} \power s{u(s)}\,\dd s.
  \end{equation}
  From \eqref{eq:67} and \eqref{eq:85}  we obtain
  \begin{align*}
    \Cost\vvmnametil {\sft_i}{u(\sft_{i})}{u(\sft_{i+})} &+\scalarmuco
    {}u({\sft_i},{\sft_{i+1}}) 
    + \Cost\vvmnametil {\sft_{i+1}}{u(\sft_{i+1,-})}{u(\sft_{i+1})}
    \\&\le  \ene {\sft_i}{u(\sft_{i})}-
    \ene {\sft_{i+1}}{u(\sft_{i+1})}+
    \int_{\sft_i}^{\sft_{i+1}} \power s{u(s)}\,\dd s,
  \end{align*}
  so that summing up all the contributions (recalling that
  $u(\sft_{0,+})=u(\sft_0)=u(0)$
  and $u(\sft_{N,-})=u(\sft_N)=u(T)$)  we get
  \begin{align*}
    \scalarmuco {}u(0,T)&+\sum_{i=1}^{N}
   \Cost\vvmnametil {\sft_i}{u(\sft_{i,-})}{u(\sft_{i})}
    +\Cost\vvmnametil {\sft_i}{u(\sft_{i})}{u(\sft_{i,+})} 
    \le
    \ene {0}{u(0)}-
    \ene {T}{u(T)}+
    \int_0^T \power s{u(s)}\,\dd s.
  \end{align*}
  If $\mathrm J_u$ is finite we get
  \eqref{eq:84-oneside} choosing $N=\#(\mathrm J_u)$
  and recalling \eqref{repre-rn} and
  \eqref{eq:37}.
  If $\mathrm J_u$ is infinite, we simply pass to the limit
  as $N\up+\infty$.    We leave to the reader the obvious modifications in the
    case $\mathrm J_u\cap \{0,T\}\neq \emptyset$.
\end{proof}
The jump conditions
\eqref{eq:67} should be compared with the general estimate 
\eqref{eq:36}, that at every jump point $t\in \mathrm
J_w$ of an arbitrary curve $w\in \BV([0,T];D_E,\Psiz)$
rephrases as 
\begin{equation}
  \label{eq:35first}
  \Big|\ene t{w(t_+)}-\ene t{w(t)}\Big|\le \Delta_{\mathfrak
    f_t}(w(t),w(t_+)), \quad
  \Big|\ene t{w(t)}-\ene t{w(t_-)}\Big|\le \Delta_{\mathfrak f_t}(w(t_-),w(t)).
\end{equation}
We extend now the differential characterization of
$\BV$ solutions in \cite[Thm.\,4.3]{MRS10} to
the present setting.
\begin{theorem}[Differential characterization of $\BV$ solutions]
  \label{prop:diff-charact-bv}
  Let $u\in \BV([0,T];V)$  with distributional
  derivative decomposed as in Remark \ref{re:BVV}.
Then $u$ is a $\BV$ solution of the \ris\  $\RIS$
if and only if it satisfies the doubly nonlinear differential
inclusion in the $\BV$ sense
 \begin{equation}
    \label{eq:66bis-BV}
    \tag{DN$_{\BV}$}
    \partial\Diss{\Bo}{}\Big(\frac{\dd u_\dd'}{\dd \lambda}(t)\Big)
    +\frsub\ene t{u(t)}\ni 0\quad  \text{for  $\lambda$-a.a.\ $t\in (0,T)$}\quad\text{with } 
    \lambda= \|u_{\rm C}'\|+\Leb 1,
  \end{equation}
  and the jump conditions \eqref{eq:67}.
  In particular \eqref{eq:66bis-BV} yields the pointwise inclusion
  \begin{equation}
    \label{eq:20}
        \tag{DN$_{\mathscr L}$}
    \partial\Diss{\Bo}{}\big(\dot u(t)\big)
    +\frsub\ene t{u(t)}\ni 0\quad \text{for $\Leb 1$-a.a.\ $t\in (0,T)$}.
  \end{equation}
\end{theorem}
\begin{proof}
  We briefly recall the argument presented in \cite[Prop.\,2.7,
  Thm.\,4.3]{MRS10}.  Let us first notice that \eqref{eq:66bis-BV}
  yields the local stability condition, since the support of $\lambda$
  is the full interval $[0,T]$ and $K^*$ contains the range of
  $\partial\Psiz$. By the distributional chain rule \eqref{eq:86} we
  get
  \begin{displaymath}
    e_\dd'= -\Psiz(\nn)\|u_\dd'\|+\power\cdot u\Leb 1\topref{eq:15}=
    -\mu_\dd +\power\cdot u\Leb 1.
\end{displaymath}
Combining this information with the jump conditions \eqref{eq:67} and
recalling formula \eqref{eq:81bis} for $\pVarname{\vvmnametil}$ we get
\eqref{eq:84}.

Conversely, if $u$ is a solution then \eqref{eq:21bis} yields
\begin{displaymath}
    e_\dd'+\Psiz(\nn)\|u_\dd'\|-\power \cdot u\Leb 1\le 0\quad\text{in }\scrD'(0,T).
\end{displaymath}
Recalling \eqref{eq:86} we thus obtain for $-\xi\in \partial\ene
t{u(t)}\cap K^* $
\begin{displaymath}
    \Big(\langle -\xi,\nn\rangle+\Psiz(\nn)\Big)\|u_\dd'\|\le 0
    \quad\text{in }\scrD'(0,T), 
\end{displaymath}
which yields the inclusion \eqref{eq:66bis-BV} $\|u_\dd'\|$-a.e.\ in
$(0,T)$, and in particular $\Leb 1$-a.e.~in the set $\|\dot u\|>0$.
For $\Leb 1$-a.a.~points of the set $\|\dot u\|=0$ the local stability
condition still provides \eqref{eq:66bis-BV}.
\end{proof}

\subsection{Optimal jump transitions}
\label{ss:opt-trans}

Thanks to the jump conditions given by \eqref{eq:67}, we can give a
finer description of the behavior of $\BV$ solutions along jumps.  The
crucial notion is provided by the following definition.

\begin{definition}[Optimal transitions]
  \label{prop:opt-trans}
Let
$t \in
[0,T]$ and $u_-,\, u_+ \in \domainenergy$ with
\begin{equation}
\label{basic-hyp-upm}  K^* +\frsub \ene t{u_-}\ni 0,
\qquad  K^* +\frsub \ene t{u_+}\ni 0.  
\end{equation}
We say that an admissible curve $\teta \in \calT_t(u_-,u_+)$
is an
$\vvmnametil_t$-\emph{optimal transition} between
$u_-$ and $u_+$ 
if
\begin{equation}
    \label{eq:66-ojt}
    \begin{gathered}
    \ene t{u_-}-\ene t{u_+}=\Cost{\vvmnametil}t{u_-}{u_+}
    =\vvmname_t[\vartheta,\vartheta'](r)> 0
    \quad \text{for a.a.\ }r\in
    (0,1),
    \end{gathered}
  \end{equation}
  and we denote by $\calO_t(u_-,u_+)$ the (possibly empty) collection
of such optimal transitions.

We say that $\vartheta$ is of
\begin{align}
    \text{\emph{sliding type}, if}&
    \quad\vvmV t{\teta(r)}{}=0 \quad\text{for every $r\in [r_0,r_1]$,}
  \label{eq:93-ojt}\\
    \text{\emph{viscous type}, if}&\quad
     \vvmV t{\teta(r)}{}>0\quad\text{for every
      $r\in (r_0,r_1)$.}
  \label{eq:94-ojt}
\end{align}
\end{definition}
The main interest of optimal transitions derives from the next result,
whose proof follows immediately from Theorem \ref{thm:fcost} by a simple rescaling argument.
\begin{proposition}
\label{prop:ojt2}
If $u
\in \BV ([0,T];\V,\Psiz)$ is a $\BV$ solution to the rate-independent
system $\RIS$, then
for every $t \in \mathrm{J}_u$ there exists an
$\vvmnametil_t$-optimal transition $\teta^t \in
\calO_t(u(t_-),u(t_+))$ such that $u(t)=\teta^t (r)$ for some $r \in
[0,1]$.
\end{proposition}
We now provide a characterization of \emph{sliding}
and \emph{viscous} optimal transitions in terms of doubly nonlinear
differential inclusions. 
\begin{proposition}[The structure of optimal transitions]
\label{prop:ojt1}
Let  $t
\in [0,T]$ and $u_-,\, u_+ \in \domainenergy$ fulfilling
\eqref{basic-hyp-upm} be given and let
$\teta \in \calT_t(u_-,u_+)$ be an admissible
transition curve with constant normalized velocity
$\frf_t[\vartheta,\vartheta'](r)\equiv c>0$ $\foraa r\in (0,1)$.
 Then
\begin{enumerate}[(1)]\itemsep0.4em
\item $\teta$ is an optimal transition of sliding type if and only
if
it satisfies
\begin{gather}
\label{charact-slid}
\exists \xi(r)\in -\frsub \ene t {\teta(r)})\cap K^*
\quad \forevery r \in [0,1],\\
\label{eq:34-a}
\frac \dd{\dd r}\ene t{\vartheta(r)}+\scalardens\Psiz \vartheta=0
  \quad
  \foraa\, r \in (0,1).
\end{gather}
In particular, when $\vartheta$ is differentiable $\Leb 1$-a.e.\ in
$(0,1)$,
\eqref{charact-slid} and \eqref{eq:34-a} are equivalent to
\begin{equation}
  \label{eq:35}
  \partial\Diss{\Bo}{}(\dot{\teta}(r))+ \frsub \ene t {\teta(r)} \ni
  0 \quad
  \foraa\, r \in (0,1).
\end{equation}
\item $\teta$ is an optimal transition of viscous type if and only if
  it is differentiable $\Leb 1$-a.e.~in $(0,1)$ and there exists maps
  $\xi\in L^1(0,1;V^*),$ and $\eps:(0,1)\to (0,+\infty)$ such that
\begin{equation}
  \label{eq:96}
  \xi(r)\in \Big(\partial\Diss{\Bo}{}(\dot{\teta}(r))+
\partial\Diss{\V}{}(\eps(r)\dot{\teta}(r))\Big)\cap \Big(-\frsub \ene t {\teta(r)} \Big) \quad
\foraa\, r \in (0,1);
\end{equation}
in particular,
\begin{align} \nonumber
& \eps(r)=\Lambda_t(\vartheta(r);\dot\vartheta(r))\quad
  \foraa\, r\in (0,1),
\\ 
&\text{where }
\label{not-contact-viscosities}
\Lambda_t(\vartheta;v):=(F^*)'(\fre_t(\vartheta))/F(\|v\|)\quad \vartheta\in D,\
v\in V\setminus\{0\}.
\end{align}
Equivalently, there exists
an absolutely continuous, surjective time rescaling
$\mathsf r:(s_0,s_1)\to (0,1)$,
with $-\infty\le s_0<s_1\le +\infty$ and $\dot {\mathsf r}(s)>0$ for
$\Leb 1$-a.a.~$s\in (s_0,s_1)$, such that the rescaled transition
$\theta(s):=\vartheta(\mathsf r(s))$ satisfies the viscous
differential inclusion
\begin{equation}
\label{eq:32}
\partial\Diss{}{}(\dot{\theta}(s))+\partial\Diss{\V}{}(\dot{\theta}(s))+
\frsub \ene t {\theta(s)} \ni 0 \quad
\foraa\, s \in (s_0,s_1).
\end{equation}
\item  If
  $\vartheta$ is an optimal transition, then
it can be decomposed in a canonical way
    into an (at most) countable collection of
    optimal \emph{sliding and viscous} transitions. Namely,
     there exist (uniquely determined) disjoint open intervals
    $(S_j)_{j\in \sigma}$ and $(V_k)_{k\in \upsilon}$ of $(0,1)$, with
    $\sigma,\upsilon\subset \N$, such that
    $(0,1)\subset \big(\cup_{j\in\sigma}
    S_j)\cup\overline{\big(\cup_{k\in \upsilon} V_k\big)}$
    and
     \[
      \vartheta\Restr {S_j}\quad\text{is of sliding type,}\quad
      \vartheta\Restr {V_k}\quad\text{is of viscous type.}
      \]
\end{enumerate}
\end{proposition}
\begin{proof}
  (1) It is easy to check that if an admissible transition $\vartheta$
  satisfies \eqref{charact-slid}--\eqref{eq:34-a} then $\vartheta$ is
  an optimal transition of sliding type. Indeed, by the chain rule of
  Theorem \ref{thm-from-mrs12}
   $r\mapsto \ene t{\vartheta(r)}$ is absolutely continuous, and integrating
   \eqref{eq:34-a} we get \eqref{eq:66-ojt}.
   The converse implication is even easier by combining the chain rule
   along $\vartheta$, the fact that $\frf_t[\vartheta,\vartheta']=\Psiz[\vartheta']$,
   and \eqref{eq:66-ojt}.

   (2) Similarly, if $\vartheta,\eps,\xi$ satisfy \eqref{eq:96},
   the chain rule yields
   \begin{align*}
     \frac\dd{\dd r}\ene t{\vartheta(r)}&=
     - \langle \xi(r),\dot\vartheta(r)\rangle =
     -\Psiz_{\eps(r)}(\eps(r)\dot\vartheta(r))-\Psiz_{\eps(r)}^*(\vartheta(r))
     \\&\le -\Psiz(\dot\vartheta(r))
     -\frac{1}{\eps(r)}F(\eps(r)\|\dot\vartheta(r)\|)-\frac 1{\eps(r)}
     F^*(\fre_t(\vartheta(r)))
     \\&
     \le -\Psiz(\dot\vartheta(r))-\fre_t(\vartheta(r))\|\dot\vartheta(r)\|=
     -\frf_t(\vartheta(r),\dot\vartheta(r))=-c<0.
   \end{align*}
   Integrating in time we get one inequality of  \eqref{eq:66-ojt}; the converse one
   is always true. Then,   all the above inequalities are in fact equalities:
   in particular $\fre_t(\vartheta(r))>0$ in $(0,1)$, since $F(r)>0$ if $r>0$ by
   \eqref{e:2.1}. We then conclude that $\teta$ is an optimal transition of viscous type. 

   The converse implication follows from the fact that
   \begin{displaymath}
     \fre_t(\vartheta)\|\dot\vartheta\|=\frac{1}{\eps }F(\eps\|\dot\vartheta\|)+\frac 1{\eps}
     F^*(\fre_t(\vartheta))\quad\text{if }\eps=\Lambda_t(\vartheta,\dot \vartheta).
   \end{displaymath}
   Observing that $\dot\vartheta$ is locally bounded in $(0,1)$ so
   that $r\mapsto 1/\eps(r)$ is also locally bounded,
   in order to get \eqref{eq:32} we simply operate the absolutely continuous time rescaling
   \begin{displaymath}
     \sfs(r):=\int_{1/2}^r \eps^{-1}(r)\,\dd r,\quad \sfr:=\sfs^{-1},\quad
     \theta(s):=\vartheta(\sfr(s)),\quad
     \dot\theta(s)=\eps(\sfr(s))\dot\vartheta(\sfr(s)).
   \end{displaymath}
   (3) We can simply split the parameter interval $(0,1)$
   into the open sets $V:=\{r:\fre_t(\vartheta(r))>0\}$, $S:=[0,1]\setminus \overline V$,
   and then we consider their connected components.
\end{proof}
As a last result, we show that optimal transitions capture
the asymptotic profile of rescaled solutions to \eqref{viscous-dne}
around a jump point.
\begin{proposition}[Asymptotic profiles and optimal transitions]
  \label{prop:profile}
  Let $\eps_k\down0$ and let $(u_{\eps_k},\xi_{\eps_k})$ be a sequence of
  solutions to the viscous doubly nonlinear equation \eqref{xi-selection}, so that
  $u_{\eps_k}$ converge to a $\BV$ solution $u$ of the \ris\ $\RIS$
  as $k\to\infty$ according to Theorem \ref{th:1}.
  For every $t\in \mathrm J_u$ let
  $\alpha_k<t<\beta_k$ be two sequences such that
  \begin{equation}
    \label{eq:110}
    \alpha_k     \up t,\quad \beta_k\down t,\quad
    \lim_{k\to\infty}u_{\eps_k}(\alpha_k)=u(t_-),\quad
    \lim_{k\to\infty}u_{\eps_k}(\beta_k)=u(t_+).
  \end{equation}
  Then
  \begin{equation}
    \label{eq:111}
    \lim_{k\to\infty}
    \int_{\alpha_k}^{\beta_k}
    \Big(\Psi_{\eps_k}(\dot{u}_{\eps_k}){+}\Psi_{\eps_k}^*(-\xi_{\eps_k}) \Big)\,\dd r
    =\Delta_{\frf_t}(u(t_-),u(t_+)),
  \end{equation}
  and there exist a further subsequence
  (not relabeled), increasing and
  surjective time rescalings $\sft_k\in \AC([0,1]; [\alpha_k,\beta_k])$,
  and an optimal transition $\vartheta\in \calO_t(u(t_-),u(t_+))$ such
  that
  \begin{gather}
    \label{eq:109bis}
    \lim_{k\to\infty}u_{\eps_k}\circ\sft_k= \vartheta\quad
    \text{strongly in $V$, uniformly on }[0,1].
  \end{gather}
\end{proposition}
\begin{proof}
Estimate
  \eqref{eq:64bis} from  Theorem \ref{thm:fcost}
  provides the inequality
  \begin{displaymath}
    \liminf_{k\to\infty}
    \int_{\alpha_k}^{\beta_k}
    \Big(\Psi_{\eps_k}(\dot{u}_{\eps_k}){+}\Psi_{\eps_k}^*(\xi_{\eps_k}) \Big)\,\dd r
    \ge\Delta_{\frf_t}(u(t_-),u(t_+)).
  \end{displaymath}
  On the other hand, applying \eqref{e:conv3} to each interval
  $[\alpha_h,\beta_h]$ we obviously get
  \begin{displaymath}
    \limsup_{k\to\infty}
    \int_{\alpha_k}^{\beta_k}
    \Big(\Psi_{\eps_k}(\dot{u}_{\eps_k}){+}\Psi_{\eps_k}^*(\xi_{\eps_k})
    \Big)\,\dd r
    \le \pVar\frf u{\alpha_h}{\beta_h}\quad\forevery h\in \N.
  \end{displaymath}
  Passing to the limit as $h\up\infty$
  we obtain \eqref{eq:111}.
  We then apply assertion (F3) of Theorem \ref{thm:fcost}
  to find an admissible transition
  $\vartheta\in \calT_t(u(t_-),u(t_+))$ and rescalings $\sft_k$ such that
  \eqref{eq:109} holds. Relation \eqref{eq:111} shows that
  $\vartheta$ is optimal.
\end{proof}

\subsection{$V$-parameterizable solutions}
\label{ss:Vparam}

In this section we will focus on
a more restrictive notion of solution,
exhibiting better regularity properties:
they belong to $\BV([0,T];\V) $ and at all jump points the
left and the right limits can be connected by an optimal transition
with finite
\emph{$\V$-length}. Moreover, we will require that
the total $\V$-length of the connecting
paths is finite.


\begin{definition}[$\V$-parameterizable $\BV$ solutions]
\label{def:Vparam-solution}
A balanced viscosity solution $u$ of the \ris\ $\RIS$ (in the sense of
Definition \ref{def:BV-solution}) is called
\emph{$\V$-parameterizable} if $u \in \BV ([0,T];\V)$ and
\begin{equation}
  \label{special-jump-cond}
\begin{array}{lll}
 & \displaystyle \text{ i)} \quad & \displaystyle
\forall\, t \in \mathrm{J}_u \quad
\exists\, \teta^t \in \mathcal{O}_t (u(t_-),u(t_+))\cap \AC ([0,1];\V),
\\
& \displaystyle \text{ ii)} \quad & \displaystyle
 \sum_{t \in
\mathrm{J}_u} \int_{0}^1
 \Vnorm{\dot{\teta}^t(r)} \,\dd
r<\infty.
\end{array}
\end{equation}
\end{definition}

The notion of \emph{$\V$-parameterizable} $\BV$ solution
slightly differs from the concept of
 \emph{connectable $\BV$ solution} introduced in
\cite[Def.\,4.21]{Miel08?DEMF}, which only requires condition $i)$.

As one can expect, a limit curve of solutions to \eqref{viscous-dne} satisfying
a uniform $\BV([0,T];V)$-bound is a $V$-parameterizable solution.

\begin{theorem}
  \label{thm:obvious}
  Let $(u_\eps)_{\eps>0}$ be a family of solutions to \eqref{viscous-dne} satisfying
  \eqref{conve-initi-data} at $t=0$ and the uniform bound
  \begin{equation}
    \label{eq:98}
     \exists\, C>0 \ \ \forall\, \eps>0\, : \quad 
    \Var{}
    {u_\eps} 0T\le C. 
  \end{equation}
  Then any limit curve as in Theorem \ref{th:1} is a $V$-pa\-ra\-me\-te\-ri\-za\-ble
  $\BV$ solution to the \ris\  $\RIS$.

  \noindent Similarly, let $(\Utaue n)_{\taue}$ be a family of discrete solutions
  to \eqref{eq:58}, satisfying \eqref{eq:114} and \eqref{eq:103-k}.
  If
  \begin{equation}
    \label{eq:75}
  \exists\, C>0 \ \ \forall\, \tau, \, \eps>0 \,:\quad \Var{}{\mathrm U_{\taue}}0T 
=\sum_{n=1}^{N_\tau} \|\mathrm U^n_\taue-\mathrm U^{n-1}_\taue\|\le C,
\end{equation}
then any  accumulation point  of the  piecewise affine interpolants 
$\pwL\UU\taue$ as in Theorem \ref{th:2} is a $V$-parameterizable
solution.
\end{theorem}
\begin{proof}
  The proofs of the two statements are very similar, thus 
  we only prove the first one.

  Since the total variation functional is lower semicontinuous
  with respect to pointwise
  convergence, any limit curve $u$ obtained as in Theorem \ref{th:1}
  clearly belongs to $\BV([0,T];V)$.

  In order to check $i)$ of \eqref{special-jump-cond} we apply
  Proposition \ref{prop:profile} and we find a sequence
  of rescalings $\sft_k:[0,1]\to [\alpha^t_k,\beta^t_k]$
  (we explicitly indidate the dependence of
  the time intervals $[\alpha_k,\beta_k]$ on $t$)
  and an
  optimal transition $\vartheta^t\in \calO_t(u(t_-),u(t_+)$ with
  \eqref{eq:110} and \eqref{eq:109bis}. This shows that
  \begin{equation}
    \label{eq:113}
    \Var{}
    {\vartheta^t}01\le
    \liminf_{k\to\infty}
    \Var{}
    {u_{\eps_k}}{\alpha^t_k}{\beta^t_k}<\infty,
  \end{equation}
  so that $\vartheta^t\in \BV(0,1;V)$.
  Since $\vartheta^t$ is also continuous, up to a further time rescaling we
  can obtain an optimal transition absolutely continuous in $V$.

  A slight refinement of the above argument also provides $ii)$:
  we consider an arbitrary finite collection of points $t_1,t_2,\ldots
  t_h   \subset \mathrm J_u$ and we choose
  a common subsequence $u_{\eps_k}$ satisfying
  \eqref{eq:110} in each interval. For sufficiently big
  $k$ so that
  the intervals $[\alpha_k^{t_j},\beta_k^{t_j}]$ are disjoint,
  \eqref{eq:113} yields
  \begin{displaymath}
    \sum_{j=1}^h
    \Var{}
    {\vartheta^{t_j}}01\le
    \liminf_{k\to\infty}\sum_{j=1}^h
    \Var{}
    {u_{\eps_k}}{\alpha^{t_j}_k}{\beta^{t_j}_k}
    \le \liminf_{k\to\infty}
    \Var{}
    {u_{\eps_k}}0T
    \topref{eq:98}\le C.
  \end{displaymath}
  Since the number $h$ of jump points is arbitrary, we obtain \emph{ii)}.
\end{proof}
The next results show that one
 can actually prove \eqref{eq:98} and \eqref{eq:75}
for the particular choice
\begin{equation}
\label{particular-dissipations}
\Diss\Vanach{}(v) = \frac12\Vnorm{v}^2,\quad F(r):=\frac 12 r^2,
\end{equation}
under slightly more restrictive assumptions on the energy functional
and on the initial data: besides the usual
\eqref{e:2.1}--\eqref{eq:33} and \eqref{Ezero}--\eqref{hyp:en3}, we
will also assume that for every $E>0$ there exist constant
$\alpha_E,\Lambda_E, L_E>0$ such that the energy functional satisfies
the G\r{a}rding-like subdifferentiability inequality
\begin{equation}
\label{subdiff-charact}
\ene tv - \ene tu \geq \pairing{}{}{\xi}{v-u} +{\alpha_E}
\Vnorm{v-u}^2 - {\Lambda_E} \Dnorm{v-u}\|v-u\|
\quad \text{if }u,v\in D_E,\ \xi \in \frsub \ene tu.
\end{equation}
We will also require that the power functional is uniformly Lipschitz in $D_E$, viz.
\begin{equation}
  \label{eq:77}
  \left|\power  tu -\power  tv \right|\leq L_E\Vnorm{u-v}
  \quad \text{if }t\in [0,T], \ u,v\in D_E.
\end{equation}
Then, we have the following result. 
\begin{theorem}[A priori estimates for discrete Minimizing Movements]
\label{th:3-discrete}
Assume that \eqref{particular-dissipations}--\eqref{eq:77} hold.
Then any family of solutions $(\mathrm U^n_{\taue})$ of
\eqref{eq:58} fulfilling, for some constants $E_0,Q>0$,
\begin{gather}
  \label{eq:116}
  \Psiz(\Utaue 0)+\ene0 {\Utaue 0}\le E_0,\quad \tau\le Q\eps, 
  \quad  K^* + \partial\ene 0{\Utaue 0}\ni 0, 
\end{gather}
satisfies estimates \eqref{eq:75}.  In particular, if \eqref{eq:114},
\eqref{eq:103-k} and \eqref{eq:116} hold, any curve $u$ obtained as
limit of the piecewise affine interpolants $\mathrm U_\taue$ (cf.\ 
Theorem \ref{th:2}) is a $V$-parameterizable solution.
\end{theorem}

The proof will be given in Section \ref{ss:5.2}. A similar priori
estimate in the form $\int_0^T \| \dot{u}_\eps (t)\| \, \dd t \leq C$
was derived in \cite{Mielke-Zelik} for semi- and quasilinear partial
differential equations with smooth nonlinearities. There Galerkin
approximation and differentiation in time is used.  Like in the
present case, where we have to confine ourselves to Minimizing
Movement solutions (cf.\ Corollary \ref{coro:3-continuous} below), in
\cite{Mielke-Zelik} the a priori estimate in $\BV ([0,T];\V)$ can only
be shown for a suitable subclass of solutions to \eqref{viscous-dne},
cf.\ \cite[Def.\,4.3]{Mielke-Zelik}.  This establishes an interesting
parallel between our Minimizing Movement approach, and the one in
\cite{Mielke-Zelik}.

\begin{corollary}[A priori estimate for Minimizing Movement solutions]
\label{coro:3-continuous}
Assume that \eqref{particular-dissipations}--\eqref{eq:77}
hold.  Then every family $(u_\eps)_\eps \subset \AC ([0,T];\V)$ of
Minimizing Movement solutions to \eqref{viscous-dne}, fulfilling
\begin{equation}
  \label{other-assum-data}
  u_\eps(0)\to u_0\quad\text{in $V$},\quad
  \ene 0{u_\eps(0)}\to \ene 0{u_0},\quad
 K^* +\partial\ene 0{u_\eps(0)}\ni 0, 
\end{equation}
satisfies estimate \eqref{eq:98}.  Any limit $u$ is a
$V$-parameterizable solution to the \ris\ \RIS.
\end{corollary}
\begin{proof}
  Choose $\Utaue 0=u_\eps(0)$ and apply Theorem \ref{thm:ref0},
  passing to the limit in estimate \eqref{eq:75}.
\end{proof}

 The following result is an immediate consequence of Corollary
\ref{coro:3-continuous} or Theorem \ref{th:3-discrete}. 

\begin{corollary}[Existence of $V$-parameterizable $\BV$ solutions]
  \label{cor:3-limit}
  If \eqref{particular-dissipations}--\eqref{eq:77} hold,
  then for every $u_0\in D$ with
 $K^* + \partial\ene 0{u_0}\ni 0$ 
  there exists a $V$-parameterizable $\BV$ solution to
  the \ris\  $\RIS$ starting form $u_0$.
\end{corollary}
\normalcolor
Notice that the subdifferentiability 
condition \eqref{subdiff-charact} implies
\eqref{hyp:en-subdif} as well as 
\begin{equation}
\label{eq:89}
\begin{gathered}
  \pairing{}{}{\eta-\xi}{v-u} \geq 2\alpha_E \Vnorm{v-u}^2 -
  2\Lambda_E \Dnorm{v-u}\,\|v-u\| -L_E |t-s|\,\|v-u\|\\
  \text{whenever $\eta
    \in \frsub \ene tv$, $\xi \in \frsub \ene su$,}\ u,v \in
  \domainenergy_E,\ s,t\in [0,T].
\end{gathered}
\end{equation}
To check \eqref{eq:89},  
it is sufficient to write \eqref{subdiff-charact} for $u$ and $v$ at
time $s, t$ respectively. Adding the two inequalities and
using \eqref{eq:77} we get the bound (assuming $s<t$)
\begin{displaymath}
  \ene tv-\ene sv+\ene su-\ene tu\le
  \int_s^t\big( \power rv-\power r u\big)\,\dd r\le L_E\,(t-s)\|u-v\|.
\end{displaymath}
Observe that
in \eqref{subdiff-charact}, as in \eqref{hyp:en-subdif},
we  allow for a \emph{negative} modulus of convexity
in the $\Psiz$-term, provided  that it is possible to gain
an even small
\emph{positive} modulus of subdifferentiability
in the stronger $\V$-norm. This is akin to the G\r arding inequality for
elliptic operators.

The next result provides a useful criterium on the energy functional
$\cE$ to establish the sub\-dif\-fer\-entiability condition
\eqref{subdiff-charact}. It is a sort of \emph{(generalized)
  $\lambda$-convexity} condition, involving two norms.  Notice that
both \eqref{eq:18-new} and \eqref{subdiff-charact} are required to
hold on \emph{sublevels} of $\cE$, only.

\begin{lemma}
\label{conseq-convex}
Suppose that for every $E>0$ there exist constant
$\alpha_E,\Lambda_E>0$ such that the energy functional $\cE_t: \V \to
(-\infty,+\infty]$ satisfies
\begin{equation}
  \label{eq:18-new}
\begin{aligned}
  \ene t{(1-\theta)u+\theta v} \leq (1-\theta)\ene tu + \theta \ene t
  v -\theta(1-\theta) 
\left(\alpha_E \Vnorm{u-v}^2 -\Lambda_E \Dnorm{u-v}\|u-v\|\right)
\end{aligned}
\end{equation}
for every $u,v\in D_E$ and $\theta\in [0,1]$.
Then its Fr\'echet subdifferential $\frsub \cE_t :\V
\rightrightarrows \V^*$ satisfies \eqref{subdiff-charact}.
\end{lemma}
\begin{proof}
  For $\xi $ lying in the Fr\'echet subdifferential $\frsub \ene tu$
  there holds for every $v,\, u \in D_E$ and $\theta\down0$
\begin{align*}
\langle \xi, \theta(v-u) \rangle + o(\theta \Vnorm{v-u})  & \leq \ene
t{(1-\theta) u +\theta v}-\ene tu 
\\
& \leq \theta (\ene tv -\ene tu) -  \theta (1-\theta) (\alpha_E \Vnorm{v-u}^2-
\Lambda_E \Dnorm{v-u}\|v-u\|).
\end{align*}
Dividing both sides of the inequality by $\theta$, the limit
$\theta \down 0$ yields the desired estimate
\eqref{subdiff-charact}. 
\end{proof}

\section{Parameterized solutions}
\label{s:5}
\subsection{Vanishing-viscosity  analysis, parameterized curves
and solutions}\label{ss:5.1} 

Under the working assumptions of \S \ref{ss:2.1} (in particular,
\eqref{e:2.1}--\eqref{def-psiV} and
\eqref{Ezero}--\eqref{hyp:en-subdif}), in this section we will present
a different approach to the vanishing-viscosity analysis of
\eqref{viscous-dne}, which goes back to \cite{ef-mie06} and was
further developed in \cite{MRS09,MRS10}. The main idea is to rescale
time in \eqref{viscous-dne} and study the limiting behavior as $\eps
\down 0$ of the \emph{rescaled} viscous solutions. This naturally
leads to the notion of parameterized solution in Definition
\ref{def:2.1-parasol}: it is a \emph{space-time parameterized curve},
along which the energy $\cE$ fulfills a ``parameterized'' version of
the energy-dissipation identity \eqref{eq:52bis}.  At the end of this
section, we will also discuss the parameterized counterpart to
$\V$-parameterizable $\BV$ solutions.  Let us emphasize that,
while parameterized solutions were developed in
\cite{ef-mie06,Mielke-Zelik} in their own right, we use them
mainly to obtain the desired results for $\BV $ solutions.

\subsubsection*{\bfseries  Vanishing-viscosity  analysis}
Let $(u_\eps)_\eps$  be a family of solutions to the
``viscous'' doubly nonlinear equation
\eqref{viscous-dne}. It follows from the energy identity
\eqref{eq:52bis}
and from the variational characterization of $\frf$
\eqref{def-bipot-sez2}--\eqref{def-vvm} 
that
\begin{equation}
\label{starting-point}
 \int_{s}^{t}
    \vvm{r}{u_\eps(r)}{\dot{u}_\eps(r)}\,\dd r + \ene {t}{u(t)} \leq \ene
    {s}{u(s)} + \int_s^t \power r {u(r)} \, \dd r \quad
    \text{for all } 0 \leq s \leq t \leq T,
\end{equation}
whence, relying on the power control \eqref{hyp:en3}, we deduce that
there exists a constant $C>0$ such that
\begin{equation}
\label{dissipation-bound}
\mathsf S_\eps:= T
 + \int_0^T \vvm{r}{u_\eps(r)}{\dot{u}_\eps(r)} \,\mathrm{d}
 r   \leq C\quad\forevery \eps>0.
\end{equation}
We rescale the functions $u_\eps$ by the \emph{energy-dissipation
  arclength} $\mathsf s_\eps:[0,T]\to[0,\mathsf S_\eps]$ of the curve
$u_\eps$, defined by
\begin{equation}
\label{energy-dissipation-arlength} {\mathsf s}_\eps(t):=t
 + \int_0^t \vvm{r}{u_\eps(r)}{\dot{u}_\eps(r)} \,\mathrm{d} r.
\end{equation}
Hence,  we introduce  the rescaled functions $({\mathsf
t}_\eps,{\mathsf u}_\eps): [0,{\mathsf S}_\eps] \to [0,T]\times \V$
\begin{equation}
\label{e:resc2}
\begin{aligned}
  {\mathsf t}_\eps (s)&:={\mathsf s}_\eps^{-1}(s)\,, &{\mathsf
    u}_\eps(s)&:= u_\eps ({\mathsf t}_\eps (s)).
    \end{aligned}
    \end{equation}
    We write the ``rescaled energy identity'' fulfilled by the triple
    $({\mathsf t}_\eps,{\mathsf u}_\eps)$ by means of the space-time
    Finsler dissipation functionals $\vvmfullepsname,\frG_\eps:
    [0,T]\times \domainenergy \times [0,+\infty) \times \V \to
    [0,+\infty)$ defined by
\begin{equation}
  \label{def-meps}
  \begin{aligned}
    \vvmfulleps{\alpha}{\parat}{\parau}{\mathsf{v}}{\mathsf{p}}&:=
    \Psiz(\sfv)+\frG_\eps(\sft,\sfu;\alpha,\sfv)-\alpha \power \sft\sfu
    \qquad \text{with} 
    \\
    \frG_\eps(\sft,\sfu;\alpha,\sfv)&:=
\begin{cases}
  \frac\alpha\eps\Phi(\frac \eps\alpha\mathsf v)
 + \frac\alpha \eps F^*(\fre_t(u))
&\text{for }\alpha >0,\\
\infty&\text{for }\alpha =0,
\end{cases}
\end{aligned}
\end{equation}
where we combined \eqref{eq:3} for $\Psi_\eps^*$, yielding
\eqref{eq:7} for $\mathfrak{f}_t$, and the monotonicity of $F^*$ to
find
\begin{displaymath}
  \inf_{\xi\in -\partial\ene tu}\Psi_\eps^*(\xi)=
  \inf_{\genfrac{}{}{0pt}{1}{\xi\in -\partial\ene tu}{z\in
    K^*}}\frac 1\eps F^*(\|\xi-z\|_*)=\frac 1\eps F^*(\fre_t(u)).
\end{displaymath}
Then, the energy identity \eqref{eq:52bis} yields
for every $0\le s_1<s_2\le \mathsf S_\eps$
\begin{equation}
\label{resc-enid-eps}
\begin{aligned}
  \int_{s_1}^{s_2} \vvmfulleps{\dot{{\mathsf t}}_\eps(s)}{\parat_\eps(s)}{\parau_\eps(s)}{\dot{{\mathsf u}}_\eps(s)}{{\mathsf p}_\eps(s)}
 \,\mathrm{d} s
 &  +\ene {{\mathsf t}_\eps(s_2)}{{\mathsf u}_\eps(s_2)}
   =  \ene {{\mathsf t}_\eps(s_1)}{{\mathsf u}_\eps(s_1)} ,
 \end{aligned}
\end{equation}
and, on account of our choice \eqref{energy-dissipation-arlength} of
the reparameterization, we have the \emph{normalization
condition}
\begin{equation}
\label{resc-par-1}
\dot{{\mathsf t}}_\eps(s) +
\vvm{{\mathsf
t}_\eps(s)}{\parau_\eps(s)}{\dot{{\mathsf
u}}_\eps(s)}
  \equiv 1\quad
  \text{for a.a.}\, s \in (0,{\mathsf S}_\eps)\,.
\end{equation}

From \eqref{resc-enid-eps} it is possible to deduce a priori estimates
on the family $(\parat_\eps,\parau_\eps)_{\eps}$, thus proving that,
up to a subsequence, the functions $(\parat_\eps,\parau_\eps)$
converge in a suitable sense to a pair $(\parat,\parau): [0,S] \to
[0,T]\times \V$ (see Thm.\ \ref{thm-van-param} for a precise
statement). \normalcolor In view of the forthcoming lower
semicontinuity Proposition \ref{le:compactness}, we expect that taking
the limit $\eps \to 0$ in \eqref{resc-enid-eps} leads to the energy
estimate
\begin{equation}
  \label{eq:16}
  \int_{s_1}^{s_2}
\vvmfull{\dot{\parat}(s)}{\parat(s)}{\parau(s)}{\dot{\parau}(s)}{\mathsf{p}(s)}\,
\dd s 
  +\ene{{\mathsf t}(s_2)}{{\mathsf u}(s_2)}  \leq
  \ene{{\mathsf t}(s_1)}{{\mathsf u}(s_1)}\quad
  \text{for all } {0}\le s_1\le s_2\le \mathsf{S}.
\end{equation}
The functional ${\vvmfullname} :  [0,T]\times   \domainenergy  \times [0,+\infty)
\times \V\to [0,+\infty]$ is defined by
\begin{equation}
\label{vvmfullname-def}
\begin{aligned}
 \vvmfull{\alpha}{\parat}{\parau}{\mathsf{v}}{\mathsf{p}}&:= 
 \Diss{\Bo}{}(\mathsf{v})+\vvmfullV{\alpha}{\parat}{\parau}{\mathsf{v}}{\mathsf{p}} 
 - \alpha\power \parat\parau \quad \text{with }\\
 \vvmfullV{\alpha}{\parat}{\parau}{\mathsf{v}}{\mathsf{p}}&:=
 \frk_\parat(\parau)\alpha + \vvmV{\parat}{\parau}{\mathsf{v}}
=\left\{
  \begin{array}{ll}
  \frk_\parat(\parau) &
      \text{if }\alpha>0 ,\\
      \vvmV{\parat}{\parau}{\mathsf{v}}&\text{if }\alpha=0.
  \end{array}
 \right.
\end{aligned}
\end{equation}
Here we have adopted the convention $0\cdot
(+\infty)=0$, and $\frk$ is the indicator function
\begin{equation}
  \label{eq:34}
  \frk_\parat(\parau):=\inf\limits_{\xi\in -\frsub
      \ene\parat\parau}\mathrm{I}_{K^*}(\xi)=
    \mathrm I_{\{0\}}(\fre_t(u))=
    \begin{cases}
      0&\text{if } K^* +  \partial\ene\parat\parau \ni 0, \\
      +\infty&\text{otherwise}.
    \end{cases}
\end{equation}
Hence, it would be natural to take \eqref{eq:16} as definition of
parameterized solution. However, as already mentioned, limit curves
have to be expected in $\AC([0,\mathsf S];V,\Psiz)$,  i.e.\ they
might lose  the differentiability property with respect to
time. Thus, we need to develop a more refined definition.

\subsubsection*{\bfseries Admissible parameterized curves and solutions}
In order to properly formulate \eqref{eq:16} we need to resort to the
metric $\Psiz$-derivative  introduced in the beginning of 
Section \ref{s:appA}.  Based on that definition, we first introduce a
suitable class of \emph{parameterized} curves.

\begin{definition}[Admissible parameterized curves]
\label{def:3.5}
We call a pair $(\parat,\parau): [\mathsf{a},\mathsf{b}]\to
[0,T]\times \V$ an \emph{admissible parameterized curve} 
\begin{enumerate}
 \item  if $\parat$  is nondecreasing and absolutely continuous, $\parau \in
    \AC([\mathsf{a},\mathsf{b}];D_E,\Psiz)$ for some $E>0$,
 \item if $\sfu$ is locally $\Vanach$-absolutely continuous in the
      open set 
   \begin{equation}
      \label{eq:RIF2:6}
      G:=\big\{\:s\in [\mathsf{a},\mathsf{b}]\: :\:
      \fre_{\sft(s)}(\sfu(s))>0 \:\big\} \ = \
      \big\{\: s\in [\mathsf{a},\mathsf{b}] \: :\:
       K^* + \frsub \ene
      {\parat(s)}{\parau(s)} \not\ni 0  \: \big\},
   \end{equation}
   and $\parat$ is constant in each connected component of $G$
   (in particular $\parau$ is differentiable $\Leb 1$-a.e.\ in $G$),
 \item and if we have the estimate
   \begin{equation}
      \label{eq:RIF:4}
   \begin{aligned}
    \int_{\mathsf{a}}^{\mathsf{b}} \scalardens\Psiz{\parau}(s) \,\dd s +
    \int_G \vvmV{\parat(s)}{\parau(s)}{\dot{\parau}(s)}\, \dd s<\infty.
   \end{aligned}
   \end{equation}
\end{enumerate}
For every admissible parameterized curve and all $s\in [\sfa,\sfb]$ we set
\begin{equation}
     \label{eq:50}
 \begin{aligned}
       \frG[\sft,\sfu;\dot\sft,\dot\sfu](s):={}&
       \frk_{\sft(s)}(\sfu(s)) \dot \sft(s)+
       \fre_{\sft(s)}(\sfu(s))\|\dot u(s)\|,\\
       \frF[\sft,\sfu;\sft',\sfu'](s):={}&
       \scalardens\Psiz\sfu(s)+
       \frG[\sft,\sfu;\dot\sft,\dot\sfu](s)
       -\power{\sft(s)}{\sfu(s)}\dot \sft(s),
 \end{aligned}
\end{equation}
where, with a slight abuse of notation, we adopted the convention to set
\begin{equation}
      \label{eq:51}
      \vvmV{\parat(s)}{\parau(s)}{\dot{\parau}(s)}
        \equiv 0\quad\text{if }s\not\in G.
\end{equation}
By $\adm{\mathsf{a}}{\mathsf{b}}{0}{T}{\V}$ we denote the collection
of all the (admissible) parameterized curves.  Furthermore, we call
$({\mathsf t},{\mathsf u})$
\\[0.3em]
\hspace*{3em} \textbullet\ \emph{nondegenerate}, if $\dot{\mathsf
  t}(s)+\scalardens\Psiz \parau(s)>0$ for a.a.  $s\in
(\mathsf{a},\mathsf{b})$;
\\[0.2em]
\hspace*{3em} \textbullet\ \emph{surjective}, if ${\mathsf
  t}(\mathsf{a})=0,{\mathsf t}(\mathsf{b})=T$;
\\[0.2em]
\hspace*{3em} \textbullet\ $\mathsf m$-\emph{normalized} for a
positive $\mathsf m\in L^\infty(0,\sfS)$ (typically $\mathsf m\equiv
1$), if $({\mathsf t},{\mathsf u})$ fulfills
 \begin{equation}
   \label{e:norma-cond}
   \dot\sft( s)+ \scalardens
   \Psiz\parau(s)+\fre_{\sft(s)}(\sfu(s))\|\dot u(s)\|
  = \mathsf m(s)\quad
  \text{for a.a. } s \in (\mathsf{a},\mathsf{b})\,.
\end{equation}
Two (admissible) parameterized curves $s\in
[\mathsf{a},\mathsf{b}]\mapsto (\sft(s),\sfu(s))$ and $\sigma\in
[\mathsf{c},\mathsf{d}]\mapsto
(\hat{\sft}(\sigma),\hat{\sfu}(\sigma))$ are equivalent if there
exists an \emph{absolutely continuous and surjective change of
  variable} $\sfs:\sigma\in [\mathsf{c},\mathsf{d}]\mapsto
\sfs(\sigma)\in [\mathsf{a},\mathsf{b}]$ such that
\[
    \hat{\sft}(\sigma)=\sft({\sfs(\sigma)}),\quad
    \hat{\sfu}(\sigma)=u({\sfs(\sigma)})\quad\text{for all $\sigma\in
      (\mathsf{c},\mathsf{d})$},\qquad 
    \dot{\sfs}(\sigma)>0\quad \forae\, \sigma \in (\mathsf{c},\mathsf{d}).
\]
\end{definition}
\noindent
The above concept is nothing but the parameterized counterpart to the
notion of admissible curve from Definition \ref{def:admissible}: a
crucial feature of parameterized curves is their $\Leb1$-a.e.\
differentiability on the set $G$.

In the next definition of parameterized solutions we will impose (a
suitable version of) \eqref{eq:16} as an equality. Indeed, the upper
energy estimate has been motivated throughout
\eqref{resc-enid-eps}--\eqref{eq:16} via lower semicontinuity
arguments. The lower energy estimate is a consequence of the chain
rule of the forthcoming Theorem \ref{th:3.8}.

\begin{definition}[Parameterized  solutions]
\label{def:2.1-parasol}
A {\em parameterized solution of the \ris\ $\RIS$} is a surjective and
nondegenerate curve $({\mathsf t},{\mathsf u}) \in
\adm{\mathsf{a}}{\mathsf{b}}{0}{T}{\V}$ (cf.\ Def.\ \ref{def:3.5})
satisfying
\begin{equation}
  \label{eq:16-param-tech}
  \begin{aligned}
    \int_{s_1}^{s_2} \frF[\sft,\sfu;\sft',\sfu'] \,\dd s  
     +\ene{{\mathsf t}(s_2)}{{\mathsf u}(s_2)} =
  \ene{{\mathsf t}(s_1)}{{\mathsf u}(s_1)}
  \qquad
  \text{for all }\mathsf{a}\le s_1\le s_2\le \mathsf{b}.
  \end{aligned}
\end{equation}
\end{definition}

  Since $\frF$ defined in \eqref{eq:50} contains the term
 $\frk_\parat(\parau)\dot\parat$, the equation
 \eqref{eq:16-param-tech} encompasses the local stability condition
 \eqref{eq:65bis}. 
 It follows from \eqref{eq:RIF:4} and the power-control condition
\eqref{hyp:en3} that, along a parameterized solution, the map $s
\mapsto \ene{\sft(s)}{\sfu(s)}$ is absolutely continuous on
$[\mathsf{a}, \mathsf{b}]$.
 
\subsubsection*{\bfseries The main existence and convergence result}
The main result of this section states that any limit curve
of the rescaled family $({\mathsf t}_{\eps},{\mathsf
u}_{\eps})$ of solutions to \eqref{viscous-dne} is a \emph{parameterized
solution}.
\begin{theorem}
\label{thm-van-param} Assume \eqref{e:2.1}--\eqref{def-psiV}
 and \eqref{Ezero}--\eqref{hyp:en-subdif}. Let
$(u_\eps)_\eps\subset \AC ([0,T];\V)$ be a family of solutions to
the doubly nonlinear equation
\eqref{viscous-dne}, such that
\begin{equation}
\label{conve-initi-databis} u_\eps (0) \to u_0 \quad \text{in $\V$
 and}\quad \ene 0{u_\eps (0)}\to \ene 0{u_0} \quad \text{as
$\eps \downarrow 0$}
\end{equation}
as in \eqref{conve-initi-data}. Choose non-decreasing surjective
time-rescalings $\parat_\eps : [0,\mathsf{S}] \to [0,T]$, define
 $\parau_\eps: [0,\mathsf{S}] \to \V$ by
$\parau_\eps (s):= u_\eps (\parat_\eps (s))$ for all $s \in
[0,\mathsf{S}]$ and suppose that 
\begin{equation}
\label{e:uniform-integrability} \exists\, \mathsf{m} \in L^\infty
(0,\mathsf{S}) \ : \ \ \mathsf{m}_\eps:= \dot{\parat}_\eps +
\frf_{\sft_\eps}(\parau_\eps,\dot\parau_\eps)
\weaksto \mathsf{m} \quad \text{in
$L^\infty(0,\mathsf{S})$}  \text{ and }   \mathsf{m}>0\ \aein\,
(0,\mathsf{S})\,. 
\end{equation}
Then, there exist a subsequence $\eps_k \down
0$ and a 
\emph{parameterized solution} $(\parat,\parau)\in \AC ([0,\mathsf{S}];
[0,T] \times \V)$ to the \ris\ $\RIS$, such that the following
convergences hold as $k \to \infty$:
\begin{align}
  \label{e:conv1-para} (\parat_{\eps_k},\parau_{\eps_k})   &\longrightarrow
(\parat,\parau)\ \text{ in }{\mathrm C}^0([0,\mathsf{S}];[0,T]\times\V),
\\
 \label{e:conv2-para}
 \ene{\parat_{\eps_k}(s)}{\parau_{\eps_k}(s)}
 &\longrightarrow \ene{\parat(s)}{\parau(s)}\ \text{ uniformly in }
   [0,\mathsf{S}],
\\
\label{e:conv3-para}
\int_{s_1}^{s_2} \Big(\Psiz(\dot \parau_{\eps_k})+
 \frG_{\eps_k}(\sft_{\eps_k},\sfu_{\eps_k};\dot\sft_{\eps_k},\dot\sfu_{\eps_k})\Big)
 \;\!\dd s &\longrightarrow
 \int_{s_1}^{s_2}
 \Big(\scalardens\Psiz\parau+\frG[{\parat},{\parau};{\dot{\parat}},{\dot{\parau}}]
 \Big) \;\!\dd s
\end{align}
for all $0\le s_1\le s_2\le \mathsf{S}$.  Moreover, $(\sft,\sfu)$ is
$\mathsf m$-normalized.
\end{theorem}
\noindent
We have already seen that the choice
\eqref{energy-dissipation-arlength}--\eqref{e:resc2} provides the
normalization condition \eqref{resc-par-1}, and thus (up to a
multiplication factor converging to $1$) the curves
$(\sft_\eps,\sfu_\eps)$ satisfy \eqref{e:uniform-integrability} with
$\mathsf m\equiv 1$.

The proof of this result is postponed to the end of
\S\,\ref{ss:8-vanvisc}.

\subsubsection*{\bfseries Chain rule and further properties of parameterized solutions}
We present now
a \emph{parame\-trized}  version of the chain rule 
\eqref{eq:45tris} (cf.\ also \eqref{ch-rule-ineq}), satisfied by admissible parameterized curves.
In fact, \eqref{eq:RIF:5bis} is
a \emph{metric-like} chain-rule inequality,
since it involves the $\Psi$-metric derivative of the curve.
A
key ingredient of its proof  is
 the \emph{uniform subdifferentiability} condition \eqref{hyp:en-subdif}.
\begin{theorem}[Chain-rule inequality for parameterized curves]
  \label{th:3.8}
  If $(\parat,\parau) \in \adm{\mathsf{a}}{\mathsf{b}}{0}{T}{\V}$ then
  the map $s\mapsto \ene {\parat(s)}{\parau(s)}$
  is \emph{absolutely continuous} on $[\mathsf{a},\mathsf{b}]$
  and the following \emph{chain-rule inequality} holds $\forae\, s \in
  (\mathsf{a},\mathsf{b})$
  (recalling \eqref{eq:51})
    \begin{align}
      \label{eq:RIF:5bis}
      \Big|\frac \rmd{\rmd s}\EE_{\parat(s)}(\parau(s))-\power
      {\parat(s)}{\parau(s)}\dot{\parat}(s)
      \Big|
      \le
      \scalardens \Psiz{\parau}(s) +
    \vvmV {\parat(s)}{\parau(s)}{\dot{\parau}(s)}.
    \end{align}
    Moreover, if $\sfu$ is a.e.\ differentiable, then
    $\forae\, s \in  (\mathsf{a}, \mathsf{b}) $ we have
    \begin{equation}
     \label{eq:4}
     \begin{aligned}
     \frac \rmd{\rmd s}\EE_{\sft(s)}(\parau(s))-\power
     {\parat(s)}{\parau(s)}\dot{\parat}(s)= 
     -\la \xi,\dot{\parau}(s)\ra \ge  -\vvm
     {\parat(s)}{\parau(s)}{\dot{\parau}(s)}
     \text{ for all } \xi\in -\partial\ene {\sft(s)}{\sfu(s)}.
  \end{aligned}
  \end{equation}
\end{theorem}
\noindent We postpone the proof to Section \ref{ss:8-chain}.  As a
straightforward consequence of the chain-rule inequality
\eqref{eq:RIF:5bis}, we can characterize parameterized solutions by a
simpler \emph{one-sided} inequality on the interval $(\mathsf{a},
\mathsf{b})$. The result below corresponds to Corollary
\ref{prop:BV-charact} for $\BV$ solutions.

\begin{corollary}
\label{PROP:charact-param}
For every surjective and nondegenerate
admissible curve in $(\sft,\sfu)\in \adm{\mathsf{a}}{\mathsf{b}}{0}{T}{\V}$
the following three conditions are equivalent:
  \begin{align}
    & \text{\emph{ i)}} && \notag\text{$(\sft,\sfu)$ is a
      parameterized solution of the \ris\ $\RIS$};\\ 
      \label{eq:16-param-tech-onesided}
     & \text{\emph{ ii)}} &&     \int_{\mathsf{a}}^{\mathsf{b}}
        \frF[\sft,\sfu;\sft',\sfu']\,\dd s
        +\ene{{\mathsf t}(\mathsf{b})}{{\mathsf u}(\mathsf{b})} \leq
        \ene{{\mathsf t}(\mathsf{a})}{{\mathsf
            u}(\mathsf{a})}; \\
      \label{eq:RIF:5bis-new}
     & \text{\emph{ iii)}} &&  \frac \rmd{\rmd s}\EE_{\parat(s)}(\parau(s))-\power
      {\parat(s)}{\parau(s)}\dot{\parat}(s) = - \scalardens\Psiz
      {\parau}(s) - \vvmV {\parat(s)}{\parau(s)}{\dot{\parau}(s)}
      \ \forae\ s \in (\mathsf{a},\mathsf{b}).
%
  \end{align}
\end{corollary}
When $\sfu$ is $\Leb 1$-a.e.~differentiable,
it is also possible to
characterize parameterized solutions in terms of a doubly nonlinear differential inclusion
involving the dissipation potentials $\Psiz$ and  $\Psiv$
(to be compared with the differential characterization
of  $\BV$ solutions in Theorem \ref{prop:diff-charact-bv}).
\begin{proposition}
\label{prop:diff-charact-param}
If $({\mathsf t},{\mathsf u}) $ is a 
$\Leb 1$-a.e.~differentiable parameterized solution of the
 \ris\  $\RIS$, then
there exist measurable functions $\lambda: (\mathsf{a}, \mathsf{b}) \to
[0,+\infty)$ and $\xi:(\sfa,\sfb)\to V^*$ such that
\begin{equation}
\label{e:529}
\xi(s)\in \Big(\partial \Diss{\Bo}{} (\dot{\parau}(s))
+\partial\Diss{\V}{}(\lambda(s)\dot{\parau}(s))\Big)\cap \Big(
 -\frsub
\ene{\parat(s)}{\parau(s)}\Big), \quad \lambda(s)\dot{\parat}(s)=0
\quad \foraa\, s \in (\mathsf{a},\mathsf{b}).
\end{equation}
Conversely, if an absolutely continuous, surjective, nondegenerate and
$\Leb 1$-a.e.~differentiable curve $(\sft,\sfu):[\sfa,\sfb]\to
[0,T]\times D_E$ satisfies \eqref{e:529} for some measurable maps
$\lambda,\xi$ and $s\mapsto \ene{\sft(s)}{\sfu(s)}$ is absolutely
continuous in $[\sfa,\sfb]$, then $(\sft,\sfu)$ is a parameterized
solution to the \ris\ $\RIS$.
\end{proposition}

The reformulation of the notion of parameterized solutions in terms of
the subdifferential inclusion \eqref{e:529} reflects the following
\emph{mechanical interpretation}:
\begin{itemize}
\item the regime $(\dot\sft>0, \ \dot \sfu \equiv 0)$ corresponds to
  \emph{sticking};
\item the regime $(\dot\sft>0, \ \dot \sfu \neq 0)$ corresponds to
  \emph{rate-independent sliding} ( $\lambda =0$ implies the local
  stability $K^* + \frsub \ene{\sft}{\sfu} \ni 0$);
\item when $\dot\sft=0$ (i.e.\ at a jump in the (slow) external time
  scale, encoded in the function $\sft$), the system may switch to a
  \emph{viscous regime} (when $\lambda>0$), and the solution follow a
  viscous transition path.
\end{itemize}
 
\begin{proof}
  If $({\mathsf t},{\mathsf u}) $ is a $\Leb 1$-a.e.~differentiable
  parameterized solution, \eqref{eq:RIF:5bis-new} and \eqref{eq:4}
  show that for every selection $\xi\in
  -\partial\ene{\sft(s)}{\sfu(s)}$ we have
  \begin{equation}
    \label{eq:93}
    \langle\xi,\dot \sfu(s)\rangle=\Psiz(\dot
    \sfu(s))+\fre_{\sft(s)}(\sfu(s))\|\dot\sfu(s)\|
  \quad\foraa s\in (\sfa,\sfb).
  \end{equation}
  If $\fre_{\sft(s)}(\sfu(s))=0$ then choosing $\xi\in K^*$
  we get \eqref{e:529} with $\lambda(s)=0$.
  If $\fre_{\sft(s)}(\sfu(s))>0$ then $\dot\sft(s)=0$ so that
  $\dot \sfu(s)\neq 0$ by the nondegeneracy condition;  we
  obtain \eqref{e:529} by choosing
  $\lambda(s)=\Lambda_{\sft(s)}(\sfu(s),\dot\sfu(s))$,
  see \eqref{not-contact-viscosities}.

  Conversely, assume \eqref{e:529} and that the energy
  map is absolutely continuous.  If $\lambda(s)=0$ then
  $\fre_{\sft(s)}(\sfu(s))=0$ so that $\langle\xi,\dot \sfu(s)\rangle=
  \Psiz(\dot \sfu(s))$. If $\lambda(s)>0$ then $\dot\sft(s)=0$ so that
  $\dot u(s)\neq 0$ and
   \begin{align}
     \notag
     \langle\xi,\dot \sfu(s)\rangle&=
    \Psiz(\dot
    \sfu(s))+\frac1{\lambda(s)}\Phi(\lambda(s)\dot \sfu(s))+
    \frac 1{\lambda(s)}\Phi^*(\xi)\ge
    \Psiz(\dot
    \sfu(s))+\fre_{\sft(s)}(\sfu(s))\|\dot\sfu(s)\|
    \ge \langle\xi,\dot \sfu(s)\rangle.
  \end{align}
  Hence, all the above estimates are equalities, and therefore
  $\fre_{\sft(s)}(\sfu(s))>0$.  Furthermore, \eqref{eq:93} holds.
  Combining this with the fact that at almost all points the energy is
  differentiable with derivative
  $  \frac\dd{\dd s}\ene{\sft(s)}{\sfu(s)}=\power
    {\sft(s)}{\sfu(s)}\dot\sft(s)-
    \langle \xi,\dot \sfu(s)\rangle$
  in $L^1(\sfa,\sfb)$, we conclude that $(\sft,\sfu)$ is admissible
  and \eqref{eq:RIF:5bis-new} holds.
\end{proof}

\subsubsection*{\bfseries Parameterized and  $\BV$ solutions}

\begin{proposition}[Equivalence between $\BV$ and parameterized solutions]
  \label{prop:bv}
  \
\begin{enumerate}[\rm (BVP1)]
\item If $(\parat,\parau)\in\adm{\mathsf{a}}{\mathsf{b}}{0}{T}{\V}$ is
  surjective and nondegenerate, then any curve
  \begin{equation}
      \label{eq:126}
      u:[0,T]\to \V\quad
      \text{with}\quad
      u(t)\in \big\{{\mathsf u}(s):\parat(s)=t
      \big\}
  \end{equation}
  belongs to $\BV([0,T];D_E,\Psiz)$ for some $E>0$, satisfies the
  local stability condition \eqref{eq:65bis}, and for every $0\le
  t_0<t_1\le T$ with $G$ defined as in \eqref{eq:RIF2:6} we have
  \begin{equation}
    \label{eq:95bis}
    \pVar\frf u{t_0}{t_1}\le
    \int_{\sfs(t_0)}^{\sfs(t_1)} \scalardens\Psiz\sfu(s)\,\dd s+
    \int_{[\sfs(t_0),\sfs(t_1)]\cap G} \fre_{\sft(s)}(\sfu(s))\|\dot
    \sfu(s)\|\,\dd s;
  \end{equation}
  in particular $\pVar\frf u0T<\infty$.
\item If $(\parat,\parau):[0,{\mathsf S}]\to [0,T]\times \V$ is a
  parameterized solution of the \ris\ $\RIS$, then any curve
  $u:[0,T]\to \V$ satisfying \eqref{eq:126}
  is a $\BV$ solution 
  in the sense of Definition \ref{def:BV-solution}.
\item Conversely, if $u\in \BV([0,T];D_E,\Psiz)$ satisfies
  \eqref{eq:65bis} with $\pVar\frf u0T<\infty$, then there exists a
  nondegenerate, surjective $(\sft,\sfu)\in \mathscr
  A(0,\sfS;[0,T]\times V)$ such that \eqref{eq:126} holds and
  \begin{equation}
    \label{eq:95tris}
    \pVar\frf u{0}{T}=
    \int_{0}^{\sfS} \scalardens\Psiz\sfu(s)\,\dd s+
    \int_{[0,\sfS]\cap G} \fre_{\sft(s)}(\sfu(s))\|\dot
    \sfu(s)\|\,\dd s.
  \end{equation}
  Thus if $u$ is a $\BV$ solution of the \ris\ $\RIS$ then
  $(\sft,\sfu)$ is a parameterized solution.
  \end{enumerate}
\end{proposition}
\begin{proof}
  (BVP1): let $\sfs:[0,T]\to [\sfa,\sfb]$
  be any inverse of $\sft$. Notice that
  $t\in \mathrm J_u$ if and only if $t\in \mathrm J_{\sfs}$ and
  $\sft(s)\equiv t$ for every $s\in [\sfs(t_-),\sfs(t_+)]$.
  We can also define $\sfs(t)$ in $[\sfs(t_-),\sfs(t_+)] $
  so that $u(t)=\sfu(\sfs(t))$ for every
  $t\in [0,T]$. By this choice it is immediate to see that
  $u\in \BV([0,T];D_E,\Psiz)$ with
  \begin{displaymath}
    \Var\Psiz u{t_0}{t_1}=\Var\Psiz\sfu{\sfs(t_0)}{\sfs(t_1)}=
    \int_{\sfs(t_0)}^{\sfs(t_1)}\scalardens \Psiz\sfu(r)\dd r\quad
    \forevery 0\le t_0<t_1\le T.
  \end{displaymath}
  On the other hand, the curve $\sfu:[\sfs(t_-),\sfs(t_+)]\to V$
  is an admissible transition connecting $u(t_-)$ to $u(t_+)$ with
  \begin{displaymath}
    \Delta_{\frf_t}(u(t_-),u(t)) \le
    \int_{\sfs(t_-)}^{\sfs(t)} \frf_{\sfs(r)}[\sfu,\sfu'](r)\,\dd r,\quad
    \Delta_{\frf_t}(u(t),u(t_+)) \le
    \int_{\sfs(t)}^{\sfs(t_+)} \frf_{\sfs(r)}[\sfu,\sfu'](r)\,\dd r,
\end{displaymath}
which yields \eqref{eq:95bis}.
Since $\dot\sft=0$ in $G$,
$\sft(G)$ is $\Leb 1$-negligible, so that its complement
(where the local stability condition
\eqref{eq:65bis} holds) is dense in $[0,T]$.
Since $\fre$ is lower semicontinuous, every point in
$[0,T]\setminus \mathrm J_u$ satisfies \eqref{eq:65bis}.

  (BVP2) is now immediate: since  \eqref{eq:65bis} holds,
  it is sufficient to check \eqref{eq:84-oneside};
  this follows by combining \eqref{eq:95bis},
  \eqref{eq:16-param-tech},
  and
  the change of variable formula
  \begin{equation}
    \label{eq:97}
    \int_0^T \power t{u(t)}\,\dd t=
    \int_0^{\sfS} \power {\sft(s)}{\sfu(s)}\dot\sft(s)\,\dd s.
  \end{equation}

  In order to prove (BVP3), we
  introduce the parameterization
\begin{gather}\label{eq:43}
  \mathsf s(t):=t+\pVar \frf u0t,\quad \mathsf S:=\sfs(T),\quad
  \mathrm J_u=\mathrm J_\sfs=
  (t_n)_{n\in \N},\\
    \label{eq:92}
  I_n:=(\mathsf s(t_{n-}),\mathsf s(t_{n+})),\quad
  I:=\bigcup_{n\in \N}I_n,
\quad
\sft:=\sfs^{-1}:[0,\mathsf S]\setminus I \to [0,T],\quad
  \parau:=u\circ\sft.
\end{gather}
It is immediate to check that $\sft$ and $\sfu$ are Lipschitz maps.
We extend $\sft$ and $\sfu$ to $I$ by setting
\begin{equation}
  \label{eq:1}
  \sft(s)\equiv t_n,\quad
  \sfu(s):=\vartheta_n(\mathsf r_n(s))\quad
  \text{whenever }s\in I_n,
\end{equation}
where $\mathsf r_n:\overline {I_n}\to [0,1]$ is the unique affine and
strictly increasing function mapping $\overline {I_n}$ onto $
[0,1]$
and $\vartheta_n\in \calT_{t_n}(u(t_{n-}),u(t_{n+}))$
is an admissible transition satisfying $\vartheta_n(\sfr_n(\sfs(t_n)))=
u(t_n)$ and (recall (F1) of Theorem \ref{thm:fcost})
\begin{equation}\label{eq:101}
  \int_0^1 \mathfrak f_{t_n}[\vartheta_n;\vartheta_n'](r)\,\dd r=
  \Delta_{\mathfrak f_{t_n}}(u(t_{n-}),u(t_n))+
  \Delta_{\mathfrak f_{t_n}}(u(t_{n}),u(t_{n+})).
\end{equation}
It follows that \eqref{eq:126} holds with $u=\sfu\circ \sfs$ and
\begin{align*}
  &\int_0^\sfS \scalardens\Psiz \parau(s)\,\dd s+
  \int_G \fre_{\sft(s)}(\sfu(s)) \,\|\dot\sfu(s)\|\,\dd s
  =
  \Var \Psiz\parau 0S+
  \int_G \fre_{\sft(s)}(\sfu(s)) \,\|\dot\sfu(s)\|\,\dd s
  \\&\quad=
  \Var \Psiz u0T+\sum_{n\in \N}\int_0^1
  \fre_{t_n}(\vartheta_n(r))\|\dot\vartheta_n(r)\|\,\dd r
  \\&\quad\le
  \Var \Psiz u0T-\mathrm{Jmp}_\Psiz(u;[0,T])+\mathrm {Jmp}_{\mathfrak f}(u;[0,T])
  =\pVar{\mathfrak f}u0T,
\end{align*}
so that \eqref{eq:95tris} holds and
$(\sft,\sfu)\in \mathscr A(0,\sfS;[0,T]\times V)$.

If moreover $u$ is a $\BV$ solution, then the chain rule from 
Theorem \ref{th:3.8} and \eqref{eq:97} yield inequality 
\eqref{eq:16-param-tech-onesided}.
\end{proof}

\subsection{$\V$-parameterized solutions}
\label{ss:3.2} 
We consider now the special class of parameterizable solutions,
corresponding to the notion introduced in \S \ref{ss:Vparam}, namely
those for which $\sfu$ is absolutely continuous with values in $\V$.

\begin{definition}
  \label{def-V-param} A $\V$-parameterized solution
  $(\parat,\parau):[\sfa,\sfb]\to [0,T]\times V$ of the \ris\ $\RIS$
  is a parameterized solution such that $\sfu \in\AC(\sfa,\sfb;\V)$.
\end{definition}

Since $\V$-parameterized solutions are differentiable $\Leb1$-a.e.,
one does not have to distinguish the behavior of $\sfu$ in the set $G$
of \eqref{eq:RIF2:6} from its complement.  By adopting the
``pointwise'' definition \eqref{vvmfullname-def} of $\vvmfullname$ and
$\vvmfullVname$ in place of \eqref{eq:50}, metric concepts are no
longer needed, and expressions like \eqref{eq:RIF:4} become simpler.

\begin{proposition}
\label{prop:para-Vparam} 
If $(\parat,\parau) \in \AC ([0,{\mathsf{S}}];[0,T]\times \V)$ is a
$\V$-parameterized solution to the \ris\ $\RIS$ then every $u$
satisfying \eqref{eq:126} is a $\V$-parameterizable $\BV$ solution.
If $u$ is a $\V$-parameterizable $\BV$ solution there exists
a 
$\V$-parameterized solution $({\mathsf t},{\mathsf u})$ such that
\eqref{eq:126} hold.
\end{proposition}
\begin{proof}
The approach is analogous to the proof of Proposition
\ref{prop:bv}: In one direction it follows by the identity
$\Var{} u0T=\int_0^\sfS \|\dot \sfu(s)\|\,\dd s$.
In the opposite one, we can simply replace \eqref{eq:43} by
\begin{equation}
  \label{eq:104}
  \sfs(t):=t+\pVar \frf u0t+\Var{}u0t,
\end{equation}
choosing the optimal jump transitions according to
\eqref{special-jump-cond}.
\end{proof}

Thanks to Proposition \ref{prop:para-Vparam}, Corollary
\ref{cor:3-limit} implies the following result:

\begin{corollary}[Existence of $\V$-parameterized solutions]
 If \eqref{particular-dissipations}--\eqref{eq:77} hold,
 then
  for every $u_0\in D$ with
  $K^* + \partial\ene 0{u_0}\ni 0$
  there exists a $\V$-parameterized
  solution $(\sft,\sfu)\in \AC([0,\sfS];[0,T]\times V)$ of
  the \ris\  $\RIS$.
\end{corollary}
$\V$-parameterized solutions can also be obtained
as limit of rescaled solutions to \eqref{viscous-dne} if
they satisfy the uniform bound \eqref{eq:98}:
one can simply adapt the argument discussed
in \S \ref{ss:5.1}, by replacing the definition
\eqref{energy-dissipation-arlength} of the arclength ${\mathsf s}_\eps$  with, e.g.,
\begin{equation}
  \label{eq:105}
  {\mathsf s}_\eps(t):=t
  + \int_0^t \vvm{r}{u_\eps(r)}{\dot{u}_\eps(r)} \,\mathrm{d} r+
  \int_0^t \|\dot u_\eps(r)\|\,\dd r,\quad
  \sft_\eps:=\sfs_\eps^{-1},
\end{equation}
in order to gain a uniform control of the Lipschitz constant
of the rescaled functions $\sfu_\eps$. The vanishing-viscosity 
limit in  Theorem \ref{thm-van-param} then gives the following.
 
\begin{theorem}
  \label{cor:easy}
  Let $(u_\eps)_{\eps>0}$ be a family of solutions to
  \eqref{viscous-dne} satisfying \eqref{conve-initi-data} at $t=0$ and
  the uniform bound \eqref{eq:98} (e.g.~when the assumptions of
  Theorem \ref{th:3-discrete} are satisfied) and let
  $\sft_\eps:[0,\sfS]\to[0,T]$ be nondecreasing and surjective time
  rescalings  (e.g.\ \eqref{eq:105}) such that
  $\sfu_\eps:=u_\eps\circ\sft_\eps$ satisfy
  \eqref{e:uniform-integrability} and there exists $C>0$ such that
  $\sup_{t \in (0,T)}\|\dot \sfu_\eps(t)\|\le C$ for all $\eps>0$.
  Then any limit function $(\sft,\sfu)$ as in Theorem
  \ref{thm-van-param} is a $V$-parameterized solution.
\end{theorem}

\subsubsection*{\bfseries $\V$-arclength parameterizations}
Still keeping the assumptions
\eqref{particular-dissipations}--\eqref{eq:77} of Corollary
\ref{coro:3-continuous}, in particular the choice $\Phi(v):=\frac
12\|v\|^2$,
we discuss now a different reparameterization technique
for studying the limit of solutions to \eqref{viscous-dne}.
Since estimate \eqref{eq:98} is guaranteed, 
like in  \cite{ef-mie06,Miel08?DEMF,Mielke-Zelik}  
 we are entitled to   use the
\emph{$\V$-arclength} parameterization
\begin{equation}
\label{l2-rep}
 \hat{\mathsf s}_\eps(t):=t
 + \int_0^t  \Vnorm{\dot{u}_\eps(r)} \,\mathrm{d} r
\end{equation}
and consider the rescaled functions $(
\hat{\parat}_\eps,\hat{\parau}_\eps): [0,\hat{\mathsf S}_\eps] \to
[0,T]\times \V$, with $\hat{\mathsf S}_\eps=\hat{\mathsf s}_\eps(T)$,
defined by $ \hat{\parat}_\eps(s):=\hat{\mathsf s}_\eps^{-1}(s)$ and
$\hat{\mathsf u}_\eps(s):= u_\eps (\hat{\mathsf t}_\eps (s))$.  By
construction we have $ \dot{\hat{\parat}}_\eps(s) +
\Vnorm{\dot{\hat{\parau}}_\eps(s)}=1$ for a.a.\ $s \in
(0,\hat{\mathsf{S}}_\eps),$ and the pair $(\hat{\parat}_\eps,
\hat{\parau}_\eps)$ is a solution of the ``rescaled'' doubly nonlinear
equation
\begin{equation}\label{param_sol_eps}
\partial\Diss{\Bo}{}(\dot{\hat{\parau}}_\eps(s)) +
\frac{\eps}{1{-}\Vnorm{\dot{\hat{\parau}}_\epsilon(s)}}
\partial\Phi(\dot{\hat{\parau}}_\eps(s)) + \frsub
\ene{\hat{\parat}_\eps(s)}{\hat{\parau}_\eps(s)} \ni 0 \quad
\foraa\, s \in (0,\hat{\mathsf{S}}_\eps),
\end{equation}
where we used the degree-$1$ homogeneity of $\partial\Phi$.  As in
\cite{ef-mie06,Mielke-Zelik,Miel08?DEMF}, we observe that the viscous
term in \eqref{param_sol_eps} is the subdifferential of the potential
$\widehat \Phi$ that is defined via
 \[
 \widehat\Phi(v)= f\left(\Vnorm{v}\right) \text{ with }
  f(x)=\begin{cases}
     -\log(1-x)-x &\text{if }0 \leq x<1,\\
     +\infty           &\text{if } x \geq 1.
   \end{cases}
 \]

Thus, \eqref{param_sol_eps} rewrites as
\begin{equation}\label{param_sol_eps-new}
\partial\Diss{\Bo}{}(\dot{\hat{\parau}}_\eps(s)) +
\eps\partial\hat\Phi(\dot{\hat{\parau}}_\eps(s)) + \frsub
\ene{\hat{\parat}_\eps(s)}{\hat{\parau}_\eps(s)} \ni 0 \quad
\foraa\, s \in (0,\hat{\mathsf{S}}_\eps).
\end{equation}
The sequence of dissipation potentials $\widehat{\Psi}_\eps(v):= \Diss
{\B}{}(v)+ \eps\widehat\Phi(v)$ $\Gamma$-converges monotonously, as
$\eps \down 0$, to the limiting potential
\begin{equation}\label{wtR0}
    \widehat{\Psi}(v)=\begin{cases}
    \Diss {\B}{}(v) &\text{if }\Vnorm{v}\leq 1,\\
    +\infty &\text{else}.
    \end{cases}
\end{equation}
It was shown in \cite[Prop.\,4.14]{Miel08?DEMF} that, up to a
subsequence, the parameterized solutions
$(\hat{\parat}_\eps,\hat{\parau}_\eps)$ converge in
$\mathrm{C}^0([0,\hat{\mathsf{S}}];[0,T]\times \V)$ to a pair
$(\hat{\parat},\hat{\parau}) \in
\mathrm{C}^0_{\mathrm{lip}}([0,\hat{\mathsf{S}}];[0,T]\times \V)$ such
that $\hat{\parat}(0)=0$, $\hat{\parat}$ is non-decreasing, and
\begin{equation}
\label{param_sol}
   \dot{\hat{\parat}}(s) + \Vnorm{\dot{\hat{\parau}}(s)}\in [0,1] 
\quad \text{and} \quad 
   \partial \widehat{\Psi}(\dot{\hat{\parau}}(s)) + \frsub
   \ene{\hat{\parat}(s)}{\hat{\parau}(s)} \ni 0 \quad  \foraa\, s \in
(0,\hat{\mathsf{S}}).
\end{equation}

An interesting feature of this approach is that it allows for a direct
passage to the limit in the subdifferential inclusion
\eqref{param_sol_eps-new}, without passing through an energy identity
like \eqref{resc-enid-eps}. By operating a suitable time rescaling,
it is possible to show a correspondence between $V$-parameterized
solutions in the sense of Definition \ref{def-V-param} and in the
sense of \eqref{param_sol}: the interested reader is referred to
\cite[Cor.\,4.22, Prop.\,4.24]{Miel08?DEMF}.

However, let us stress that  the technique from
\cite{ef-mie06,Mielke-Zelik} does not allow us  to prove that the
limit curve $(\hat\parat,\hat\parau)$ satisfies the normalization
condition $ \dot{\hat{\parat}} + \Vnorm{\dot{\hat{\parau}}}=1$ a.e.\
in $(0,\hat{\mathsf{S}})$.  Instead, our the variational approach of
\S\,\ref{ss:5.1}, which is based on a chain-rule and energy-identity
argument, guarantees the preservation of the normalization condition,
cf.\ Theorem \ref{thm-van-param}.  Moreover, we also obtain the
absolute continuity of the energy map $s \mapsto
\ene{\parat(s)}{\parau(s)}$.

\section{Examples}
\label{s:examples}

Throughout this section, we focus on the rate-independent system $\RIS$ given by
\[
\V=L^2(\Omega),
\quad \Diss{\Bo}{}(v)= \int_\Omega |v(x)|\,\dd x,
\quad \Diss{\V}{}(v)= \frac 12 \Vnorm{v}^2= \frac12 \int_\Omega |v(x)|^2\,\dd x
\]
with $\Omega \subset \R^d$, $d\geq 1$, a bounded Lipschitz domain, and
on the following class of energy functionals $\cE: [0,T]\times
L^2(\Omega) \to (-\infty,+\infty]$
\begin{equation}
\label{disc-gene}
\ene tu = \left\{
\begin{array}{ll}
\int_\Omega \left(\beta (|\nabla u|)+W(u) - \ell (t) u \right) \dd x
& \text{ if } u \in \mathrm W^{1,1} (\Omega), \, \beta
(|\nabla u|),\,W(u) \in L^1(\Omega),
\\
+\infty & \text{ otherwise}.
\end{array}
\right.
\end{equation}
Hereafter, we suppose that
\begin{align}
& \label{assumpt-basic-beta}
\beta: [0,+\infty) \to [0,+\infty) \text{ is convex};
\\
&
\label{assumpt-basic-phi}
W : \R \to (-\infty,+\infty] \text{ is bounded from below};
\\
& \label{assumpt-external-load}
\ell \in \mathrm{C}^1 ([0,T];L^{2}(\Omega)).
\end{align}
In all of the examples we present, $\cE$ will satisfy \eqref{Ezero}
and for each of them we will discuss the coercivity condition
\eqref{eq:17}.  Exploiting \eqref{assumpt-external-load}, it is
immediate to check that for all $u \in \domainenergy$ the function
$t\mapsto \ene tu$ is differentiable, with derivative
$\power tu= -\int_\Omega \ell'(t) u \dd x$ which fulfills both \eqref{hyp:en3}
and the Lipschitz estimate \eqref{eq:77}.  In what follows, the focus will be on the uniform
subdifferentiability \eqref{hyp:en-subdif} and on the (stronger)
generalized convexity \eqref{eq:18-new} (which yields the subdifferentiability condition \eqref{subdiff-charact} and in particular \eqref{hyp:en-subdif}). 

We start with Example \ref{ex:4.1}, where we provide sufficient
conditions on the nonlinearities $\beta$ and $W$ guaranteeing the
validity of \eqref{eq:18-new}.

\begin{example}
  \label{ex:4.1}
  \slshape
    We take
    \begin{equation}
    \label{nonlin-ex-4.1}
    \beta(|\nabla u|)=\frac12 |\nabla u|^2 \quad \text{ and } \quad   \text{$W \in \mathrm{C}^1(\R)$,
     $\lambda$-convex
     for some $\lambda \in \R$;}
\end{equation}
for instance, one may think of the double-well potential $W(u)=
(1-u^2)^2/4$.  Clearly, $\cE$ from \eqref{disc-gene} fulfills
\eqref{eq:17}.  In order to check \eqref{eq:18-new}, we fix $u,\,v \in
\domainenergy$ and estimate, for $\theta \in [0,1]$,
\begin{equation}\label{eq:EConv.1}
      \begin{aligned}
      \ene t {(1{-}\theta)u + \theta v}
      &  \leq (1-\theta) \ene t u  +\theta  \ene t v
      -\frac{(1{-}\theta)\theta}{2} \Big( \| \nabla(u{-}v) \|^2_{L^2(\Omega)} +
      \lambda \|u{-}v \|^2_{L^2(\Omega)} \Big),
      \end{aligned}
\end{equation}
where we used 1-convexity of $\beta$ and $\lambda$-convexity of
$W$. Hence, for $\lambda>0$ we have \eqref{eq:18-new} with $\alpha_E
=\lambda $ and $\Lambda_E=0$. If $\lambda <0$, we use the
Gagliardo-Nirenberg inequality 
\[
 \begin{aligned}
      \|w \|_{L^2(\Omega)} &\leq C_{\mathrm{GN}} \big(\|w
         \|_{L^1(\Omega)}^{2/{(d+2)}} \| \nabla w
         \|_{L^2(\Omega)}^{d/(d+2)} + \|w \|_{L^1(\Omega)}\big) 
      \leq   \big( \tfrac1{1+|\lambda|} \| \nabla w
         \|_{L^2(\Omega)}^2 + M_\lambda \|w \|_{L^1(\Omega)}^2 \big)^{1/2} ,
  \end{aligned}
\]
for some $M_\lambda>0$, which is equivalent to $-\|\nabla
w\|_{L^2(\Omega)}^2 \leq -(1{+}|\lambda|)\| w\|_{L^2(\Omega)}^2
+(1{+}|\lambda|)M_\lambda\|w \|_{L^1(\Omega)}^2$. Inserting this for
$w=u{-}v$ into \eqref{eq:EConv.1} we obtain estimate \eqref{eq:18-new}
with $\alpha_E =(1{+}|\lambda|)+\lambda=1>0 $ and
$\Lambda_E=(1{+}|\lambda|)M_\lambda $. In particular, we have no
dependence on the energy sublevel $E$.

In fact, it can be checked that for suitably convex
functions $\beta$ with the growth $\beta(|\nabla u|) \geq
c_1|\nabla u|^p -c_2 $ for some  $c_1, c_2>0$
the related functional $\cE$ in \eqref{disc-gene} still
complies with \eqref{eq:18-new}, if $p>p_d$ for a suitable $p_d>1$
depending on the dimension  $d$. 
\end{example}

Our next example treats the case in which $\beta$ has only linear
growth.  Even taking a \emph{convex} function $W$, the generalized
convexity condition \eqref{eq:18-new} is no longer
guaranteed. Nonetheless, since the functional $u\mapsto \ene tu$ is
convex, its Fr\'echet subdifferential reduces to the subdifferential
in the sense of convex analysis, and \eqref{hyp:en-subdif} clearly
holds.  In this setting, we show that there exist $\BV$ solutions to
the rate-independent system $\RIS$, which are \emph{not
  $\V$-parameterizable}.

\begin{example}
\label{ex:3.a}\slshape
We consider the one-dimensional domain $\Omega=(0,l)$ for some 
$l>1$ and take
\begin{equation}
    \label{nonlin-ex-4.1-bis}
    \beta\left(\left| \tfrac{\dd}{\dd x}u\right|\right)=\delta \left| \tfrac{\dd}{\dd x}u\right|
    \quad \text{ with $\delta>0$,  } \quad
    W(u)=\mathrm{I}_{[0,1]}(u) = \left\{
    \begin{array}{ll}
    0 & \text{if } u \in [0,1],
    \\
    +\infty & \text{otherwise,}
    \end{array}
    \right.
     \end{equation}
     and the external loading $\ell: [0,T]\times (0,l) \to
     \R$ with $\ell(t,x)= t{+}2{-}x$,  where $0<T \leq
     l{-}1$. Observe that, thanks to the compactifying character
     of the total-variation contribution $\delta \int_0^l |
     \tfrac{\dd}{\dd x} u| \, \dd x$, the energy $\cE$ fulfills
     \eqref{eq:17}.  We now show that the function
     \[
     u(t,x)= \chi_{[0,a(t)]} (x)= \left\{
    \begin{array}{ll}
    1 & \text{for } x \in [0,a(t)],
    \\
    0 & \text{otherwise}
    \end{array}
    \right.
     \]
     for some continuous and nondecreasing function $a: [0,T] \to
     [0,l]$, which will be specified later, is a $\BV$ solution to the
     \ris\  $\RIS$.

     Concerning the energy balance \eqref{eq:84}, we observe that,
      since $u \in \mathrm{C}^0 ([0,T] ; L^2(0,l))$ there holds
      \[
      \pVar{\vvmnametil}u{0}{t}= \Var{\| \cdot\|_{L^1(0,l)}}{u}0t= a(t)-a(0) \quad \text{for all } t \in [0,1],
      \]
where we also used that $a$ is nondecreasing.
Easy calculations give
$\ene t{u(t)}= \delta -(t{+}2)a(t)+\frac{a^2(t)}{2}$
and $\power  t{u(t)}=-a(t)$, therefore
\eqref{eq:84}
yields the flow rule for the moving interface $a$:
\[
\dot{a}(t) (a(t){-}1{-}t)=0 \quad \Rightarrow \quad a(t)= 1+t\quad
\text{for all } t \in {[0,T]. }
\]

Since $\ene t{\cdot}$ is convex, $u$ fulfills the local stability
$\eqref{eq:65bis} $ if and only if it complies with the global
stability condition \eqref{eq:64bis-intro}, which in the present
setting reads
\begin{equation}
\label{glob-stab}
\begin{array}{ll}
  \delta -\frac12 (t{+1}) (3t{+}5)
  & = \ene t{u(t)} \leq \ene t{v}+ \| v{-}u(t)\|_{L^1(0,l)}
  \\ & = \int_0^l \left( \delta|\tfrac {\dd}{\dd x} v|+ |v{-}\chi_{[0,t+1]}|- (t{+}2{-}x)v \right) \, \dd x
  \\
  & =\delta \int_0^l |\tfrac {\dd}{\dd x} v|\, \dd x +t+1 -\int_0^{t+1} (t{+}3{-}x)v \, \dd x +\int_{t+1}^l
  (x{-}1{-}t)v \, \dd x
   \end{array}
\end{equation}
for all $v \in L^1(0,l)$ and $t \in [0,1]$. With some calculations one
can show that for all $\delta \in [0,2]$ and $l \geq 4$ the function
$u(t,x)=\chi_{[0,t+1]}(x)$ fulfills \eqref{glob-stab}, hence it is a
$\BV$ solution.  Indeed, $u$ is a $\BV$ solution also in the case
$\delta=0$, in which $\cE$ does not comply with \eqref{eq:17} and our
existence results Thms.\ \ref{th:1} and \ref{th:2} do not apply.
Although $u \in \mathrm{C}^{\mathrm{lip}}([0,1]; L^1 (0,l))$, we have
that $u \notin \BV([0,1]; L^2 (0,l))$, therefore it is not a
$\V$-parameterizable $\BV$ solution.
\end{example}

We now revisit \cite[Ex.\,4.4, 4.27]{Miel08?DEMF}, which means in our
notation that $\beta \equiv 0$ and that $W$ is of double-well
type. Relying on the calculations from \cite{Miel08?DEMF}, we show
that as $\eps \to 0$ the viscous solutions converge to a curve $u$,
which is not a $\BV$ solution to the rate-independent system $\RIS$.
Observe that in this case neither \eqref{eq:17}, nor the
(parameterized) chain-rule inequality \eqref{eq:RIF:5bis}, are
fulfilled.

\begin{example}
\label{ex:4.3}
\slshape
We take $\Omega=(0,1)$, $\beta \equiv 0$, $\ell(t,x)=t+x$, and
\begin{equation}
\label{explicit-d-well}
W(u)= \left\{
\begin{array}{ll}
\frac12 (u{+}4)^2 & \text{if } u \leq -2, \\
4-\frac12 u^2 & \text{if } |u|<2,\\
\frac12 (u{-}4)^2 & \text{if } u \geq 2,
\end{array}
 \right.
\end{equation}
In \cite[Ex.\,4.4]{Miel08?DEMF} the unique
solution to the viscous problem
\begin{equation}
\label{viscous-problem}
\mathrm{Sign}(\dot{u}_\eps (t,x))+ \eps \dot{u}_\eps (t,x)
+ W'({u}_\eps (t,x)) \ni \ell(t,x) \ \text{ and } \ u_\eps(0,x)=-4 
\end{equation}
was explicitly
calculated: We have $u_\eps(t,x)=V^ \eps(t{+}x) $, where
$V^\eps(\tau)=-4$ for $\tau\leq 1{+}\eps$ and it coincides with the
unique solution $v$ of  $\mathrm{Sign}(v'(\tau))+\eps v'(\tau)
+W'(v(\tau))\ni \tau$ for $\tau\geq 1{+}\eps$.  
It was shown that, on the time-interval
$[0,6]$ the functions $(u_\eps)_\eps$ have a uniform Lipschitz bound
with values in $L^1(0,1)$, whereas $\int_0^6 \| \dot{u}_\eps
\|_{L^2(0,1)}\, \dd t$ tends to $\infty$ as $\eps \to 0$ like
$1/\sqrt{\eps}$.  Moreover, setting
\[
\bar{u}(t,x)= \max\{ -4, t{+}x{-}5\} \text{ for } t{+}x \leq 3\quad
\text{ and } \bar{u}(t,x)=  t{+}x{+}3 \text{ for } t{+}x > 3
\]
we have $\bar{u} \in \mathrm{C}^0
([0,6];L^2(0,1))\cap \mathrm{C}^{\mathrm{lip}}([0,6]; L^1(0,1))$ and
 $\sup_{t \in [0,6]}\|u_\eps(t)-\bar{u}(t)\|_{L^2(0,1)} \to 0 $ as
$\eps \to 0$, hence obviously $\ene t{u_\eps(t)}\to \ene
t{\bar{u}(t)}$ for all $t \in [0,6]$.

It can be shown that $\bar{u}(t)$ complies with the local stability
condition \eqref{eq:65bis} for all $t \in [0,6]$.  However, $u$ does
not comply with the energy balance \eqref{eq:84}. In fact, by
continuity of $\bar u$ we have $\pVar{\vvmnametil}{\bar u}{0}{t}=
\Var{\| \cdot\|_{L^1(0,l)}}{\bar u}0t$ for all $t \in [0,6]$, and
passing to the limit as $\eps \to 0$ in the \emph{viscous} energy
balance \eqref{eq:52bis} it can be calculated explicitly that for all
$t \in [0,6]$
\begin{equation}
\label{not-en-bal} \textstyle
\Var{\| \cdot\|_{L^1}}{\bar u}0t + \ene t{\bar u(t)} -\ene t{\bar u(0)}-
\!\int\limits_0^t\! \power s{\bar u(s)}
\!\; \dd s= 8 \max \{0,\min\{t{-}2,1\}\}  =: \rho(t).
\end{equation}
Therefore, following \cite{Miel08?DEMF} we observe that there is an
additional limit dissipation $\rho$ in \eqref{not-en-bal}, and
$\bar u$ is not a $\BV$ solution.

In fact, the chain-rule inequality \eqref{eq:RIF:5bis} does not hold
along the parameterized curve (cf.\ Definition \ref{def:3.5})
$(\parat,\parau) \in \adm{0}{6}{0}{6}{L^2(0,1)}$ given by $s \mapsto
(\parat(s),\parau(s)):= (s,\bar u(s)) \in [0,6]\times L^2(0,1)$. On
the one hand, since $\bar u$ satisfies \eqref{eq:65bis} on $[0,6]$, we
have  $\vvmV
{\parat(s)}{\parau(s)}{\dot{\parau}(s)}_{L^2(0,1)}\equiv 0$  on
$[0,6]$.  On the other hand, \eqref{not-en-bal} yields for almost all
$s \in (0,6)$
  \begin{equation}
  \label{not-chain-rule}
 \frac \rmd{\rmd s}\EE_{\parat(s)}(\parau(s))-
\power {\parat(s)}{\parau(s)}\dot{\parat}(s)
 =-\dot{\rho}(\parat (s))  -  |{\parau}'|_{L^1(0,1)}(s), 
\end{equation}
where $|{\parau}'|_{L^1(0,1)}$  denotes the $L^1(0,1)$-metric
derivative of $u$, cf.\ \eqref{eq:13}.  Clearly, the
right-hand side of \eqref{not-chain-rule} is strictly smaller than
$|{\parau}'|_{L^1(0,1)}(s)$ for $s \in (2,3)$.
\end{example}

In the final example we recover the coercivity condition
\eqref{eq:17} by taking a nonzero $\beta$, with linear
growth. Nonetheless, unlike Example \ref{ex:3.a} we only require $W$
to be $\lambda$-convex: in this case, the chain-rule inequality
\eqref{eq:RIF:5bis} is still not valid.

\begin{example}
\label{ex:4.4}
\slshape We take $\Omega=(0,l)$ with $l >2$, $\beta(|\tfrac {\dd}{\dd
  x} u|)= |\tfrac {\dd}{\dd x} u|$, the double-well potential $W$
\eqref{explicit-d-well}, and $\ell(t,x) \equiv 2$ for all $(t,x)\in
[0,T]\times (0,l)$, where $0<T\leq l{-}2$. We show that the
parameterized curve $s \in [0,T] \mapsto
(\parat(s),\parau(s)):=(s,\bar{u}(s)) \in [0,T] \times L^2(0,l)$ with
  \begin{equation}
  \label{explicit-bar-u}
 \bar{u}(t,x):=
 \left\{
 \begin{array}{ll}
 6 & \text{for } 0 \leq x \leq t+1,
 \\
 -2  & \text{for } t+1 <x\leq l
 \end{array}
 \right. \quad \text{for all } (t,x) \in [0,T] \times[0,l]
\end{equation}
does not comply with the chain-rule inequality \eqref{eq:RIF:5bis}.
Note that $\bar u$ satisfies $\bar u \in \mathrm{C}^0
([0,T];L^2(0,l))\cap \mathrm{C}^{\mathrm{lip}}([0,T]; L^1(0,l) )$ with
$\|\bar u(t_1){-}\bar u(t_2)\|_{L^2(0,l)} =8|t_1{-}t_2|^{1/2}$ and
$\|\bar u(t_1){-}\bar u(t_2)\|_{L^1(0,l)} =8|t_1{-}t_2|$. The latter
implies $|{\bar u}'|_{L^1(0,l)} \equiv 8$.

To see that the chain-rule inequality \eqref{eq:RIF:5bis} does not
hold, we employ  \eqref{explicit-bar-u} to find 
\begin{equation}
\label{explicit-calcul-energy}
\begin{aligned}
\ene t{\bar{u}(t)}&= \mathcal{V}(\bar{u}(t))+ \int_0^l
(W(\bar{u}(t,x)){-}2 \bar u(t,x) ) \, \dd x
\\ & = 8 + \int_0^t (W(6){-}12)\, \dd x  + \int_t^l (W({-}2){+}4) \,
\dd x \  = \ 8 +6 l -16 t,
 \end{aligned}
\end{equation}
where we have used the notation $\mathcal{V}(u):= \int_0^l |\tfrac{\dd
}{\dd x}u|\, \dd x $ for the total variation functional on $(0,l)$.
Next we show that $\bar u$ satisfies \eqref{eq:65bis}, i.e.\ $ K^*
+ \partial\ene t{\bar u(t)} \ni 0$ for all $t\in [0,T]$. For this, we
claim that
\begin{equation}
\label{eq:xi_t}
\xi_t \in \partial \ene t {\bar u(t)} \quad \text{ with } \xi_t(x)=
\left\{
\begin{array}{ll}
\frac1{1+t} & \text{for } 0<x<t{+}1,
\\
\frac{-1}{l-1-t} & \text{for } t{+}1 < x < l.
\end{array}
 \right.
\end{equation}
To see this, we use $\mathcal V(\bar u(t))=8$ and estimate, for
general $v \in \mathrm{BV}(0,l)$, as follows:
\begin{align*}
\mathcal V(v)-\mathcal V(\bar u(t))  &\geq
\mathop{\mathrm{ess\,sup}}_{  x \in (0,l) } v -
\mathop{\mathrm{ess\,inf}}_{  x \in (0,l)  } v \, - 8 
\ \geq \ \textstyle \frac1{1+t}\int_0^{1+t} v(x) \dd x - 
   \frac1{l-1-t} \int_{1+t}^l v(x) \dd x \;-\, 8\\
&\textstyle = \int_0^l \xi_t(x) \big( v(x) - \bar u(t,x)\big)  \dd x
 \ =  \ \langle \xi_t, v{-}\bar u(t)\rangle_{L^2(0,l)}.
\end{align*}
Using the $(-1)$-convexity of $W$, we obtain, for all $v\in L^2(0,l)$,
the estimate
\[
\ene t v - \ene t {\bar u(t)} \geq \langle \xi_t , v{-}\bar u(t)
\rangle - \tfrac12 \| v{-}\bar u(t)\|_{L^2(0,l)}^2,
\]
implying \eqref{eq:xi_t}, cf.\ Definition  \eqref{eq:RIF2:16}
for Fr\'echet subdifferentials.  Because of\/ 
$0\leq t \leq T\leq l{-}2$ we have $\|\xi_t\|_{L^\infty} =
\max\{\frac1{1+t} ,\frac1{l-1-t} \} \leq 1$ for all $t\in
[0,T]$. Hence, $\xi_t \in K^*=\{\; \xi \; : \; \|\xi\|_{L^\infty}
\leq 1\;\}$, and \eqref{eq:65bis} is established.

Now returning to the notation of the parameterized solution
$(\parat(s),\parau(s))=(s,\bar{u}(s))$ for $s\in [0,T]$, we find
$\vvmV {\parat(s)}{\parau(s)}{\dot{\parau}(s)}_{ L^2(0,l) 
} \equiv 0$ on $[0,T]$. Moreover, $\power {\parat(s)}{\parau(s)}
\equiv 0 $ as well, whereas $ |{\parau}'|_{L^1(0,l)} (s) \equiv
8$. Thus, on account of \eqref{explicit-calcul-energy} we conclude
that
\[
 \frac \rmd{\rmd s}\EE_{\parat(s)}(\parau(s))-
 \power {\parat(s)}{\parau(s)}\dot{\parat}(s)= -16 \ \lneqq \ -8=
  - |{\parau}'|_{L^1(0,l)}(s) -\vvmV {\parat(s)}{\parau(s)}
 {\dot{\parau}(s)}_{ L^2(0,l)  }\,,
\]
which is a contradiction to the chain-rule inequality \eqref{eq:RIF:5bis}.
\end{example}

\section{Chain-rule inequalities for $\BV$ and parameterized curves}
\label{s:chain}
In this section we will collect the proof of the chain-rule
inequalities stated in Theorems \ref{prop:bv-chainrule} and
\ref{th:3.8}. We first consider the case of parameterized curves,
hence, using the reparameterization technique of Proposition
\ref{prop:bv} we deduce Theorem \ref{prop:bv-chainrule}.

\subsection{Chain rule for admissible parameterized curves: 
proof of Theorem \ref{th:3.8}} \label{ss:8-chain}

We split the proof in two claims.
\\[0.3em]
\textbf{Claim (1):} \emph{the map $s\mapsto \EE_{\sft(s)}(\parau(s))$
  is absolutely continuous on $[\mathsf{a},\mathsf{b}]$.}  First of
all, we observe that, since $\sup_{s\in
  [\mathsf{a},\mathsf{b}]}\ene{\sft(s)}{\sfu(s)}=: E<\infty$, by
\eqref{hyp:en-subdif} we have
$\bar{\omega}:=\sup_{r,s,\sigma}\omega_r^E
(\parau(s),\parau(\sigma))<\infty$.  We decompose the open set $G$
defined by \eqref{eq:RIF2:6} as the disjoint union of open intervals
$G_k$.  We fix $\mathsf{a}\le r\le s\le \mathsf{b}$ and we consider
the following cases:
\\[0.3em]
\textbullet\ $r,s\in [0,T]\setminus G$.  By \eqref{hyp:en3} and
estimate \eqref{gronwall-dixit} there exists a constant $C>0$
(independent of $r,s$) such that
\begin{displaymath}
      |\EE_{\sft(s)}(\parau(r))-\EE_{\sft(r)}(\parau(r))|\le C\int_r^s \dot{\sft}(\sigma)  \dd \sigma,\quad
      |\EE_{\sft(r)}(\parau(s))-\EE_{\sft(s)}(\parau(s))|\le C\int_r^s  \dot{\sft}(\sigma)  \dd
      \sigma.
\end{displaymath}
In view of \eqref{hyp:en-subdif}, for $\xi(s)\in
\partial\ene{\sft(s)}{\sfu(s)}$
fulfilling $\xi(s) \in K^*$ we have
\begin{align*}
      \EE_{\sft(s)}(\parau(s))-\EE_{\sft(s)}(\parau(r))&\le \la \xi(s),\parau(s){-}\parau(r)\ra+
      \bar{\omega} \Dnorm{\parau(s){-}\parau(r)}
      \\ & \leq \Diss{\Bo}{} (\parau(s){-}\parau(r)) +  \bar{\omega} \Diss{\Bo}{} (\parau(s){-}\parau(r))
      \leq (1+ \bar{\omega}) \int_r^s \scalardens\Psiz {\parau}(\sigma)\dd
      \sigma,
\end{align*}
where the second inequality follows from \eqref{eq:18} and the last
one from \eqref{eq:13a} and the minimal representation
$m=\scalardens\Psiz \parau$.  Analogously, arguing with
$\xi(r)\in \partial \ene {\sft(r)}{\sfu(r)}\cap K^*$, we have
$\EE_{\sft(r)}(\parau(r))-\EE_{\sft(r)}(\parau(s)) \leq
(1+\bar{\omega}) \int_r^s \scalardens \Psiz\parau(\sigma)\dd
\sigma$. All in all, we conclude
\begin{equation}
      \label{eq:RIF2:9}
  |\EE_{\sft(s)}(\parau(s))-\EE_{\sft(r)}(\parau(r))|\le
   C_1 \int_r^s \Big(\dot{\sft}(\sigma) +
   \scalardens\Psiz\parau (\sigma)\Big) \dd
      \sigma,\quad   C_1:= 2(C+1+\bar \omega).
\end{equation}
\\[0.3em]
\textbullet\ $r,s$ belong to the closure $\overline {G_k}$ of the same
connected component $G_k =(a_k,b_k)$ for some $k$.  It is not
restrictive to assume that $r,s\in G_k$.  Then $\sft \equiv
\bar{\sft}$ is constant in $G_k$ by (2) in Definition \ref{def:3.5}
and $\parau \in \AC([r,s];\V)$.   We denote by $ \frsub^\circ
\mathcal{E} : [0,T] \times D \rightrightarrows V^*$ the multivalued
map defined by
\[ 
 \xi \in  \frsub^\circ \ene tu \text{ if and only if }
 \Vnorm{\xi}_*= \min\{ \Vnorm{\zeta}_*\, :\, \zeta \in \frsub \ene tu\},
\] 
with the usual convention that the latter quantity is $+\infty$ if
$\frsub\cE_t(u)$ is empty.  Since $K^*$ is bounded in $V^*$, the
definition of $\fre_{\bar t}(u)$ in \eqref{eq:7} gives the estimate
\begin{displaymath}
      \fre_{\bar t}(\parau(\theta))\ge
      \|\partial^\circ\ene{\bar t}{\parau(\theta)}\|_*-\mathsf K,
      \quad
\text{where }\mathsf K:=\sup\{ \| z\|\, : \,z\in K^*\}, 
\end{displaymath}
and we conclude that $ \int_r^s \|\partial^\circ\ene{\bar
  t}{\parau(\theta)}\|\,\|\dot\parau(\theta)\|\,\dd\theta<\infty$.
Hence the  chain rule 
(analogous to Theorem \ref{thm-from-mrs12},
see the arguments of \cite[Theorem 1.2.5]{AGS08} and \cite[Proposition 2.4]{MRS-dne})
provides the absolute continuity the energy map in $G_k$ and for $\Leb
1$-a.a.~$\theta\in G_k$ we have
\begin{align}
      \label{eq:52first}
      \frac\dd{\dd \theta}\ene{\bar t}{\parau (\theta)}
      &=\langle \xi,\dot\parau(\theta)\rangle\quad
      \forevery \xi\in \partial\ene {\bar t}{\sfu(\theta)},\\
      \label{eq:33bis}
      \Big|\frac\dd{\dd \theta}\ene{\bar t}{\parau (\theta)}
      \Big|&
    \le \Psiz(\dot\parau(\theta))+\fre_{\bar t}(\sfu(\theta))\|\dot\parau(\theta)\|.
\end{align}
\\[0.3em]
\textbullet\ $r\in G,s\in [0,T]$ with $r<s$ (or viceversa): we denote
by $\sigma$ the right boundary point of the interval $G_k\ni r$;
combining \eqref{eq:RIF2:9} with the integrated form of
\eqref{eq:33bis} we obtain
\begin{align*}
      &\big|\EE_{\sft(s)}(\parau(s))-\EE_{\sft(r)}(\parau(r))\big|\le
      \big|\EE_{\sft(s)}(\parau(s))-\EE_{\sft(\sigma)}(\parau(\sigma))\big|
      +
      \big|\EE_{\sft(\sigma)}(\parau(\sigma))-\EE_{\sft(r)}(\parau(r))\big|
      \\&\quad\le
      C_P\int_\sigma^s \Big(\dot{\sft}(\rho) + \scalardens\Psiz{\parau}(\rho)\Big) \dd
      \rho
      + \int_{r}^\sigma\Big(\Psiz(\dot\parau(\rho))
      +\fre_{\parat(\rho)}(\parau(\rho))\|\dot\parau(\rho)\|\Big)\,\dd\rho
      \\&\quad =\int_r^s h(\rho)\,\dd \rho\quad\text{with }h\in L^1(0,T).
\end{align*}
\vspace*{0.2em}      

\noindent
\textbf{Claim (2):} \emph{the chain-rule inequality
  \eqref{eq:RIF:5bis} holds.} It follows from Claim (1) there exists a
set of full measure $\mathcal{T} \subset (a,b)$ such that for all $s
\in \mathcal{T}$ the function $\sft$ is differentiable at $s$, the
first of \eqref{hyp:en3} holds at $s$, the $\Diss{\Bo}{}$-metric
derivative $ \scalardens\Psiz\parau(s)$ exists, and, if $s\in G$, the
map $\parau$ is $\V$-differentiable at $s$.  Hence, we evaluate the
derivative of the map $\EE_{\sft(\cdot)}(\parau(\cdot))$ at $s \in
\mathcal{T}$: if $s\in \cup_k\overline{ G_k}$ we immediately get the
thesis by \eqref{eq:33bis} (notice that $\Leb 1\big((\cup_k\overline{
  G_k}) \setminus G\big)=0$).  If $s\in [0,T]\setminus \cup_k
\overline{G_k}$ then $r=s-h\in [0,T]\setminus G$ for infinitely main
values of $h>0$, accumulating at $0$.  Since
$\fre_{\sft(r)}(\sfu(r))=0$ we can choose $\xi(r)\in
-\partial\EE_{\sft(r)}(\parau(r))\cap K^*$ and thanks to
\eqref{hyp:en-subdif} we have
\begin{align}
    \notag&
    \frac{\EE_{\sft({s})}(\parau({s}))-\EE_{\sft(r)}(\parau(r))}h
    =
    \frac{\big(\EE_{\sft({s})}(\parau({s}))-\EE_{\sft({r})}(\parau({s}))\big)
      +\big(\EE_{\sft({r})}(\parau({s}))-\EE_{\sft(r)}(\parau(r))\big) }h
    \\&    \label{e:3.26}
   \ge  \la \xi({r}),\frac 1h(\parau({s})-\parau(r))\ra  -
   \frac 1h\omega_{r}(\parau({s}),\parau(r))\Dnorm{\parau({s})-\parau(r)}
     +
     \frac{\EE_{\sft(s)}(\parau(s))-\EE_{\sft(r)}(\parau(s))}h
     \\
   \notag &\ge
   -\frac{1+\omega_{r}(\parau({s}),\parau(r))}h\Psiz(\parau({s})-\parau(r))
   +\frac 1h\int_r^s \power{\parat(\theta)}{\parau(s)}\dot\parat(\theta)\,\dd \theta
\end{align}
In the limit $r\up s$, with $r\in [0,T]\setminus G$, we get
the lower bound
$ \frac \rmd{\rmd s}\EE_{\parat(s)}(\parau(s))-\power
{\parat(s)}{\parau(s)}\dot{\parat}(s)\ge -\scalardens\Psiz{\parau}(s).
$ 
The corresponding upper bound can be obtained by choosing $r=s+h$,
$h>0$, in \eqref{e:3.26}, and passing to the limit as $r\down s$.

Whenever $\sfu$ is differentiable $\Leb 1$-a.e., the chain rule
\eqref{eq:4} follows from \eqref{eq:52first} and \eqref{e:3.26} by a
similar argument. Hence, Theorem \ref{th:3.8} is proved.
\qed\medskip

By applying Theorem \ref{th:3.8} to the parameterized curve $[0,1]\ni
r\mapsto (t,\vartheta(r))$ associated with any admissible transition
$\vartheta\in \calT_t(u_0,u_1)$ we immediately have the desired jump estimates.

\begin{corollary}\label{cor:F1}
  The jump estimates \eqref{eq:36} and \eqref{eq:35first} hold true.
\end{corollary}

\subsection{Chain rule for $\BV$ curves: proof of Theorem \ref{prop:bv-chainrule}}
\label{ss:chain-bv}

It is clearly not restrictive to assume $t_0=0,\ t_1=T$.
If $u\in \BV([0,T];D_E,\Psiz)$
satisfies the local stability condition and
$\pVar\frf u0T<\infty$ as in the statement of the Theorem,
we apply assertion  (BVP3) of Proposition \ref{prop:bv}: the chain-rule inequality
 \eqref{ch-rule-ineq} follows then by
 the parameterized chain rule
 \eqref{eq:RIF:5bis}, combined with
 \eqref{eq:95tris} and
\eqref{eq:97}.

Let us now check \eqref{eq:86} in the case  $u\in \BV([0,T];V)$.
We will use the simpler change of variable formula
\begin{equation}
\begin{gathered}
  \mathsf s(t):=t+\Var{} u0t,\quad \mathsf S:=\sfs(T),\quad
\end{gathered}\label{eq:43bis}
\end{equation}
keeping the same notation as in  \eqref{eq:92} for $\sft$, $\sfu$, $I_n$, and $I$. 
We will use two basic facts: the first property
concerns the diffuse part $\sfs_\dd'$ of the distributional
derivative of $\sfs$ and
has been proved in \cite[Prop.\,6.11]{MRS10}
(the proof does not rely on the finite-dimensional setting  therein considered), namely
\begin{equation}
  \label{eq:90}
  u_{\rm d}'=\nn\|u_{\rm d}'\|=(\dot \sfu\circ\sfs)\,\sfs_{\rmd}',\quad
  \Leb 1_{(0,T)}=(\dot \sft\circ\sfs)\,\sfs_{\rmd}'.
\end{equation}
The second fact is a general property of the distributional
derivative
of an increasing map, viz.
\begin{equation}
  \label{eq:91}
  \sft_\sharp \big(\Leb 1_{[0,\sfS]}\big)=\sfs_{\rm d}'.
\end{equation}
We set
\begin{displaymath}
  \sfe(s):=
  \begin{cases}
    e(\sft(s))=\ene{\sft(s)}{\sfu(s)}&\text{if }s\in (0,\mathsf{S}) \setminus I, \\
     \text{affine interpolation of }  e(t_{n-}),e(t_{n+}) 
    &\text{if }s\in I_n \text{ for some }n\in \N,
  \end{cases}
\end{displaymath}
and we extend in a similar way $\sfs $ in each interval $I_n$. Now
$\sfu $ defined by \eqref{eq:92} is absolutely continuous and
arguing as in \S\,\ref{ss:8-chain} we can easily prove that
$\sfe$ is absolutely continuous with derivative
\begin{equation}
  \label{eq:93bis}
  \dot \sfe(s)=-\langle \xi(\sft(s)),\dot\sfu(s)\rangle
  +\power {\sft(s)}{\sfu(s)}\dot\sft (s)\quad\text{for $\Leb
    1$-a.a.~$s\in (0,\mathsf{S}) $.} 
\end{equation}
On the other hand $\sfe(s)=e(\sft(s))$ whenever $s\in
[0,\sfS]\setminus \bigcup\overline{I_n}$. Since $\sft(s)\equiv t_n$
and $\dot \sft(s)\equiv 0$ in $I_n$ we obtain
$e(\sft(s))\dot\sft(s)=\sfe(s)\dot\sft(s)$ for a.a.~$s\in (0,\sfS)$.
Hence, for every $\zeta \in \rmC^1([0,T])$ with compact support in
$(0,T)$ we obtain
\begin{align*}
  \int_{[0,T]} &\zeta(t)\,\dd e_{\rm d}'(t)=
  -\int_0^T \dot \zeta(t)e(t)\,\dd t-\sum_{t\in \mathrm J(u)} \zeta(t)
  (e(t_+)-e(t_-))
  \\&=
   -\int_0^\sfS \dot\zeta(\sft(s))e(\sft(s))\dot\sft(s)\,\dd s-\sum_{n}
   \zeta(t_n)
   (e(t_{n+})-e(t_{n-}))
\\&=
   -\int_0^\sfS \dot\zeta(\sft(s))\sfe(s)\dot\sft(s)\,\dd s-\sum_{n}
   \zeta(t_n)
   (\sfe(\sfs(t_{n+}))-e(\sfs(t_{n-})))
   \\&=
   \int_0^\sfS \zeta(\sft(s))\dot \sfe(s)\,\dd s-\sum_{n}
   \int_{I_n}\zeta(\sft(s)) \dot \sfe(s)\,\dd s
   =
   \int_{[0,\sfS]\setminus I} \zeta(\sft(s))\dot \sfe(s)\,\dd s
   \\&\topref{eq:93bis}=
   -\int_{[0,\sfS]\setminus I} \zeta(\sft(s)) \langle \xi(\sft(s)),
   \dot    \sfu(s)
   \rangle \dd s+
   \int_{[0,\sfS]\setminus I} \zeta(\sft(s)) \power {\sft(s)}{\sfu (s)}\dot\sft(s)\,\dd s
   \\&\topref{eq:91}=
   \int_{[0,T]\setminus \mathrm J_u} \zeta(t)
   \Big(-\langle \xi(t),\dot \sfu(\sfs(t))\rangle+
   \power t{u(t)}\dot \sft(\sfs(t))\Big)
   \,\dd \sfs_{\rm d}'(t)
   \\&\topref{eq:90}=
   -\int_{[0,T]\setminus \mathrm J_u} \zeta(t) \langle \xi(t),
   \nn\rangle \,\dd\|u_{\rm d}'\|(t)   +\int_0^T \zeta(t)\power t{u(t)}\,\dd t.
 \end{align*}
Since the measure  $\|u_\dd'\|$ does not charge $\mathrm J_u$, we get \eqref{eq:86},
and Theorem \ref{prop:bv-chainrule} is proved. 
\qed\medskip

\section{Convergence proofs for the viscosity approximations}
\label{s:last}

\subsection{Compactness and lower semicontinuity result for
  parameterized curves}

We first provide a lower semicontinuity result that will be used to
prove Theorems \ref{thm:fcost}, \ref{th:1}, and \ref{th:2} in the next
subsections.
\label{ss:last1}
\begin{proposition}
  \label{le:compactness}
  Let $E,L>0$ and for every $n\in \N$ let $\sft_n\in
  \AC(\sfa,\sfb;[0,T])$ be nondecreasing. Assume that $\tilde\sfu_n
  :[\sfa,\sfb]\to D_E$ are measurable, $G_n\subset [\sfa,\sfb]$ are
  open (and possibly empty) subsets such that
  $\fre_{\sft_n(s)}(\tilde\sfu_n(s))=0$ in $[\sfa,\sfb]\setminus G_n$,
  $\sfu_n\in \AC ([\sfa,\sfb];V,\Psiz)\cap \AC_{\rm loc}(G_n;V)$, and
  there holds
  \begin{subequations}
    \begin{align}
      \label{eq:108}
      X_n := \sup_{s\in [\sfa,\sfb]}\|\sfu_n(s)-\tilde\sfu_n(s)\|
      &\to 0\quad \text{as }n\to\infty
      \\
      \label{eq:44}
      \dot\sft_n(s)+\scalardens\Psiz{\sfu_n}(s)+\fre_{\sft_n(s)}(\tilde\sfu_n(s))
      \|\dot\sfu_n(s)\|&\le L\quad\text{for $\Leb 1$-a.a.~$s\in
        (\sfa,\sfb)$},
\end{align}
\end{subequations}
where we adopt the convention $\fre_{\sft_n(s)}(\tilde\sfu_n(s))
\|\dot\sfu_n(s)\| \equiv 0$ if $s\not\in G_n$, as in \eqref{eq:51}.

Then there exist a subsequence (not relabeled) and a limit function
$(\sft,\sfu)\in \mathscr A(\sfa,\sfb;[0,T]\times D_E)$ such that
$(\sft_n,\sfu_n)\to(\sft,\sfu)$ uniformly in $[\sfa,\sfb]$ with
respect to the topology of $[0,T]\times V$. Moreover $(\sft,\sfu)$
satisfies the same bound \eqref{eq:44} and the following asymptotic
properties hold as $n \to \infty$:
\begin{align}
    \label{eq:46}
    \liminf_{n\to\infty} \int_\sfa^\sfb
    \scalardens\Psiz{\sfu_n}(s)\,\dd s&\ge
    \int_\sfa^\sfb \scalardens\Psiz{\sfu}(s)\,\dd s,\\
    \label{eq:47}
   \liminf_{n\to\infty} \int_\sfa^\sfb  \fre_{\sft_n(s)}(\tilde\sfu_n(s))
    \|\dot\sfu_n(s)\|\,\dd s&\ge
    \int_\sfa^\sfb \fre_{\sft(s)}(\sfu(s)) \|\dot\sfu(s)\|\,\dd s,
\\\label{eq:48bis}
    \liminf_{n\to\infty}\int_\sfa^\sfb
    \Big(\frk_{\sft_n(s)}(\tilde \sfu_n(s))\dot \sft_n(s)+
    \fre_{\sft_n(s)}(\tilde\sfu_n(s)) \|\dot\sfu_n(s)\| 
 \Big)\,\dd s
 &\ge
    \int_\sfa^\sfb
    \mathfrak{G}[\sft,\sfu;\dot\sft,\dot\sfu](s)\,\dd s.
    \intertext{If, moreover, $\sfu_n\in \AC([\sfa,\sfb];V)$, then}
    \label{eq:49}
    \liminf_{n\to\infty}\int_{G_n}
    \mathfrak{G}_{\eps_n}(\sft_n(s),
    \tilde\sfu_n(s);\dot\sft_n(s),\dot\sfu_n(s))\,\dd s&\ge
    \int_\sfa^\sfb
    \mathfrak{G}[\sft,\sfu;\dot\sft,\dot\sfu](s)\,\dd s
\end{align}
for every vanishing sequence $(\eps_n)_{n}\subset (0,\infty)$.
\end{proposition}

\noindent
 We will use later that  the assumptions of Proposition
\ref{le:compactness} cover the case $(\sft_n,\sfu_n)\in \mathscr
A(\sfa,\sfb;[0,T]\times V)$ with $\tilde\sfu_n=\sfu_n$.
\begin{proof}
  By \eqref{eq:44} the sequence
  $\sft_n$ is uniformly Lipschitz, thus relatively compact
  with respect to uniform convergence.

  Let $C_\Psiz$ be the continuity constant of $\Psiz$ and
  $\Omega:=\Omega_{D_E}$ be the modulus of continuity from
  \eqref{eq:121}: since $\Omega$ is concave and $\Omega(0)=0$ we have
  \begin{equation}
    \label{eq:122}
    \Omega(\lambda p)\le \lambda \Omega(p),\quad
    \Omega(p+q)\le \Omega(p)+\Omega(q)
    \quad\forall\ \lambda,p,q\ge0.
  \end{equation}
  Since every curve $\tilde\sfu_n$ takes values in the compact set
  $D_E$, we have in view of \eqref{eq:120} that 
  \begin{align}
    \notag\|\tilde\sfu_n(s)-\tilde\sfu_n(r)\|&\le
    \Omega\big(\Dnorm{\tilde\sfu_n(s)-\tilde\sfu_n(r)}\big)
    \le
    \Omega\big(\Psiz(\sfu_n(s)-\sfu_n(r))\big)+2C_\Psiz\Omega(X_n)
    \\&\le L\,\Omega(|s-r|)+2C_\Psiz\Omega(X_n)
    \label{eq:124}
  \end{align}
  It follows from \eqref{eq:108} that
  \begin{displaymath}
    \limsup_{n\to\infty}\|\tilde\sfu_n(s)-\tilde\sfu_n(r)\|
    \le L\,\Omega(|s-r|).
\end{displaymath}
Thus $\tilde\sfu_n$ is (asymptotically) uniformly equicontinuous and
we can apply the Arzel\`a-Ascoli Theorem (in a slightly refined form,
see e.g.~\cite[Prop.\,3.3.1]{AGS08}) to prove its uniform convergence
to a limit $\sfu$.  Passing to the limit in \eqref{eq:44} we get an
analogous estimate for $(\sft,\sfu)$.

Statement \eqref{eq:46} is an immediate consequence of the lower
semicontinuity of the $\Psiz$-total variation and of its
representation formula \eqref{eq:24}.

In order to prove \eqref{eq:47} let us observe that the lower
semicontinuity property of the map $\fre$ and the above uniform
convergence guarantee that the limit function $s\mapsto
\fre_{\sft(s)}(\sfu(s))$ is lower semicontinuous.  Thanks to
\eqref{eq:108} we can find a set $\sfM\subset [\sfa,\sfb]$ with $\Leb
1([\sfa,\sfb]\setminus \sfM)=0$ such that $\tilde\sfu_n$ converges
uniformly to $\sfu$ in $\sfM$ and
\begin{equation}
    \label{eq:50bis}
    \forall\eta>0\ \exists \bar n\in \N:\quad
    \mathsf
    \fre_{\sft_n(s)}(\tilde\sfu_n(s))\ge
    \fre_{\sft(s)}(\sfu(s))-\eta\quad
    \forevery n\ge\bar n,\ s\in \sfM.
\end{equation}
If $G$ is defined as in \eqref{eq:RIF2:6} and $[\alpha,\beta]\subset
G$, \eqref{eq:50bis} implies that there exists a positive constant
$c>0$ with $\fre_{\sft_n(s)}(\tilde\sfu_n(s))\ge c$ for $\Leb
1$-a.a.~$s\in (\alpha,\beta)$ and $n$ sufficiently big. Estimate
\eqref{eq:44} then yields that $\sfu_n$ are uniformly $V$-Lipschitz in
$[\alpha,\beta]$ so that $\sfu$ is also Lipschitz, and therefore $\Leb
1$-a.e.~differentiable.  Since $[\alpha,\beta]$ is arbitrary, we
conclude that $\sfu$ is locally absolutely continuous in $G$, and the
following Lemma \ref{le:jointlsc} yields the $\liminf$ inequality
\eqref{eq:47}.

Recalling  definitions \eqref{vvmfullname-def} and
\eqref{eq:34} for $\frG$ and $\frk$,  assertion
\eqref{eq:48bis} follows if we check that
\begin{displaymath}
    \liminf_{n\to\infty}\int_\sfa^\sfb\frk_{\sft_n(s)}(\tilde\sfu_n(s))
    \dot \sft_n(s)\,\dd s\ge
    \int_\sfa^\sfb\frk_{\sft(s)}(\sfu(s))\dot \sft(s)\,\dd s,
\end{displaymath}
which is again a consequence of Lemma \ref{le:jointlsc} ahead.

In order to prove \eqref{eq:49} let us observe that
$\frG_\eps(\sft,\tilde\sfu;\alpha,\sfv)\ge \max\big\{\frac \alpha\eps
F^*(\fre_{\sft}(\tilde\sfu)),\, \fre_\sft(\tilde\sfu)\|\sfv\|\big\}$.
 Splitting the integration domain into $(\sfa,\sfb)\setminus G$
and $G$  a further application of Lemma \ref{le:jointlsc} yields
\begin{align*}
    \liminf_{n\to\infty}
    &\int_\sfa^\sfb
    \mathfrak{G}_{\eps_n}(\sft_n(s),\tilde\sfu_n(s);\dot\sft_n(s),\dot\sfu_n(s))\,\dd s
    \\&\ge
      \liminf_{n\to\infty}\int_{(\sfa,\sfb)\setminus G}
      \frac 1{\eps_n} F^*(\fre_{\sft_n(s)}(\tilde\sfu_n(s))) \dot\sft_n(s)\,\dd s
      +\liminf_{n\to\infty}\int_G
      \fre_{\sft_n(s)}(\tilde\sfu_n(s))
    \|\dot\sfu_n(s)\|\,\dd s
    \\&\ge
    \int_{(\sfa,\sfb)\setminus G}
    \frk_{\sft(s)}(\sfu(s))\,\dot\sft(s)\,\dd s
      +\int_G
      \fre_{\sft(s)}(\sfu(s))\,
      \|\dot\sfu(s)\|\,\dd s
      =
      \int_\sfa^\sfb
    \mathfrak{G}_{}[\sft,\sfu;\dot\sft,\dot\sfu](s)\,\dd s\,.
\end{align*}
This concludes the proof of Proposition \ref{le:compactness}.   
\end{proof}

A simple proof of the following lemma can be found, e.g., in
\cite[Lem.\,4.3]{MRS12}. 

\begin{lemma}
  \label{le:jointlsc}
  Let $I$ be a measurable subset of $\R$ and let $h_n,h,m_n,m:I\to
  [0,+\infty]$ be measurable functions for $n\in \N$ that satisfy
  \begin{equation}
    \label{eq:51bis}
    \liminf_{n\to\infty}h_n(x)\ge h(x)\quad\text{for $\Leb 1$-a.a.~$x\in I$,}\quad
    m_n\weakto m\quad\text{in }L^1(I).
  \end{equation}
  Then
  \begin{equation}
    \label{eq:52}
    \liminf_{n\to\infty}\int_I h_n(x)m_n(x)\,\dd x\ge \int_I h(x)m(x)\,\dd x.
  \end{equation}
\end{lemma}

\subsection{Compactness and lower semicontinuity for
non-parameterized curves}
\label{ss:clsc}

\subsubsection*{\bfseries Proof of Theorem \ref{thm:fcost}}
To address assertion (F2) let $\vartheta_n\in
\calT_t(u_{0,n},u_{1,n})$ be a sequence of admissible transitions such
that
\begin{equation}
    \label{eq:54}
    \int_0^1 \frf_t[\vartheta_n;\vartheta_n'](r)\,\dd r\le
    \Delta_{\frf_t}(u_{0,n},u_{1,n})+\eps_n\quad 
    \text{with } \eps_n\ge0 \text{ and } \lim_{n\to\infty}\eps_n=\eps\ge0.
\end{equation}
By operating the change of variable
\begin{displaymath}
    \sfs_n(r):=c_n\Big(r+\int_0^r
    \frf_t[\vartheta_n;\vartheta_n'](w)\,\dd w\Big),\quad 
    \sfr_n:=\sfs_n^{-1}:[0,\sfS]\to[0,1],\quad
    \sfu_n:=\vartheta_n\circ \sfr_n:[0,\sfS]\to V,
\end{displaymath}
where $c_n$ is a normalization so that $\sfS:=\sfs_n(1)$ is
independent of $n$, we see that the functions $\sfr_n$ are
uniformly Lipschitz and the curve $s\mapsto (\sfr_n(s),\sfu_n(s)) $
satisfies \eqref{eq:108}--\eqref{eq:44} with
$\tilde\sfu_n\equiv\sfu_n$.

We can thus extract subsequences (still denoted by $\sfr_n,\sfu_n$)
converging uniformly to $\sfr,\sfu$ respectively. The previous
Proposition \ref{le:compactness} guarantees that $\sfu$ is an
admissible transition connecting $u_-$ to $u_+$ and $\liminf$
inequalities \eqref{eq:46} and \eqref{eq:47} show that
\begin{displaymath}
    \eps+\Delta_{\frf_t}(u_-,u_+)\ge\liminf_{n\to\infty} \int_0^1
    \frf_t[\vartheta_n;\vartheta_n'](r) \,\dd r \geq
    \int_0^\sfS \frf_t[\sfu;\sfu'](r)\,\dd r\ge \Delta_{\frf_t}(u_-,u_+).
\end{displaymath}
This proves the lower semicontinuity of the Finsler cost
functional. Since we may choose $0 <\eps_n \to \eps=0 $, the previous
inequalities yield that $\sfu$ attains the infimum in \eqref{eq:69},
so that also assertion (F1) is proved, since the jump
estimate \eqref{eq:36} has been proved in Corollary \ref{cor:F1}.

Let us now consider the last assertion (F3): it is not restrictive to
assume $u_-\neq u_+$ so that $\Delta\ge \Psiz(u_+-u_-)>0$.  For $r\in
[0,\beta_n-\alpha_n]$ we set
\begin{align*}
    \sfs_n(r):={}&c_n\Big(r+\int_{\alpha_n}^{\alpha_n+r}\Psiz_{\eps_n}(u_n(\zeta))+
    \Psiz_{\eps_n}^*(\xi_{n}(\zeta))\,\dd \zeta\Big),\\
    \sft_n:={}&\sfs_n^{-1}:[0,1]\to [\alpha_n,\beta_n],\quad
    \sfu_n:=u_n\circ \sft_n,\quad
    \tilde\sfu_n:=\tilde u_n\circ \sft_n:[0,1]\to V,
\end{align*}
where $c_n$ is a normalization constant such that
$\sfs_n(\beta_n-\alpha_n)=1$.  Again, it is not difficult to see that
the triple $(\sft_n,\sfu_n,\tilde\sfu_n)$ satisfies the assumptions of
Proposition \ref{le:compactness}.  Moreover
\begin{equation}
    \label{eq:53bis}
    \int_{\alpha_n}^{\beta_n}\Big(\Psiz_{\eps_n}(u_n(r))+
    \Psiz_{\eps_n}^*(\xi_n(r))\Big)\,\dd r=
    \int_0^\sfS \Big(\Psiz(\dot\sfu_n (s))+ \frG_{\eps_n}(\sft_n (s),
              \tilde\sfu_n (s) ;\dot\sft_n (s),\dot\sfu_n (s) )\Big)\,\dd s.
\end{equation}
We can thus apply Proposition \ref{le:compactness} to pass to the
limit obtaining an admissible limit curve $(\sft,\sfu)\in \mathscr
A(0,1;[0,T]\times D_E)$ such that $\sft(s)\equiv t$, $\sfu(0)=u_-$ and
$\sfu(1)=u_+$.  In particular $\sfu\in \calT_t(u_-,u_+)$ and combining
\eqref{eq:53bis} with \eqref{eq:49} we get
\begin{align*}
    \Delta &=
    \lim_{n\to\infty}
    \int_0^1
    \Big(\Psiz(\dot\sfu_n(s))+
    \frG_{\eps_n}(\sft_n(s),\sfu_n(s);\dot\sft_n(s),\dot\sfu_n(s))\Big)\,\dd s
    \ge
    \int_0^1
    \Big(\scalardens \Psiz\sfu(s)+
    \frG[t,\sfu;0,\dot\sfu](s)\Big)\,\dd s
    \\&=
    \int_0^1 \Big(\scalardens \Psiz\sfu(s)+
    \fre_t(\sfu(s))\|\dot\sfu(s)\|\Big)\,\dd s
    \ge \Cost\frf t{u_-}{u_+}.
\end{align*}
This concludes the proof of Theorem   \ref{thm:fcost}. \qed\bigskip

The next result is a counterpart to Proposition \ref{le:compactness}
for the lower semicontinuity, but now for the non-parameterized
setting.
\begin{proposition}
  \label{cor:2}
  Let $E,C>0$  and
  for $n\in \N$ let $u_n\subset \AC([0,T];V)$, $\tilde
  u_n:[0,T]\to D_E$, $\xi_n\to[0,T]\to V^*$ measurable,
  $\eps_n\in (0,\infty)$ be sequences satisfying
  \begin{subequations}
          \label{eq:55}
    \begin{gather}
      \label{eq:115}
      \int_0^T \Big(\Psi_{\eps_n}(\dot u_{n})+
        \Psi_{\eps_n}^*(\xi_{n})\Big)\,\dd t\le C,\quad \xi_{n}(t)\in
        -\partial\ene t{\tilde u_{n}(t)} \quad\text{for $\Leb 1$-a.a.\
          $t\in (0,T)$}, \\
      \label{eq:117}  X_n:=\sup_{t\in [0,T]}
      \|u_n(t)-\tilde u_n(t)\|\to 0, \quad
      \eps_n\down0\qquad \text{as }n\up\infty.
    \end{gather}
  \end{subequations}Then there exists a subsequence (not relabeled) and a limit function $u\in \BV([0,T];D_E,\Psiz)$
  such that convergence  \eqref{e:conv1} holds, $u$ satisfies the local stability
  condition \eqref{eq:65bis}, and
  \begin{equation}
    \label{eq:56}
    \liminf_{n\to\infty}  \int_r^s \Big(\Psi_{\eps_n}(\dot u_{\eps_n}(t))+\Psi_{\eps_n}^*(\xi_{\eps_n}(t))\Big)\,\dd t
    \ge
    \pVar\frf u rs\quad
    \forevery 0\le r<s\le T.
  \end{equation}
\end{proposition}
\begin{proof}
  To  obtain a pointwise convergent subsequence,  we
  proceed as in the proof Proposition \ref{le:compactness}.  Setting
  $\mathrm V_n(t):=\int_0^t \Psi_{\eps_n}(\dot u_n)\,\dd r$ and using
  $\tilde u_n(t)\in D_E$ we get a similar estimate as in \eqref{eq:124}:
  \begin{equation}
    \label{eq:125}
    \|\tilde\sfu_n(t)-\tilde\sfu_n(s)\|
    \le
    \Omega\big(\mathrm V_n(t){-}\mathrm V_n(s)\big)+2C_\Psiz\Omega(X_n)\quad
    \forevery 0\le s<t\le T,\ n\in \N.
  \end{equation}
  Since the functions $\mathrm V_n$ are increasing and uniformly
  bounded by $C$, by Helly's Theorem we can extract a subsequence (not
  relabeled) pointwise converging to the increasing function $\mathrm
  V$; passing to the limit in \eqref{eq:125} along such a subsequence,
  we obtain
  \begin{equation}
    \label{eq:129}
    \limsup_{n\to\infty}\|\tilde\sfu_n(t)-\tilde\sfu_n(s)\|\le
    \Omega\big(\mathrm V(t){-}\mathrm V(s)\big).
  \end{equation}
  Applying the compactness result \cite[Prop.\,3.3.1]{AGS08} we obtain
  the pointwise convergence of (a subsequence of) $\tilde u_n$ and
  thus \eqref{e:conv1} follows by \eqref{eq:117}.

  By the strong-weak closedness \eqref{eq:45} of the graph of
  $(\mathcal{E},\partial \mathcal{E})$ we have
  \begin{displaymath}
    \liminf_{n\to\infty} \Psi_{\eps_n}^*(\xi_{\eps_n}(t))\ge
    \frk_t(u(t))\quad \text{for $\Leb1$-a.a.~$t\in (0,T)$}.
  \end{displaymath}
  Therefore Fatou's lemma yields $\int_0^T \frk_t(u(t))\,\dd t<\infty$. As
  $\frk_t(u(t)) \in \{0,+\infty\}$,  we arrive at
  \begin{displaymath}
    \frk_t(u(t))=0\quad\text{for $\Leb 1$-a.a.~$t\in (0,T)$}.
  \end{displaymath}
  Since $\frk$ is a lower semicontinuous function we conclude that
  $\frk_t(u(t))=0$ for every $t\in [0,T]\setminus \mathrm J_u$ and
  also $\frk_t(u(t_\pm))=0$ whenever $t\in \mathrm J_u$. Thus $u$ satisfies the
  local stability condition \eqref{eq:65bis}.

  To prove \eqref{eq:56} let us introduce the nonnegative and bounded
  Borel measures $\nu_n$ in $[0,T]$ via
  \begin{equation}
    \label{eq:57}
    \nu_n:= \Big(\Psi_{\eps_n}(\dot u_{n})+\Psi_{\eps_n}^*(\xi_{n})\Big)\Leb 1.
  \end{equation}
  Possibly extracting a further subsequence, it is not restrictive to
  assume that $\nu_n\weakto^*\nu$ in duality with $\mathrm
  C^0([0,T])$. Since $\nu_n\ge \Psiz(\dot u_n)\Leb 1$ for every
  interval $(\alpha,\beta)\subset [0,T]$,
  \begin{displaymath}
    \nu([\alpha,\beta])\ge \limsup_{n\to\infty} \int_\alpha^\beta \Psiz(\dot u_n)\,\dd t
    \ge \liminf_{n\to\infty} \Var\Psiz{u_n}\alpha\beta\ge
    \Var\Psiz{u}\alpha\beta\ge \scalarmuco {}u([\alpha,\beta])
  \end{displaymath}
  which in particular yields $\nu\ge \scalarmuco{}u$ (with $\mu$ from
  \eqref{eq:14}).

  Let us now take $t\in \mathrm J_u$ and two
  sequences $\alpha_n\uparrow t$ and $\beta_n\downarrow t$ such that
  \begin{gather*}
    \lim_{n\to\infty} u_n(\alpha_n)=u(t_-),\quad
    \lim_{n\to\infty} u_n(\beta_n)=u(t_+).
  \end{gather*}
  Applying assertion (F3) of Theorem \ref{thm:fcost}
  and the upper semicontinuity property of weak$^*$ convergence
  of measures on closed sets, we get
  \begin{align}
    \notag
    \nu(\{t\})&\ge \limsup_{n\to\infty}\nu_n([\alpha_n,\beta_n])
    \ge \liminf_{n\to\infty}
    \int_{\alpha_n}^{\beta_n}\Big(\Psi_{\eps_n}(\dot
    u_{n})+\Psi_{\eps_n}^*(\xi_{n})\Big)\,\dd t
    \\&\ge
    \label{eq:70}
    \Cost\frf t{u(t_-)}{u(t_+)}=
    \scalarmuj {}u({\{t\}}),
  \end{align}
  and similarly
  \begin{equation}
    \label{eq:61}
    \limsup_{n\to\infty}\nu_n([\alpha_n,t])\ge
    \Cost\frf t{u(t_-)}{u(t)},\quad
    \limsup_{n\to\infty}\nu_n([t,\beta_n])\ge
    \Cost\frf t{u(t)}{u(t_+)}.
  \end{equation}
  It follows from \eqref{eq:70} that $\nu\ge \scalarmu {}u$.
  If now $0\le r<s\le T$ we can choose $r_n>r$ and $s_n<s$ such that
  $r_n\down r$ with $u_n(r_n)\to u(r_+)$ and
  $s_n\up s$ with $u_n(s_n)\to u(s_-)$. Eventually we have
  \begin{align*}
    \liminf_{n\to\infty} & \int_r^s \Big(\Psi_{\eps_n}(\dot
    u_{n})+\Psi_{\eps_n}^*(\xi_{n})\Big)\,\dd t
    \geq \liminf_{n\to\infty} \nu_n([r,r_n])+
    \liminf_{n\to\infty} \nu_n((r_n,s_n))+
    \liminf_{n\to\infty} \nu_n([s_n,s])
    \\&
    \ge \Cost\frf r{u(r)}{u(r_+)}+\nu((r,s))+
    \Cost\frf s{u(s_-)}{u(s)}
    \\&\ge
    \Cost\frf r{u(r)}{u(r_+)}+\scalarmu{}u((r,s))+
    \Cost\frf s{u(s_-)}{u(s)}
   \topref{eq:81bis}=
    \pVar\frf u rs.
    \qedhere
  \end{align*}
\end{proof}

\subsection{Convergence of  the vanishing-viscosity  approximations}
 \label{ss:8-vanvisc}

Here we prove Theorem  \ref{th:1}, which states that the limit $u$
of solutions $u_\eps$ to the doubly nonlinear equations $(\ref{viscous-dne})_\eps$
are Balanced Viscosity (\BV) solutions. 

\subsubsection*{\bfseries Proof of  Theorem \ref{th:1}.}
Let $(u_\eps)_\eps \subset \AC ([0,T];\V)$ be a family of solutions to
\eqref{viscous-dne} fulfilling \eqref{conve-initi-data} at $t=0$: in
particular, $E_0:=\sup_\eps \Psiz(u_\eps(0))+\ene
0{u_\eps(0)}<\infty$.

We combine the energy identity \eqref{eq:52bis}, written for $s=0$
and for any $t \in (0,T]$, with the estimate for $\mathcal P_t$ in
\eqref{hyp:en3}, obtaining
\begin{displaymath}
\begin{aligned}
\Psiz(u_\eps(t))&+ \ene {t}{u_\eps(t)}\le
\Psiz(u_\eps(0))+\int_{0}^{t}
    \Big(\Psi_\eps(\dot{u}_\eps(r)){+}\Psi_\eps^*(\xi_\eps(r)) \Big)\,\dd r + \ene {t}{u_\eps(t)}
    \\& =\Psiz(u_\eps(0))+\ene
    {0}{u_\eps(0)} + \int_0^t \power {r}{u_\eps(r)} \, \dd r
    \leq E_0 + C_P \int_0^t \Big(\Psiz(u_\eps(r))+\ene{r}{u_\eps(r)}\Big) \, \dd r.
    \end{aligned}
\end{displaymath}
Applying a standard version of the Gronwall Lemma (cf.\ e.g.\
\cite[Lem.\,A.4]{brezis73}), we deduce that there exist constants
$E,C>0$ such that for every $\eps>0$ and $t\in [0,T]$ we have
\[
\Psiz(u_\eps(t))+ \ene {t}{u_\eps(t)} \le E:=E_0\exp(C_PT)\quad \text{and}\quad
 \int_{0}^{T}
 \Big(\Psi_\eps(\dot{u}_\eps(r)){+}\Psi_\eps^*(\xi_\eps(r)) \Big)\,\dd r\leq C.
\]
By Proposition \ref{cor:2}, for every vanishing sequence
$(\eps_k)_{k}$ there exists a further subsequence and $u \in
\BV([0,T];D_E,\Psiz)$ such that convergence \eqref{e:conv1} holds. By
lower semicontinuity, we also have
\begin{equation}
\label{lsc-ene} \liminf_{k \to \infty} \ene t{u_{\eps_k}(t)} \geq
\ene t{u(t)} \quad \text{for all } t \in [0,T].
\end{equation}
Furthermore, by \eqref{hyp:en3} we have $|\power t{u_{\eps_k}(t)}|\leq
C_P E$ for all $k \in \N$ and $t \in [0,T]$.  Therefore, applying
Fatou's Lemma we obtain
\begin{equation}
\label{lsc-ene-partialt}
\limsup_{k \to \infty} \int_s^t \power r{u_{\eps_k}(r)} \dd r \leq \int_s^t \power r{u(r)} \dd r
\quad \text{for all } 0 \leq s \leq t \leq T.
\end{equation}
We can now pass to the limit in the energy identity \eqref{eq:52bis}
as $k \to \infty$.  Combining \eqref{eq:56} $r=0$ and $s=T$ with
\eqref{lsc-ene}, we immediately get \eqref{eq:84-oneside}. We thus
deduce that $u$ is a \BV\ solution.

The energy identity \eqref{eq:84} satisfied by $u$
on the interval $[0,T]$ and the elementary property
of real sequences
\begin{equation}
  \label{eq:119}
  a,b\in \R,\quad
  \left\{\begin{aligned}
    \liminf_{n\to\infty}a_n&\ge a\\
    \liminf_{n\to\infty}b_n&\ge b,
  \end{aligned}\right.
\quad
  \limsup_{n\up\infty}(a_n+b_n)\le a+b\quad
  \Longrightarrow\quad
  \left\{
    \begin{aligned}
      \lim_{n\to\infty}a_n&= a\\
      \lim_{n\to\infty}b_n&= b,
    \end{aligned}\right.
\end{equation}
yield that
\begin{equation}
  \label{eq:73}
  \lim_{k\to\infty}\ene{T}{u_{\eps_k}(T)}=\ene T{u(T)},\quad
  \lim_{k\to\infty}\int_{0}^{T}
    \Big(\Psi_{\eps_k}(\dot{u}_{\eps_k}){+}\Psi_{\eps_k}^*(\xi_{\eps_k})
    \Big)\,\dd r
    =\pVar\frf u0T.
\end{equation}
A further application of \eqref{eq:56} on the intervals $[0,t]$ and
$[t,T]$ combined with
\eqref{lsc-ene}, the additivity of the total variation, and
\eqref{eq:119} provides convergences \eqref{e:conv2} and \eqref{e:conv3}. Hence,
Theorem \ref{th:1} is proved. 
\qed

\subsubsection*{\bfseries Convergence of the discrete viscous approximations.}
Let us  consider the time-incremental minimization
problem \eqref{eq:58}, giving rise to the  discrete solutions
$(U^n_{\tau,\eps})_{n=1}^{N}$ which fulfill the \emph{discrete Euler
equation}
\begin{equation}
\label{discrete-Euler}
\partial\Diss{\Bo}{}\left(\frac{U^n_{\tau,\eps}-U^{n-1}_{\tau,\eps}}{\tau} \right)
+\partial\Diss{\V}{}\left(\eps
\frac{U^n_{\tau,\eps}-U^{n-1}_{\tau,\eps}}{\tau} \right)+ \frsub
\ene{t_n}{U^n_{\tau,\eps}} \ni 0 \qquad \text{ for all
$1,\ldots,N_\tau$.}
\end{equation}
We denote by $\pwC {U}{\tau,\eps}$ the left-continuous piecewise
constant interpolants, thus taking the value $\Utaue n$ for $ t \in
(t_{n-1},t_n],$ and by $\pwL {U}{\tau,\eps}$ the piecewise affine
interpolant
\begin{equation}
\label{e:pwl} \pwL {U}{\tau,\eps} (t):= \frac{t-t_{n-1}}{\tau}\Utaue
n + \frac{t_n-t}{\tau}\Utaue {n-1}  \quad \text{for $ t \in
  [t_{n-1},t_n],$} \quad n=1,\ldots, N_\tau.
\end{equation}
As in \cite{MRS-dne},   we  also consider the \emph{variational
interpolant} $\pwM U{\tau,\eps}$  of the elements $(\Utaue
n)_{n=1}^{N}$, first introduced by \textsc{E. De Giorgi} in the
frame of  the \emph{Minimizing Movements} approach to
gradient flows
(see~\cite{DeGiorgi-Marino-Tosques80, DeGiorgi93, Ambrosio95,AGS08}). The
functions $\pwM U{\tau,\eps}: [0,T] \to \V$ are defined by
$ \pwM U{\tau,\eps}(0)=u_\eps(0)$
and
\begin{equation}
  \label{interpmin}
      \text{for }
      t=t_{n-1} + r \in (t_{n-1}, t_{n}],
\ \
     \pwM U\taue(t)
      \in \argmin_{U \in \domainenergy}\left\{r
      \Diss{\eps}{}\left(\frac{U-\Utaue{n-1}}{r}\right) +
 \ene{t}{U} \right\},
\end{equation}
choosing the minimizer
in \eqref{interpmin} so that the map
$t\mapsto \pwM U\tau (t)$ is Lebesgue measurable
in $(0,T)$.  Notice that we may assume $
\pwC U{\tau,\eps}(t_n)= 
\pwL \UU{\tau,\eps}(t_n)= \pwM U{\tau,\eps}(t_n)$
  for every  $n =1, \ldots,N_\tau.$
Moreover, with the variational interpolants $\pwM \UU{\taue}$  we
can associate a  measurable function $\pwM {\xi}{\taue}:(0,T) \to
\V^*$ fulfilling the Euler equation for the minimization problem
\eqref{interpmin}, i.e.
\begin{equation}
  \label{interpxi} \pwM {\xi}{\taue}(t) \in -\frsub \ene{t}{\pwM
    \UU\taue(t)} \cap \left( \partial\Diss{\eps}{}\left(\frac{\pwM
        \UU\taue(t) - \Utaue{n-1}
}{t-t_{n-1}} \right) \right) \qquad
\forall\, t \in (t_{n-1}, t_n], \ \ n =1,\ldots, N_\tau,
\end{equation} 
cf. \cite{MRS-dne} for further details. Finally, we also set
$\pwC
{{t}}{\tau}(t):=t_k$
for $t \in ( t_{k-1},t_{k}]$.
Observe that,  for every $t \in [0,T]$ there holds  $\pwC {\mathsf{t}}{\tau}
(t) \down t$
as $\tau \down 0$.  

We recall now a list of important properties
of the discrete solutions, stated in
\cite[Sec.\,6]{MRS-dne}.

\begin{proposition}
  \label{prop:first-a-priori-discrete}
  For every $\eps>0$ and $\tau>0$
  the \emph{discrete energy inequality}
 \begin{equation}
 \label{eq:discr-en-ineq-1}
 \begin{aligned}
\int_{\pwC {\mathsf{t}}{\tau}(s)}^{\pwC {\mathsf{t}}{\tau}(t)}\left(
 \Diss{\eps}{}( \pwL {\dot{\UU}}{\taue}) {+}
 \Diss{\eps}{*}(\pwM {\xi}{\taue})\right) \,\dd r
+ \ene{\pwC {\mathsf{t}}{\tau}(t)}{\pwC{\UU}{\taue}(t)}  \leq
\ene{\pwC {\mathsf{t}}{\tau}(s)}{\pwC{\UU}{\taue}(s)} +
\int_{\pwC{\mathsf{t}}{\tau}(s)}^{\pwC {\mathsf{t}}{\tau}(t)}
\power {r}{\pwM {\UU}{\taue}(r)} \dd r
\end{aligned}
\end{equation}
holds for every
$0\leq s \leq t \leq T$.
  If moreover
  $\Psiz(\Utaue 0)+\ene 0{\Utaue 0}\le E_0$
  for all
  $\tau>0$ and $\eps>0$,  then there exist constants $E,S>0$ such that
  for every $\tau,\eps>0$ we have
  \begin{align}
& \label{aprio1} \sup_{t \in [0,T]} \big(\cg{\pwC{\UU}{\taue}(t)}+
\cg{\pwM{\UU}{\taue}(t)}\big) \leq E\,,
\\
& \label{aprio2}
 \Var{\Diss \B{}}{\pwL
{\UU}{\taue}}{0}{T} \leq
 \int_0^T  \Diss{\eps}{}( \pwL
{\dot{\UU}}{\taue}(s))\, \dd s \leq S, \qquad \int_0^T
\Diss{\eps}{*}(\pwM {\xi}{\taue}(s))\, \dd s \leq S,
\\
 & \label{aprio3}
 \sup_{t \in [0,T]}\left( \Vnorm{\pwL U\taue (t) {-} \pwM
U\taue (t)}+ \Vnorm{ \pwL U\taue (t) {-} \pwC U\taue (t)}
\right) \leq S\: \omega\left(
\frac{\tau}{S\eps}\right),\\
&\nonumber\hspace*{7em}
\text{where } \omega(r):=\sup\big\{ v\in [0,\infty): r \,F(r^{-1}v)\le 1\big\}
\end{align}
satisfies $\lim_{r\downarrow0}\omega(r)=0$, in view of the
superlinearity of $F$.\medskip 
\end{proposition}

\subsubsection*{\bfseries Proof of Theorem \ref{th:2}}
We argue exactly as in the proof of Theorem
\ref{th:1}, observing  that Proposition
\ref{prop:first-a-priori-discrete} enables us to 
apply Proposition \ref{cor:2} with
the choices $u_k:={\rm U}_{\tau_k,\eps_k}$, $\tilde
u_k:=\widetilde{\rm U}_{\tau_k,\eps_k}$ along any sequences
$\tau_k,\eps_k$ satisfying \eqref{eq:103-k}.

Up to the extraction of a suitable subsequence, Proposition
\ref{cor:2} shows that there exist $u \in \BV([0,T];D_E,\Psiz)$
satisfying the local stability condition \eqref{eq:65bis} such that
\begin{gather}
\label{other-interp-converge} \pwC {U}{\tau_k,\eps_k}(t), \ \pwL
{U}{\tau_k,\eps_k}(t), \ \pwM {U}{\tau_k,\eps_k}(t) \to u(t) \quad
\text{in $\V$ for all $t \in [0,T]$,}\\
\sup_{t\in [0,T]} \Big(\|\pwL
{U}{\tau_k,\eps_k}(t)-\pwM {U}{\tau_k,\eps_k}(t)\|+
\|\pwL
{U}{\tau_k,\eps_k}(t)- \pwC {U}{\tau_k,\eps_k}(t)\|\Big)\to0.
\end{gather}
We can also pass to the limit as $k \to \infty$ in the discrete energy
inequality \eqref{eq:discr-en-ineq-1} with $s=0$.  Indeed, we use
convergences \eqref{other-interp-converge}, the lower semicontinuity
of the energy $\cE$, and the $\liminf$ inequality \eqref{eq:56} to
obtain \eqref{eq:84-oneside}.  Thus, by Corollary
\ref{prop:BV-charact} we conclude that $u$ is a $\BV$ solution to the
\ris\ \RIS.

The proof of the further energy convergence \eqref{e:conv2-discr}
follows by the very same lines as in the end of the proof of Theorem
\ref{th:1}, see \eqref{eq:119}--\eqref{eq:73}.  Thus, Theorem
\ref{th:2} is proved.  \qed\medskip

\subsubsection*{\bfseries Proof of Theorem\ \ref{thm-van-param}}
Let $(\parat_\eps,\parau_\eps)_\eps$ be a family of rescaled viscous
solutions as in the statement of Theorem \ref{thm-van-param}.
Exploiting condition \eqref{e:uniform-integrability} as well as the
energy identity \eqref{resc-enid-eps} we can apply Proposition
\ref{le:compactness} in the interval $[0,\sfS]$ (with
$\tilde\sfu_n\equiv \sfu_n$ and $G_n=[0,\sfS]$) and find a vanishing
subsequence $(\eps_n)_n$ and a parameterized curve $(\parat,\parau)$
such that convergences \eqref{e:conv1-para} hold.  The second part of
\eqref{hyp:en3}, the closedness-continuity property \eqref{eq:45}, and
Lemma \ref{le:jointlsc} yield
\begin{equation}
\label{liminf-limsup}
\begin{aligned}
&\liminf_{k \to \infty} \ene {\parat_{\eps_k}(s)}{\parau_{\eps_k}(s)}
  \geq \ene {\parat (s)}{\parau(s)} \quad \text{for all }s \in [0,\mathsf{S}],
\\
&\limsup_{k \to \infty} \int_{s_0}^{s_1} \power r{\sfu_{\eps_k}(r)}
\dot\sft_{\eps_k}(r)\dd r \leq \int_{s_0}^{s_1} \power r{\sfu(r)}\dot \sft(r) \dd r
\quad \text{for all } 0 \leq s_0 <s_1 \leq \sfS.
\end{aligned}
\end{equation}
Combining \eqref{liminf-limsup} with \eqref{eq:46} and \eqref{eq:49},
we pass to the limit as $\eps_k \to 0$ in the energy identity
\eqref{resc-enid-eps} and conclude that $(\parat,\parau)$ fulfill the
energy estimate \eqref{eq:16-param-tech-onesided} with $\mathsf{a}=0$
and $\mathsf{b}=\mathsf{S}$.  Therefore thanks to Corollary
\ref{PROP:charact-param} we have that $(\parat,\parau)$ is a
parameterized solution to the \ris\ $\RIS$.

The enhanced convergences \eqref{e:conv2-para} and
\eqref{e:conv3-para} can be proved with similar arguments as in the
end of the proof of Theorem \ref{th:1}.

In order to show that $(\sft,\sfu)$ satisfies the
$\mathsf{m}$-normalization condition \eqref{e:norma-cond}, we observe
that $\dot\sft_\eps\weakto^* \sft$ and $\frf_{\sft_{\eps_k}} (
\sfu_{\eps_k}, \dot\sfu_{\eps_k})\weakto^*\frf= \mathsf m-\dot\sft$ in
$L^\infty(0,\sfS)$. The liminf estimates \eqref{eq:46} and
\eqref{eq:47} (localized on arbitrary intervals of $[0,\sfS]$) yield
that $\frf\ge \mathfrak h:=\Psiz(\dot\sfu)+
\frG_{}[\sft,\sfu;\dot\sft,\dot\sfu]$ $\Leb 1$-a.e.~in $(0,\sfS)$.
Moreover, $\frf_{\sft_\eps}(\sfu_{\eps},\dot\sfu_{\eps})\le \mathfrak
h_{\eps}:=\Psiz(\dot\sfu_{\eps})+ \frG_{\eps}(\sft_\eps,\sfu_\eps;
\dot\sft_\eps,\dot\sfu_\eps)$ and convergence \eqref{e:conv3-para}
implies
\begin{displaymath}
  \mathfrak h_{\eps_k}\weakto \mathfrak h\quad
  \text{in the sense of distributions of $\scrD'(0,\sfS)$},
\end{displaymath}
so that $\frf\le \mathfrak h$. We conclude that $\frf=\mathfrak h$
and $\dot\sft+\mathfrak h=\mathsf m$, and Theorem \ref{thm-van-param} is proved. 
\qed\medskip

\subsection{Uniform $\BV$-estimates for discrete Minimizing Movements}
\label{ss:5.2}

The aim of this section is to prove Theorem \ref{th:3-discrete},
i.e.~the uniform bound
  \begin{equation}
    \label{eq:75bis}
  \exists\, C>0\ \forall \, \tau>0, \eps >0:  \quad   
 \sum_{n=1}^{N_\tau} \|\mathrm U^n_\taue-\mathrm U^{n-1}_\taue\|\le C
\end{equation}
for all discrete Minimizing Movements, whenever the stronger structural
assumptions \eqref{particular-dissipations}--\eqref{eq:77} hold and
the discrete initial data satisfy \eqref{eq:116}.
We start with an elementary discrete Gronwall-like lemma.

\begin{lemma}[A discrete Gronwall lemma]
\label{discr-Gronw-lemma}
Let $\gamma>0$ and let $ (a_n), (b_n) \subset [0,+\infty) $ be
positive sequences, satisfying
\begin{equation}
\label{discr-Gronw-hyp} (1{+}\gamma)^2\,a_n^2 \leq
a_{n-1}^2 + b_n a_n  \quad \forall n \geq 1.
\end{equation}
Then, for all $k \in \N $ there holds
\begin{equation}
\label{discrete-Gronw-est}
 \sum_{n=1}^k  a_n \leq
 \frac {1}{\gamma}\Big(a_0 +  \sum_{n=1}^k  b_n\Big).
\end{equation}
\end{lemma}
\begin{proof}
  We first show that assumption \eqref{discr-Gronw-hyp} yields
  \begin{equation}
    \label{eq:99}
    (1{+}\gamma) a_n\le a_{n-1}+b_n.
  \end{equation}
  Indeed, \eqref{eq:99} is trivially true if $(1{+}\gamma)a_n\le
  a_{n-1}$.  If $(1{+}\gamma)a_n > a_{n-1}$ we divide both sides in
  \eqref{discr-Gronw-hyp} by $(1{+}\gamma)a_n$ and estimate the
  right-hand side by  $\frac{a_{n-1}^2}{(1{+}\gamma)a_n}
  +\frac{b_n}{1{+}\gamma} < a_{n-1} + b_{n}$.  Summing \eqref{eq:99}
  from $n=1$ to $k$ and setting $S_k:= \sum_{n=1}^k a_n$ we find
  $ (1{+}\gamma) S_k\le a_0+ S_{k-1}+\sum_{n=1}^k b_n$, 
which yields \eqref{discrete-Gronw-est} since $S_{k-1}\le S_k$.
\end{proof}

\subsubsection*{\bfseries Proof of Theorem \ref{th:3-discrete}}

From estimate \eqref{aprio1} it follows that $\Utaue n \in
\domainenergy_E$ for all $n$ and all $\eps,\,\tau>0$. Therefore
\eqref{subdiff-charact}--\eqref{eq:77} (and a fortiori \eqref{eq:89})
hold with constants $\alpha_E,\,\Lambda_E,\, L_E$.

Notice moreover that setting $\Utaue{-1}:=0$, the discrete Euler
equation \eqref{discrete-Euler} is satisfied also for $n=0$.  Let us
set $\Vtaue n:=\tau^{-1}(\Utaue n-\Utaue{n-1})$, $\Xitaue n\in
-\partial\ene{t_n}{\Utaue n}\cap \partial\Psiz_\eps(\Vtaue n)$
according to \eqref{discrete-Euler}.  We subtract
\eqref{discrete-Euler} at $n$ from \eqref{discrete-Euler} at $n+1$,
and take the duality with $\Vtaue {n+1}$, observing that the
generalized convexity condition \eqref{eq:89} yields
\begin{equation}
  \label{eq:131}
  \langle \Xitaue{n+1}-\Xitaue n,\Vtaue{n+1}\rangle\le
  - 2\alpha_E\tau\|\Vtaue{n+1}\|^2+2\tau \Lambda_E \Dnorm{\Vtaue{n+1}}
  \|\Vtaue{n+1}\|+2\tau \|\Vtaue{n+1}\|.
\end{equation}
On the other hand the homogeneity of $\Psiz$ and $\Psiv$ yield
\begin{align*}
  \langle \partial\Psiz(\Vtaue {n+1}),\Vtaue {n+1}\rangle &= \Psiz(\Vtaue{n+1}),\quad
  \langle \partial\Psiz(\Vtaue {n}),\Vtaue {n+1}\rangle \le \Psiz(\Vtaue{n+1}),\\
  \langle \partial\Psiv(\eps\Vtaue {n+1}),\Vtaue {n+1}\rangle &=\eps\|\Vtaue{n+1}\|^2,\quad
  \langle \partial\Psiv(\eps\Vtaue {n}),\Vtaue {n+1}\rangle \le \frac\eps2\|\Vtaue{n+1}\|^2
  +\frac \eps2\|\Vtaue n\|^2.
\end{align*}
and therefore
\begin{equation}
  \label{eq:130}
  \langle \Xitaue{n+1}-\Xitaue n,\Vtaue{n+1}\rangle\ge
  \frac \eps2\|\Vtaue{n+1}\|^2-
  \frac\eps 2\|\Vtaue n\|^2.
\end{equation}
Combining \eqref{eq:131} and  \eqref{eq:130} we get
\[
\Vnorm{\Vtaue {n+1}}^2 + \frac{4\alpha \tau}{\eps}
\Vnorm{\Vtaue {n+1}}^2 \leq \Vnorm{\Vtaue{n}}^2 +
\frac{4\tau}{\eps} \left( L_E +\Lambda_E \Dnorm{\Vtaue n}\right) \Vnorm{\Vtaue n},
\]
Observe that the above inequality can be rewritten in the form of
\eqref{discr-Gronw-hyp} with the choices $a_n= \Vnorm{\Vtaue{n}}$,
$b_n= \frac {4\tau}\eps\big(L_E + \Bnorm{\Vtaue n}\big)$, and
$\gamma:=(1+4\alpha\tau/\eps)^{1/2}-1$. Using $a_0=\|\Vtaue 0\|=0$ and
applying Lemma \ref{discr-Gronw-lemma}  elementary
computations yield
\begin{equation*}
 \sum_{n=1}^{N_\tau-1} \tau \Vnorm{\Vtaue{n}}
\leq   \Big(4Q+\frac2\alpha\Big)\Big(T L_E+E),
\end{equation*}
which is the desired estimate \eqref{eq:75}. 
\qed




\def\cprime{$'$} \def\cprime{$'$}

\end{document}